\documentclass[11pt]{amsart}

\usepackage{amssymb,amsfonts}
\usepackage{hyperref}
\usepackage{amssymb,textcomp}
\usepackage{enumerate}
\usepackage{comment}

\usepackage[utf8]{inputenc}
\usepackage[T1]{fontenc}

\usepackage[mathscr]{eucal}

\usepackage[a4paper, twoside=false, vmargin={2cm,3cm}, includehead]{geometry}

\usepackage{hyperref}
 \hypersetup{
     colorlinks=true,
     linkcolor=red,
     filecolor=blue,
     citecolor = blue,
     urlcolor=cyan,
     }

\newtheorem{lemma}{Lemma}[section]
\newtheorem{theorem}[lemma]{Theorem}
\newtheorem*{theorem*}{Theorem}
\newtheorem{corollary}[lemma]{Corollary}
\newtheorem{question}{Question}
\newtheorem{proposition}[lemma]{Proposition}
\newtheorem*{proposition*}{Proposition}

\newtheorem{conjecture}{Conjecture}

\newtheorem*{problem*}{Problem}

\theoremstyle{definition}
\newtheorem*{claim*}{Claim}

\newtheorem*{definition}{Definition}

\newtheorem{example}{Example}

\newtheorem*{remark}{Remark}

\newcommand{\C}{{\mathbb C}}

\newcommand{\N}{{\mathbb N}}

\newcommand{\R}{{\mathbb R}}

\newcommand{\T}{{\mathbb T}}
\newcommand{\Z}{{\mathbb Z}}

\newcommand{\CA}{{\mathcal A}}
\newcommand{\CB}{{\mathcal B}}
\newcommand{\CC}{{\mathcal C}}
\newcommand{\CD}{{\mathcal D}}
\newcommand{\CE}{{\mathcal E}}

\newcommand{\CI}{{\mathcal I}}

\newcommand{\CX}{{\mathcal X}}
\newcommand{\CY}{{\mathcal Y}}

\newcommand{\CZ}{{\mathcal Z}}

\newcommand{\FD}{{\mathfrak D}}
\newcommand{\FI}{{\mathfrak I}}
\newcommand{\FL}{{\mathfrak L}}

\newcommand{\ba}{{\mathbf{a}}}

\newcommand{\bc}{{\mathbf{c}}}
\newcommand{\be}{{\mathbf{e}}}

 \renewcommand{\b}{{\textbf{b}}}
 \newcommand{\uh}{{\underline{h}}}

 \newcommand{\um}{{\underline{m}}}

\newcommand{\bzero}{{\boldsymbol{0}}}

\newcommand{\eps}{\epsilon}
\newcommand{\veps}{\varepsilon}

\newcommand{\ueps}{{\underline{\epsilon}}}
\newcommand{\un}{{\underline{n}}}

\newcommand{\norm}[1]{\left\Vert #1\right\Vert}
\newcommand{\nnorm}[1]{\lvert\!|\!| #1|\!|\!\rvert}

\newcommand{\inv}{^{-1}}

\DeclareMathOperator*{\E}{\mathbb{E}}

\DeclareMathOperator{\Spec}{Spec}
\DeclareMathOperator{\supp}{Supp}

\newcommand{\abs}[1]{\mathopen{}\left| #1\mathclose{}\right|}

\newcommand{\brac}[1]{\mathopen{}\left( #1 \mathclose{}\right)}

\begin{document}
	

		\title[]{Seminorm control for ergodic averages	with commuting transformations along pairwise dependent polynomials}
		
	\thanks{The authors were supported  by the Hellenic Foundation for Research and Innovation, Proj. No: 1684.}

\author{Nikos Frantzikinakis and Borys Kuca}
\address[Nikos Frantzikinakis]{University of Crete, Department of mathematics and applied mathematics, Voutes University Campus, Heraklion 71003, Greece} \email{frantzikinakis@gmail.com}

\address[Borys Kuca]{University of Crete, Department of mathematics and applied mathematics, Voutes University Campus, Heraklion 71003, Greece}
\email{boryskuca@uoc.gr}



	\maketitle

\begin{abstract}
   We examine multiple ergodic averages of commuting transformations with polynomial iterates in which the polynomials may be pairwise dependent. In particular, we show that such averages are controlled by the  Gowers-Host-Kra seminorms whenever the system satisfies some mild ergodicity assumptions. Combining this result with the general criteria for joint ergodicity established in our earlier work, we determine a necessary and sufficient condition under which such averages are jointly ergodic, in the sense that they converge in the mean to the product of integrals, or weakly jointly ergodic, in that they converge to the product of conditional expectations. As a corollary, we deduce a special case of a conjecture by Donoso, Koutsogiannis, and Sun in a stronger form.
\end{abstract}

			\tableofcontents

\section{Introduction}
\subsection{Main results} An important question in ergodic theory is to examine the limiting behaviour of multiple ergodic averages of the form
\begin{align}\label{E:general ergodic average}
    \frac{1}{N}\sum_{n=1}^N T_1^{p_1(n)}f_1 \cdots T_\ell^{p_\ell(n)}f_\ell.
\end{align}
Here and throughout the paper, we consider a \emph{system} $(X, \CX, \mu, T_1, \ldots, T_\ell)$, i.e. invertible commuting measure-preserving transformations  $T_1, \ldots, T_\ell$ acting on a Lebesgue probability space $(X, \CX, \mu)$, polynomials $p_1, \ldots, p_\ell\in\Z[n]$ that need not be distinct but are always assumed to have zero constant terms, and functions $f_1, \ldots, f_\ell\in L^\infty(\mu)$. The motivation for studying the limiting behaviour of \eqref{E:general ergodic average} comes from the proof of the multidimensional polynomial Szemer\'edi theorem \cite{BL96}, in which the averages \eqref{E:general ergodic average} are the central object of investigation.


It has been proved by Walsh \cite{Wal12} that the averages \eqref{E:general ergodic average} converge in $L^2(\mu)$; however, little is known about the nature of the limit except in several special cases. In this paper, we examine the following question.
\begin{question}\label{Q:joint ergodicity}
When are the polynomials $p_1, \ldots, p_\ell\in\Z[n]$
\begin{enumerate}
    \item \label{i:joint erg} \emph{jointly ergodic} for the system $(X, \CX, \mu, T_1, \ldots, T_\ell)$, in the sense that
\begin{align}\label{E:convergence to product of integrals}
    \lim_{N\to\infty}\norm{\frac{1}{N}\sum_{n=1}^N T_1^{p_1(n)}f_1 \cdots T_\ell^{p_\ell(n)}f_\ell -
    \int f_1\ d\mu \cdots \int f_\ell\ d\mu}_{L^2(\mu)} = 0
\end{align}
for all $f_1, \ldots, f_\ell\in L^\infty(\mu)$?
    \item \label{i:weak joint erg} \emph{weakly jointly ergodic} for the system, in the sense that
\begin{align}\label{E:convergence to expectations}
    \lim_{N\to\infty}\norm{\frac{1}{N}\sum_{n=1}^N T_1^{p_1(n)}f_1 \cdots T_\ell^{p_\ell(n)}f_\ell -
    \E(f_1|\CI(T_1)) \cdots \E(f_\ell|\CI(T_\ell))}_{L^2(\mu)} = 0
\end{align}
for all $f_1, \ldots, f_\ell\in L^\infty(\mu)$?
\end{enumerate}

\end{question}

The first step in deriving the identities \eqref{E:convergence to product of integrals} and \eqref{E:convergence to expectations} is usually to establish control over the $L^2(\mu)$ limit of \eqref{E:general ergodic average} by one of the Gowers-Host-Kra seminorms constructed in \cite{HK05a}, leading to the following question.


\begin{question}\label{Q:factors}
When are the polynomials $p_1, \ldots, p_\ell\in\Z[n]$ \emph{good for seminorm control} for the system $(X, \CX, \mu, T_1, \ldots, T_\ell)$, in the sense that there exists $s\in\N$ such that
\begin{align}\label{E:factors}
    \lim_{N\to\infty}\norm{\frac{1}{N}\sum_{n=1}^N T_1^{p_1(n)}f_1 \cdots T_\ell^{p_\ell(n)}f_\ell}_{L^2(\mu)} = 0
\end{align}
holds for all functions $f_1, \ldots, f_\ell\in L^\infty(\mu)$ whenever $\nnorm{f_j}_{s,T_j} = 0$ for some $j \in\{1, \ldots, \ell\}$?
\end{question}

Question \ref{Q:joint ergodicity}\eqref{i:joint erg} was originally posed by Bergelson and was motivated by
 a result of Berend and Bergelson \cite{BB84} that covered the case of linear polynomials.
It was investigated thoroughly by Donoso, Koutsogiannis, and Sun \cite{DKS19}, as well as in subsequent work of the three authors and Ferr\'e-Moragues \cite{DFMKS21},
  in which they identified a set of sufficient (but not necessary) conditions under which Questions \ref{Q:joint ergodicity}\eqref{i:joint erg} and \ref{Q:factors} can be answered affirmatively for general polynomials. Their conditions turned out to  also be necessary when all the polynomials are equal.

In \cite{FrKu22}, we have addressed Questions \ref{Q:joint ergodicity} and \ref{Q:factors} under the assumption that the polynomials $p_1, \ldots, p_\ell$ are pairwise independent, without any extra assumption on the system. Specifically, we have showed that for every family of pairwise independent polynomials $p_1, \ldots, p_\ell\in\Z[n]$,
there exists $s\in\N$ such that the identity \eqref{E:factors} holds for all systems and all $L^\infty(\mu)$ functions under the stated seminorm assumptions. We then gave a necessary and sufficient spectral condition under which the identities \eqref{E:convergence to product of integrals} and \eqref{E:convergence to expectations} hold.

In this paper, we drop the assumption of pairwise independence. We are thus interested in answering Questions \ref{Q:joint ergodicity} and \ref{Q:factors} for averages \eqref{E:general ergodic average} in which some of the polynomial sequences $p_1, \ldots, p_\ell$ may be pairwise dependent or even identical. An example of this is the average
\begin{align}\label{E:123ex}
    \frac{1}{N}\sum_{n=1}^N T_1^{n^2}f_1 \cdot T_2^{n^2}f_2 \cdot T_3^{n^2+n} f_3.
\end{align}
The pairwise dependence of the polynomials $n^2, n^2, n^2+n$ means that, contrary to the results in \cite{FrKu22}, we cannot establish the seminorm control described in Question~\ref{Q:factors} 
 for all systems. Rather, we need to identify a special property of the system that makes the seminorm control possible.   The needed property turns out to be the following.
\begin{definition}[Good and very good ergodicity property]
	Let $\ell\in\N$ and $p_1,\ldots, p_\ell\in \Z[n]$. We say that the system $(X,\CX,\mu,T_1,\ldots, T_\ell)$ has the \emph{good ergodicity property for the polynomials $p_1,\ldots, p_\ell$}\footnote{Sometimes we also say that  the polynomials $p_1,\ldots, p_\ell$ have  the {\em good ergodicity property for the system} $(X,\CX,\mu,T_1,\ldots, T_\ell)$, or the tuple $(T_j^{p_j(n)})_{j=1,\ldots, \ell}$ has the {\em good ergodicity property}.},  if whenever
	$p_i/c_i=p_j/c_j$ for some $i\neq j$ and nonzero $c_i,c_j\in\Z$ with $\gcd(c_i,c_j)=1$, then
	$
	\CI(T_i^{c_i}T_j^{-c_j}) = \CI(T_i)\cap \CI(T_j)$,\footnote{We assume throughout that all the equalities and inclusions of $\sigma$-algebras hold up to null sets with respect to a given measure on the system.}
	 i.e. a function cannot be invariant under $T_i^{c_i}T_j^{-c_j}$ except in the trivial case when it is simultaneously invariant under $T_i$ and $T_j$. If  $T_i^{c_i}T_j^{-c_j}$ is ergodic for all the aforementioned indices $i,j$ and values $c_i, c_j$, then we say that the system has the \emph{very good ergodicity property for the polynomials $p_1,\ldots, p_\ell$}.
\end{definition}
\begin{remark}
  As we work under the standing assumption that the polynomials have zero constant terms, the equality $p_i/c_i=p_j/c_j$ holds for some nonzero $c_i, c_j\in\Z$ precisely when the polynomials $p_i, p_j$ are linearly dependent.
\end{remark}

For instance, the system $(X, \CX, \mu, T_1, T_2, T_3)$ has the good ergodicity property for the families $n^2, n^2, n^2 + n$ or $2n^2, 2n^2, n^2+n$ if and only if the only functions invariant under $T_1 T_2^{-1}$ are those invariant under $T_1$ and $T_2$, and it has the very good ergodicity property if $T_1 T_2^{-1}$ is ergodic, i.e. only constant functions are invariant under $T_1 T_2^{-1}$.

We first address Question \ref{Q:factors} for systems with the good ergodicity property.
\begin{theorem}[Seminorm control]\label{T:Host Kra characteristic}
    Let $\ell\in\N$ and $p_1, \ldots, p_\ell\in \Z[n]$ be polynomials with the good ergodicity property for a system  $(X,\CX, \mu,T_1, \ldots, T_\ell)$. Then there exists $s\in\N$, depending only on $p_1, \ldots, p_\ell$, such that for all  functions $f_1, \ldots, f_\ell\in L^\infty(\mu)$, we have
\begin{align*}
        \lim_{N\to\infty}\norm{\frac{1}{N}\sum_{n=1}^N T_1^{p_1(n)}f_1 \cdots T_\ell^{p_\ell(n)}f_\ell}_{L^2(\mu)} = 0
\end{align*}
    whenever $\nnorm{f_j}_{s, T_j} = 0$ for some $j\in\{1, \ldots, \ell\}$.
\end{theorem}

Subsequently, we use Theorem \ref{T:Host Kra characteristic} and results from \cite{FrKu22} to address Question \ref{Q:joint ergodicity}. All the concepts appearing in the results below will be defined precisely in Section \ref{S:background}.
\begin{theorem}[Weak joint ergodicity]\label{T: weak joint ergodicity}
    The polynomials $p_1, \ldots, p_\ell\in \Z[n]$ are weakly jointly ergodic for the system $(X,\CX, \mu,T_1, \ldots, T_\ell)$ if and only if the following two conditions hold:
    \begin{enumerate}
        \item the system has the good ergodicity property for the polynomials;
        \item for all nonergodic eigenfunctions $\chi_j\in \CE(T_j)$, $j\in\{1, \ldots, \ell\}$, we have
        \begin{align}\label{E: good for eq nonerg}
            \lim_{N\to\infty}\norm{\frac{1}{N}\sum_{n=1}^N T_1^{p_1(n)}\chi_1 \cdots T_\ell^{p_\ell(n)}\chi_\ell -
            \E(\chi_1|\CI(T_1)) \cdots \E(\chi_\ell|\CI(T_\ell))}_{L^2(\mu)} = 0.
        \end{align}
    \end{enumerate}
\end{theorem}
\begin{remark}
	Using terminology from \cite{Fr21, FrKu22}, Condition~(ii) equivalently states that the polynomials $p_1,\ldots, p_\ell$ are good for equidistribution for the system $(X,\mu,T_1,\ldots, T_\ell)$.
\end{remark}
If additionally the transformations $T_1, \ldots, T_\ell$ are ergodic, we get the following result.
\begin{corollary}[Joint ergodicity]\label{C: joint ergodicity}
    The polynomials  $p_1, \ldots, p_\ell\in \Z[n]$ are jointly ergodic for the system $(X,\CX, \mu,T_1, \ldots, T_\ell)$ if and only if the following two conditions hold:
    \begin{enumerate}
    	 \item all the transformations $T_1, \ldots, T_\ell$ are  ergodic
    	 and the system has the very good ergodicity property for the polynomials;
        \item for eigenvalues $\alpha_j\in \Spec(T_j)$, $j\in\{1, \ldots, \ell\}$, we have
        \begin{align}\label{E: good for eq}
            \lim_{N\to\infty} \frac{1}{N}\sum_{n=1}^N e(\alpha_1 p_1(n) +\cdots+ \alpha_\ell p_\ell(n)) = 0
        \end{align}
        unless $\alpha_1 = \cdots = \alpha_\ell = 0$.
    \end{enumerate}
\end{corollary}
 Theorem \ref{T:Host Kra characteristic} and Corollary \ref{C: joint ergodicity}
extend Theorems~2.8 and 2.14 in \cite{FrKu22} that cover the case of pairwise independent polynomials.

Theorem \ref{T: weak joint ergodicity} and Corollary \ref{C: joint ergodicity}
 can be put in the context of the following conjecture by Donoso, Koutsogiannis, and Sun (the version presented below is a special case of {\cite[Conjecture 1.5]{DKS19}}) that was motivated by previous results of  Berend and Bergelson \cite{BB84}.
In the statement that follows, we  say that a sequence of commuting transformations $(T_n)_{n\in\N}$ on a probability space $(X, \CX, \mu)$ is \emph{ergodic} for $\mu$ if
 \begin{align*}
 	\lim_{N\to\infty}\norm{\frac{1}{N}\sum_{n=1}^N T_n f - \int f d\mu}_{L^2(\mu)}=0
 \end{align*}
 for every $f\in L^\infty(\mu)$.
\begin{conjecture}\label{C:DKS conjecture}
The polynomials $p_1, \ldots, p_\ell\in \Z[n]$ are jointly ergodic for the system $(X,\CX, \mu,T_1, \ldots, T_\ell)$ if only if the following two conditions are satisfied:
\begin{enumerate}
    \item for all distinct $i, j\in\{1, \ldots, \ell\}$, the sequence $(T_i^{p_i(n)} T_j^{-p_j(n)})_{n\in\N}$ is ergodic for $\mu$;
    \item the sequence $(T_1^{p_1(n)}, \ldots, T_\ell^{p_\ell(n)})_{n\in\N}$ is ergodic for $\mu\times \cdots \times \mu$.
\end{enumerate}
\end{conjecture}
Conjecture \ref{C:DKS conjecture} thus lists conditions that have to be checked in order to verify the joint ergodicity of a family of polynomials for a system.

\begin{corollary}\label{C:DKS conjecture holds}
    Conjecture \ref{C:DKS conjecture} holds.
\end{corollary}
In fact, our Theorem \ref{T: weak joint ergodicity} is stronger than Corollary \ref{C:DKS conjecture holds} in a number of ways. First, Theorem \ref{T: weak joint ergodicity} gives a criterion for weak joint ergodicity, not just for joint ergodicity, meaning that the transformations $T_1, \ldots, T_\ell$ need not be ergodic for us to be able to say anything meaningful. Second, our good ergodicity property lists strictly fewer conditions to check in order to verify joint ergodicity than the condition (i) in Conjecture \ref{C:DKS conjecture}. For instance, for the average \eqref{E:123ex}, the condition (i) in Conjecture \ref{C:DKS conjecture} requires us to check the ergodicity of the three sequences $((T_1 T_2\inv)^{n^2})_{n\in\N}$, $(T_1^{n^2} T_3^{-(n^2+n)})_{n\in\N}$, and $(T_2^{n^2}T_3^{-(n^2+n)})_{n\in\N}$. 
By contrast, the good ergodicity property of \eqref{E:123ex} holds if and only if $\CI(T_1 T_2^{-1}) = \CI(T_1)\cap\CI(T_2)$, which is in any way a necessary condition for $((T_1 T_2\inv)^{n^2})_{n\in\N}$ to be ergodic.


 Finally, we remark that the original version of Conjecture \ref{C:DKS conjecture} from \cite{DKS19} is stated for more general tuples
\begin{align}\label{E:multi general average}
	(T_1^{p_{11}(\un)}\cdots T_\ell^{p_{1\ell}(\un)}, \ldots, T_1^{p_{\ell 1}(\un)}\cdots T_\ell^{p_{\ell\ell}(\un)}),
\end{align}
a simple example of which would be the tuple
\begin{align*}
	(T_1^{n^2}T_2^{n^2+n}, T_3^{n^2}T_4^{n^2+n}).
\end{align*}
It is possible that an extension of our method would establish an analogue of Theorem \ref{T:Host Kra characteristic} for such averages. However, besides the fact that new problems arise, the technical complexity of some of our arguments in this paper is already formidable, and it would likely grow significantly if we wanted to tackle the more complicated averages \eqref{E:multi general average}. We have therefore refrained from seeking an extension of Theorem \ref{T:Host Kra characteristic} to averages of tuples as in  \eqref{E:multi general average}, sticking instead to the simpler and  arguably more natural averages \eqref{E:general ergodic average}.

\subsection{Extensions to other averaging schemes}
 Our arguments can be modified  to cover multivariate polynomials and averages over arbitrary F\o lner sequences\footnote{A sequence $(I_N)_{N\in\N}$ of finite subsets of $\Z^D$  is called {\em F\o lner}, if $\lim_{N\to\infty}\frac{|(I_N+\uh)\triangle I_N|}{|I_N|}=0$ for every $\uh\in\Z^D$.}.  While these modifications do not require any new ideas, they force us to introduce even more complicated notation and deal with straightforward but tedious technicalities. For this reason, we omit their proofs. We start with a generalisation of Theorem \ref{T:Host Kra characteristic}.
\begin{theorem}\label{T:Host Kra characteristic Folner}
    Let $D, \ell\in\N$ be integers, $(I_N)_{N\in\N}$ be a F\o lner sequence on $\Z^D$, and  $p_1, \ldots, p_\ell\in \Z[\un]$ have  the good ergodicity property for a system  $(X,\CX, \mu,T_1, \ldots, T_\ell)$. Then there exists $s\in\N$, depending only on $p_1, \ldots, p_\ell$, such that for all $1$-bounded functions $f_1, \ldots, f_\ell\in L^\infty(\mu)$, we have
\begin{align*}
        \lim_{N\to\infty}\norm{\frac{1}{|I_N|}\sum_{\un\in I_N} T_1^{p_1(\un)}f_1 \cdots T_\ell^{p_\ell(\un)}f_\ell}_{L^2(\mu)} = 0
\end{align*}
    whenever $\nnorm{f_j}_{s, T_j} = 0$ for some $j\in\{1, \ldots, \ell\}$.
\end{theorem}

Theorem \ref{T:Host Kra characteristic Folner} and  \cite[Theorem~2.7]{FrKu22} give the following generalisation of Theorem \ref{T: weak joint ergodicity}.
\begin{theorem}\label{T: weak joint ergodicity Folner}
    Let $D, \ell\in\N$ be integers and  $(I_N)_{N\in\N}$ be a F\o lner sequence on $\Z^D$.
     The polynomials $p_1, \ldots, p_\ell\in \Z[n]$ are weakly jointly ergodic for the system $(X,\CX, \mu,T_1, \ldots, T_\ell)$ along $(I_N)_{N\in\N}$, in the sense that
    \begin{align*}
    \lim_{N\to\infty}\norm{\frac{1}{|I_N|}\sum_{\un\in I_N} T_1^{p_1(\un)}f_1 \cdots T_\ell^{p_\ell(\un)}f_\ell -
    \E(f_1|\CI(T_1)) \cdots \E(f_\ell|\CI(T_\ell))}_{L^2(\mu)} = 0
    \end{align*}
    for all $f_1, \ldots, f_\ell\in L^\infty(\mu)$, if and only if the following two conditions hold:
    \begin{enumerate}
        \item the system has the good ergodicity property for the polynomials;
        \item for all nonergodic eigenfunctions $\chi_j\in \CE(T_j)$, $j\in\{1, \ldots, \ell\}$, we have
        \begin{align*}
            \lim_{N\to\infty}\norm{\frac{1}{|I_N|}\sum_{\un\in I_N} T_1^{p_1(\un)}\chi_1 \cdots T_\ell^{p_\ell(\un)}\chi_\ell -
            \E(\chi_1|\CI(T_1)) \cdots \E(\chi_\ell|\CI(T_\ell))}_{L^2(\mu)} = 0.
        \end{align*}
    \end{enumerate}
\end{theorem}

\subsection{Outline of the article}
We begin by recalling in Section \ref{S:background} basic notions and results from ergodic theory, especially those related to the families of Gowers-Host-Kra and box seminorms, dual functions, as well as nonergodic eigenfunctions. Next, we state in Section \ref{S:lemmas} preliminary technical lemmas that are used to prove our main results, most of which are variations of results from \cite{DFMKS21,Fr21, FrKu22}. Having stated all preliminary definitions and lemmas, we discuss at length in Section \ref{S:n^2, n^2, n^2+n} two baby cases of Theorem \ref{T:Host Kra characteristic} that illustrate some of our techniques and point out the necessity of the good ergodicity property. We then proceed in Section \ref{S: longer families} to discuss the formalism and the general strategy for handling longer families. In Section \ref{S:maneuvers}, we give more details of various maneuvers outlined in Section \ref{S: longer families}. These moves take the form of several highly technical propositions that play a crucial part in the inductive proof of Theorem \ref{T:Host Kra characteristic}. By presenting relevant examples, we also show various obstructions that need to be overcome to prove Theorem \ref{T:Host Kra characteristic} in full generality.  Section \ref{S:smoothing} is entirely devoted to the proof of Theorem \ref{T:Host Kra characteristic}: it contains an intricate induction scheme used for the proof and proofs of various intermediate results that together amount to Theorem \ref{T:Host Kra characteristic}. Lastly, in Section \ref{S: other proofs}, we derive Theorem \ref{T: weak joint ergodicity} and Corollaries \ref{C: joint ergodicity} and \ref{C:DKS conjecture holds}.

Some of the techniques used  in this paper were inspired by our earlier work in \cite{FrKu22} where we dealt with pairwise independent polynomials $p_1, \ldots, p_\ell$.  The lack of pairwise independence introduces serious additional complications. Consequently, we are forced to keep track of more information about the averages \eqref{E:general ergodic average} than in \cite{FrKu22}, particularly concerning the properties of the functions present therein and the coefficients of the polynomial iterates. The methods developed in this paper therefore differ in a number of places from the techniques employed in \cite{FrKu22}, and the argument from \cite{FrKu22} is most emphatically not a special case of the argument presented in the current paper.

The need to have a better grip on the averages necessitates more extensive formalism than one in \cite{FrKu22}, making our argument rather hard to digest on a first reading. To compensate for this, we have included numerous examples that illustrate the main new obstacles and ideas in the proofs. The reader is invited to first go over these examples  before delving into the details of the proofs.


\section{Ergodic background and definitions}\label{S:background}
In this section, we present various notions from ergodic theory together with some basic results.
\subsection{Basic notation}\label{SS:notation}
We start with explaining basic notation used throughout the paper.

The letters $\C, \R, \Z, \N, \N_0$ stand for the set of complex numbers, real numbers, integers, positive integers, and nonnegative integers.    With $\T$, we denote the one dimensional torus, and we often identify it with $\R/\Z$ or  with $[0,1)$. We let $[N]:=\{1, \ldots, N\}$ for any $N\in\N$.  With $\Z[n]$, we denote the collection of polynomials with integer coefficients.


For an element $t\in \R$, we let $e(t):=e^{2\pi i t}$.

  If $a\colon \N^s\to \C$ is a  bounded sequence for some $s\in \N$ and $A$ is a non-empty finite subset of $\N^s$,  we let
  $$
  \E_{n\in A}\,a(n):=\frac{1}{|A|}\sum_{n\in A}\, a(n).
  $$

  We commonly use the letter $\ell$ to denote the number of transformations in our system or the number of functions in an average while the letter $s$ usually stands for the degree of ergodic seminorms. We normally write  tuples of length $\ell$ in bold, e.g. $\b\in\Z^\ell$, and we underline tuples of length $s$ (or $s+1$, or $s-1$) that are typically used for averaging, e.g. $\uh\in\Z^s$. For a vector $\b=(b_1,\ldots, b_\ell)\in \Z^\ell$ and a system $(X, \CX, \mu, T_1, \ldots, T_\ell)$, we let
 $$
 T^{\b}:=T_1^{b_1}\cdots T_\ell^{b_\ell},
 $$
 and we denote the $\sigma$-algebra of $T^\b$ invariant functions by $\CI(T^\b)$.
For $j\in[\ell]$, we set $\be_j$ to be the unit vector in $\Z^\ell$ in the $j$-th direction, and we let $\be_0 = \mathbf{0}$, so that $T^{\be_j} = T_j$ for $j\in[\ell]$ and $T^{\be_0}$ is the identity transformation.

We often write $\ueps\in\{0,1\}^s$ for a vector of 0s and 1s of length $s$. For $\ueps\in\{0,1\}^s$ and $\uh, \uh'\in\Z^s$, we set
\begin{itemize}
    \item $\ueps\cdot \uh:=\eps_1 h_1+\cdots+ \eps_s h_s$;
    \item $\abs{\uh} := |h_1|+\cdots+|h_s|$;
    \item $\uh^\ueps := (h_1^{\eps_1}, \ldots, h_s^{\eps_s})$, where $h_j^0:=h_j$ and $h_j^1:=h_j'$ for $j=1,\ldots, s$;
\end{itemize}

We let $\CC z := \overline{z}$ be the complex conjugate of $z\in \C$.

For a tuple $\eta\in \N_0^\ell$ and $I\subset[\ell]$, we define the restriction $\eta|_I := (\eta_i)_{i\in I}$.

\subsection{Ergodic seminorms}\label{SS:seminorms}
We review some basic facts about two families of ergodic seminorms: the Gowers-Host-Kra  seminorms and the box seminorms.
\subsubsection{Gowers-Host-Kra seminorms}
Given a system $(X, \CX, \mu,T)$, we will use the family of ergodic seminorms $\nnorm{\cdot}_{s, T}$,  also known as \emph{Gowers-Host-Kra seminorms}, which were originally introduced in  \cite{HK05a} for ergodic systems. A detailed exposition of their basic properties can be found in  \cite[Chapter~8]{HK18}.
These seminorms are inductively defined for  $f\in L^\infty(\mu)$ as follows (for convenience, we also define $\nnorm{\cdot}_0$, which is   not a seminorm):
$$
\nnorm{f}_{0,T}:=\int f\, d\mu,
$$
and for $s\in \N_0$, we let
\begin{equation*}
	\nnorm{f}_{s+1,T}^{2^{s+1}}:=\lim_{H\to\infty}\E_{h\in [H]} \nnorm{\Delta_{T; h}f}_{s,T}^{2^{s}},
\end{equation*}
where
$$
\Delta_{T;h}f:=f\cdot T^h\overline{f}, \quad h\in \Z,
$$
is the \emph{multiplicative derivative of $f$}  with respect to $T$.
The limit can be shown to exist by successive applications of the mean ergodic theorem, and for $f\in L^\infty(\mu)$ and $s\in \N_0$, we have $\nnorm{f}_{s,T}\leq \nnorm{f}_{s+1,T}$ (see \cite{HK05a} or  \cite[Chapter~8]{HK18}).
It follows immediately from the definition that
$$
\nnorm{f}_{1,T}=\norm{\E(f|\CI(T))}_{L^2(\mu)},
$$
where $\CI(T):=\{f\in L^2(\mu)\colon Tf=f\}$. We also have
\begin{equation}\label{E:seminorm2}
	\nnorm{f}_{s,T}^{2^s}=\lim_{H_1\to\infty}\cdots \lim_{H_s\to\infty}\E_{h_1\in [H_1]}\cdots \E_{h_s\in [H_s]} \int \Delta_{s, T; \uh}f\, d\mu,
\end{equation}
where  for $\uh=(h_1,\ldots, h_s)\in \Z^s$, we let
$$
\Delta_{s,T;\uh}f:=\Delta_{T;h_1}\cdots \Delta_{T;h_s}f=\prod_{\ueps\in \{0,1\}^s}\mathcal{C}^{|\ueps|}T^{\ueps\cdot \uh}f
$$
be the \emph{multiplicative derivative of $f$ of degree $s$} with respect to $T$.

It can be shown that we can take any $s'\leq s$  of the  iterative limits to be simultaneous limits (i.e. average over $[H]^{s'}$ and let $H\to\infty$)
without changing the value of the limit in \eqref{E:seminorm2}. This was originally proved in \cite{HK05a} using  the main structural result of \cite{HK05a}; a more ``elementary'' proof  can be deduced from \cite[Lemma~1.12]{BL15} once the convergence of the uniform Ces\`aro averages is known (and yet another proof can be found in \cite[Lemma~1]{Ho09}).
For $s':=s$, this gives the identity
\begin{equation}\label{E:seminorm}
\nnorm{f}_{s,T}^{2^s}=\lim_{H\to\infty}\E_{\uh\in [H]^s} \int \Delta_{s, T; \uh}f\, d\mu.
\end{equation}
 Moreover,  for $1\leq s'\leq s$, we have
\begin{equation}\label{E:seminorm4}
	\nnorm{f}_{s,T}^{2^{s}}=\lim_{H\to\infty}\E_{\uh\in [H]^{s-s'}} \nnorm{\Delta_{s-s',T;\uh}f}_{s'}^{2^{s'}}.
\end{equation}


It has been established in \cite{HK05a} for ergodic systems and in \cite[Chapter~8, Theorem~14]{HK18} for general systems that the seminorms are intimately connected with a certain family of factors of the system. Specifically, for every $s\in\N$ there exists a factor $\CZ_s(T)\subseteq\CX$, known as the \emph{Host-Kra factor} of \emph{degree} $s$, with the property that
\begin{equation*}
\nnorm{f}_{s, T} = 0  \text{ if and only if } f \text{ is orthogonal to } \CZ_{s-1}(T).
\end{equation*}
 Equivalently, $\nnorm{\cdot}_{s, T}$ defines a norm on the space $L^2(\CZ_{s-1}(T))$ (for a proof see  \cite[Theorem~15, Chapter~9]{HK18}).

\subsubsection{Box seminorms} More generally,  we use analogues of \eqref{E:seminorm} defined with regards to several commuting transformations. These seminorms originally appeared in the work of Host \cite{Ho09}; their finitary versions are often called \emph{box seminorms}, and we sometimes employ this terminology. Let $(X, \CX, \mu, T_1, \ldots, T_\ell)$ be a system.
 For each $f\in L^{\infty}(\mu)$, $h\in \Z$, and $\b \in \Z^\ell$, we define
 $$
 \Delta_{\b; h} f := f \cdot T^{\b h} \overline{f}
 $$
   and for $\uh\in\Z^s$ and $\b_1,\ldots, \b_s\in \Z^\ell$, we let
$$
\Delta_{{\b_1}, \ldots, {\b_{s}}; \uh} f  := \Delta_{{\b_1; h_1}}\cdots\Delta_{{\b_s; h_s}} f = \prod_{\ueps\in\{0,1\}^s} \CC^{|\ueps|} T^{\b_1 \eps_1 h_1 + \cdots + \b_s \eps_s h_s}f.
$$
 We let
$$
\nnorm{f}_{\emptyset}:=\int f\, d\mu
$$
and
\begin{align}\label{ergodic identity}
	\nnorm{f}_{{\b_1}, \ldots, {\b_{s+1}}}^{2^{s+1}}:=\lim_{H\to\infty}\E_{h\in[H]}\nnorm{\Delta_{{\b_{s+1}}; h}f}_{{\b_1}, \ldots, {\b_{s}}}^{2^s}.
\end{align}
In particular, if $\b_1 = \cdots = \b_s:=\b$, then $ \Delta_{\b_1, \ldots, \b_s; \uh}=\Delta_{s, T^\b; \uh}$ and $  \nnorm{\cdot}_{\b_1, \ldots, \b_s}=\nnorm{\cdot}_{s, T^\b}$.
 We remark that these seminorms were defined in a slightly different way in \cite{Ho09}
 and the above identities were established in  \cite[Section~2.3]{Ho09}.

 Iterating \eqref{ergodic identity}, we get the identity
 \begin{equation}\label{E:seminormb1}
 	\nnorm{f}_{{\b_1}, \ldots, {\b_{s}}}^{2^{s+1}}=\lim_{H_1\to\infty}\cdots \lim_{H_s\to\infty}\E_{h_1\in [H_1]}\cdots \E_{h_s\in [H_s]} \int\Delta_{{\b_1; h_1}}\cdots\Delta_{{\b_s; h_s}} f\, d\mu,
 \end{equation}
which extends \eqref{E:seminorm2}. In a complete analogy with the remarks made for the Gowers-Host-Kra seminorms, we have the following: using \cite[Lemma~1]{Ho09}\footnote{Which implies  the convergence of the uniform Ces\`aro averages over $\uh\in \Z^s$ of  $\int \Delta_{{\b_1}, \ldots, {\b_{s}}; \uh} f\, d\mu$.} and \cite[Lemma~1.12]{BL15}, we get
 that we can take any $s'\leq s$  of the  iterative limits to be simultaneous limits (i.e. average over $[H]^{s'}$ and let $H\to\infty$)
without changing the value of the limit in \eqref{E:seminormb1}.
Taking $s'=s$  gives  the identity
$$
\nnorm{f}_{{\b_1}, \ldots, {\b_{s}}}^{2^s}=\lim_{H\to\infty}\E_{\uh\in [H]^s}\,\int \Delta_{{\b_1}, \ldots, {\b_{s}}; \uh} f\, d\mu.
$$
More generally, for any $1\leq s'\leq s$ and $f\in L^\infty(\mu)$, we get the identity
\begin{align}\label{inductive formula}
	\nnorm{f}_{\b_1, \ldots, \b_s}^{2^s} = \lim_{H\to\infty}\E_{\uh\in[H]^{s-s'}}\nnorm{\Delta_{\b_{s'+1}, \ldots, \b_s; \uh}f}_{\b_1, \ldots, \b_{s'}}^{2^{s'}},
\end{align}
which generalises \eqref{E:seminorm4}.

As an example of a box seminorm that  is not a Gowers-Host-Kra seminorm, consider $s=2$ and the vectors $\be_1=(1,0)$, $\be_2=(0,1)$, in which case
$$
\nnorm{f}_{{\be_1},{\be_2}}^4 = \lim_{H\to\infty}\E_{h_1, h_2\in [H]^2}\int f\cdot T_1^{h_1}\overline{f}\cdot T_2^{h_2}\overline{f}\cdot T_1^{h_1}T_2^{h_2}f\, d\mu.
$$
More generally, for  $s=2$ and $\ba=(a_1,a_2)$, $\b=(b_1,b_2)$, we have
$$
\nnorm{f}_{{\ba},{\b}}^4=\lim_{H\to\infty}\E_{h_1, h_2\in [H]^2}\int f\cdot T_1^{a_1h_1}T_2^{a_2h_1}\overline{f}\cdot T_1^{b_1h_2}T_2^{b_2h_2}\overline{f}\cdot T_1^{a_1h_1+b_1h_2}T_2^{a_2h_1+b_2h_2}f\, d\mu.
$$
If the vector $\ba$ repeats $s$ times, we abbreviate it as $\ba^{\times s}$, e.g.
\begin{align*}
    \nnorm{f}_{\ba^{\times 2}, \b^{\times 3}, \bc}=\nnorm{f}_{\ba, \ba, \b, \b, \b, \bc}.
\end{align*}

Box seminorms satisfy the following Gowers-Cauchy-Schwarz inequality  \cite[Proposition~2]{Ho09}
\begin{align}\label{E:GCS}
    \limsup_{H\to\infty}\abs{\E_{\uh\in[H]^s}\int \prod_{\ueps\in\{0,1\}^s} \CC^{|\ueps|} T^{\b_1 \eps_1 h_1 + \cdots + \b_s \eps_s h_s}f_\ueps\, d\mu}\leq \prod_{\ueps\in\{0,1\}^s}\nnorm{f_\ueps}_{\b_1, \ldots, \b_s}.
\end{align}
(One can replace the limsup with a limit since it is known to exist.)

We frequently bound one seminorm in terms of another. An inductive application of formula \eqref{ergodic identity}, or alternatively a simple application of the Gowers-Cauchy-Schwarz inequality \eqref{E:GCS}, yield the following monotonicity property:
\begin{align*}
    \nnorm{f}_{\b_1, \ldots, \b_s}\leq \nnorm{f}_{\b_1, \ldots,\b_s,  \b_{s+1}},
\end{align*}
a special case of which is the aforementioned bound $\nnorm{f}_{s,T}\leq \nnorm{f}_{s+1, T}$ for any $f\in L^\infty(\mu)$ and system $(X, \CX, \mu, T)$.

In many of our arguments, we have to deal simultaneously both with a collection of transformations and their powers. The relevant box seminorms are compared in the following lemma.
\begin{lemma}[{\cite[Lemma 3.1]{FrKu22}}]\label{L:seminorm of power}
    Let $\ell,s\in\N$, $(X, \CX, \mu, T_1, \ldots, T_\ell)$ be a system, $f\in L^\infty(\mu)$ be a function, $\b_1, \ldots, \b_s\in\Z^\ell$ be vectors, and $r_1, \ldots, r_s\in\Z$ be nonzero. Then
    \begin{align*}
       \nnorm{f}_{\b_1, \ldots, \b_s}\leq \nnorm{f}_{r_1 \b_1, \ldots, r_s \b_s},
    \end{align*}
    and if $s\geq 2$, we additionally get the bound
    \begin{align*}
        \nnorm{f}_{r_1 \b_1, \ldots, r_s \b_s} \leq (r_1\cdots r_s)^{1/2^{s}}\nnorm{f}_{\b_1, \ldots, \b_s}.
    \end{align*}
\end{lemma}

The following lemma allows us to compare box seminorms depending on the invariant $\sigma$-algebras of the transformations involved.
\begin{lemma}\label{bounding seminorms}
Let $\ell,s\in\N$, $(X, \CX, \mu, T_1, \ldots, T_\ell)$ be a system of commuting transformations, and $\b_1, \ldots, \b_s, \bc_1, \ldots, \bc_s\in\Z^\ell$ be vectors with the property that $\CI(T^{\b_i})\subseteq \CI(T^{\bc_i})$
for each $i\in[s]$. Then $\nnorm{f}_{\b_1, \ldots, \b_s}\leq \nnorm{f}_{\bc_1, \ldots, \bc_s}$ for each $f\in L^\infty(\mu)$.
\end{lemma}
\begin{proof}
We prove this by induction on $s$. For $s=1$, we simply have
\begin{align*}
    \nnorm{f}_{\b_1}=\norm{\E(f|\CI(T^{\b_1}))}_{L^2(\mu)}\leq \norm{\E(f|\CI(T^{\bc_1}))}_{L^2(\mu)} = \nnorm{f}_{\bc_1},
\end{align*}
where we use the fact that $\norm{\E(f|\CA)}_{L^2(\mu)}\leq\norm{\E(f|\CB)}_{L^2(\mu)}$ whenever $\CA\subseteq\CB$.
For $s>1$, we use the induction formula for seminorms and the result for $s=1$ to deduce that
\begin{align*}
    \nnorm{f}_{\b_1, \ldots, \b_s}^{2^s} &=\lim_{H\to\infty}\E_{\uh\in[H]^{s-1}}\nnorm{\Delta_{\b_1, \ldots, \b_{s-1};\uh} f}_{\b_s}^{2^{s-1}}\\
    &\leq \lim_{H\to\infty}\E_{\uh\in[H]^{s-1}}\nnorm{\Delta_{\b_1, \ldots, \b_{s-1};\uh} f}_{\bc_s}^{2^{s-1}} = \nnorm{f}_{\b_1, \ldots, \b_{s-1}, \bc_s}^{2^s}.
\end{align*}
The claim follows by iterating this procedure $s-1$ more times.
\end{proof}

\subsection{Dual functions and sequences}\label{SS:dual}
Let $s\in\N$ and $\{0,1\}^s_* = \{0,1\}^s\setminus\{\underline{0}\}$. For a system $(X, \CX, \mu, T)$ and $f\in L^\infty(\mu),$ we define
\begin{align*}
    \CD_{s, T}(f) := \lim_{M\to\infty}\E_{\um\in [M]^s}\prod_{\ueps\in\{0,1\}^s_*}\CC^{|\ueps|}T^{\ueps\cdot\um}f
\end{align*}
(the limit exists in $L^2(\mu)$ by \cite{HK05a}).
  We call $\CD_{s,T}(f)$ the \emph{dual function} of $f$ of \emph{level} $s$ with respect to $T$.
  The name comes because of the identity
\begin{align}\label{dual identity}
    \nnorm{f}_{s, T}^{2^s} = \int f \cdot \CD_{s, T}(f)\, d\mu,
\end{align}
a consequence of which is that the span of dual functions of degree $s$ is dense in $L^1(\CZ_{s-1}(T))$.

Let $(X, \CX, \mu, T_1, \ldots, T_\ell)$ be a system.
Using  the identities \eqref{inductive formula} and  \eqref{dual identity} we get
\begin{equation*}
\nnorm{f}_{{\b_1}, \ldots, {\b_{s}}, \b_{s+1}^{\times s'}}^{2^{s+s'}}=\lim_{H\to\infty}\E_{\uh\in [H]^s}\int  \Delta_{{\b_1}, \ldots, {\b_{s}}; \uh} f \cdot \CD_{s', T^{\b_{s+1}}}(\Delta_{{\b_1}, \ldots, {\b_{s}}; \uh}f)\, d\mu,
\end{equation*}
the special case of which is
\begin{equation*}
\nnorm{f}_{{\b_1}, \ldots, \b_{s+1}}^{2^{s+1}} = \lim_{H\to\infty}\E_{\uh\in [H]^s} \int  \Delta_{{\b_1}, \ldots, {\b_{s}}; \uh}f \cdot \E(\Delta_{{\b_1}, \ldots, {\b_{s}}; \uh}\overline{f}|\CI(T^{\b_{s+1}}))\, d\mu.
\end{equation*}

For $s\in\N$, we denote
 \begin{equation*}
 \FD_s := \{(T_j^n \CD_{s', T_j}f)_{n\in\Z}\colon\; f\in L^\infty(\mu),\ j\in[\ell],\ 1\leq s'\leq s\}
 \end{equation*}
to be the set of sequences of 1-bounded functions coming from dual functions of degree up to $s$ for the transformations $T_1,\ldots, T_\ell$, and moreover we define $\FD := \bigcup_{s\in\N}\FD_s$.

The utility of dual functions comes from the following approximation result.
\begin{proposition}[Dual decomposition {\cite[Proposition 3.4]{Fr15a}}]\label{dual decomposition}
    Let $(X, \CX, \mu, T)$ be a system, $f\in L^\infty(\mu)$, $s\in\N$, and $\varepsilon>0$. Then we can decompose $f = f_1 + f_2 + f_3$, where
    \begin{enumerate}
        \item (Structured component) $f_1 = \sum_k c_k \CD_{s, T}(g_k)$ is a linear combination of finitely many dual functions of level $s$ with respect to $T$;
        \item (Small component) $\nnorm{f_2}_{L^1(\mu)}\leq \varepsilon$;
        \item (Uniform component) $\nnorm{f_3}_{s, T} = 0$.
    \end{enumerate}
\end{proposition}
Proposition \ref{dual decomposition} will be used as follows. Suppose that the $L^2(\mu)$ limit of the average $\E_{n\in[N]}\, T_1^{p_1(n)}f_1 \cdots T_\ell^{p_\ell(n)}f_\ell$ vanishes whenever $\nnorm{f_\ell}_{s, T_\ell} = 0$. If
\begin{align*}
    \limsup_{N\to\infty}\norm{\E_{n\in[N]}\, T_1^{p_1(n)}f_1 \cdots T_\ell^{p_\ell(n)}f_\ell}_{L^2(\mu)}>0
\end{align*}
for some functions $f_1, \ldots, f_\ell\in L^\infty(\mu)$, then we decompose $f_\ell$ as in Proposition \ref{dual decomposition} for sufficiently small $\varepsilon>0$ so that
\begin{align*}
    \limsup_{N\to\infty}\norm{\E_{n\in[N]}\prod_{j\in[\ell-1]}T_j^{p_j(n)}f_j \cdot \sum_k c_k T_\ell^{p_\ell(n)}\CD_{s, T_\ell}(g_k)}_{L^2(\mu)}>0
\end{align*}
for some (finite) linear combinations of dual functions. Applying the triangle inequality and the pigeonhole principle, we deduce that there exists $k$ for which
\begin{align*}
    \limsup_{N\to\infty}\norm{\E_{n\in[N]}\prod_{j\in[\ell-1]}T_j^{p_j(n)}f_j \cdot \CD(p_\ell(n))}_{L^2(\mu)}>0,
\end{align*}
where $\CD(n)(x) := T_\ell^{n}\CD_{s, T_\ell}g_k(x)$ for $n\in\N$ and $x\in X$. This way, we essentially replace the term $T_\ell^{p_\ell(n)}f_\ell$ in the original average by the  more structured piece $\CD(p_\ell(n))$.

\subsection{Eigenfunctions and criterion for weak joint ergodicity}\label{SS:eigenfunctions}
Following \cite{FH18}, 	we  define the notion of eigenfunctions that appears in the statement of Theorem \ref{T: weak joint ergodicity}.
	\begin{definition}
			Let  $(X, \CX, \mu,T)$ be a system,  $\chi\in L^\infty(\mu)$, and  $\lambda\in L^\infty(\mu)$ be a $T$-invariant function. We say that $\chi\in L^\infty(\mu)$ is a {\em non-ergodic eigenfunction}
			with {\em eigenvalue $\lambda$} if
			\begin{enumerate}
                \item $|\chi(x)|$ has value $0$ or $1$ for $\mu$-a.e. $x\in X$ and $\lambda(x)=0$ whenever $\chi(x)=0$;
                \item $T\chi=\lambda\, \chi$, $\mu$-a.e..
			 \end{enumerate}
		\end{definition}
		We denote the set of nonergodic eigenfunctions with respect to $T$ by $\CE(T)$.
		 For ergodic systems, a non-ergodic eigenfunction is either the  zero function  or  a classical  unit modulus eigenfunction.
		 For general systems, each function $\chi\in \CE(T)$ satisfies $\chi(Tx)={\bf 1}_E(x) \, e(\phi(x))\, \chi(x)$ for some $T$-invariant set $E\in \CX$ and measurable $T$-invariant function $\phi \colon X\to \T$.

\begin{definition}[Weak joint ergodicity]        We say that  a collection of sequences  $a_1, \ldots, a_\ell\colon \N\to\Z$ is {\em weakly jointly ergodic} for the system  $(X, \CX, \mu, T_1, \ldots, T_\ell)$, if  $$
        		\lim_{N\to\infty}\norm{\frac{1}{N}\sum_{n=1}^N T_1^{a_1(n)}f_1 \cdots T_\ell^{a_\ell(n)}f_\ell -
        			\E(f_1|\CI(T_1)) \cdots \E(f_\ell|\CI(T_\ell))}_{L^2(\mu)} = 0
        $$
        for all $f_1, \ldots, f_\ell\in L^\infty(\mu)$.
        \end{definition}
        The notion of non-ergodic eigenfunction is important for us because of the following criterion for weak joint ergodicity from \cite{FrKu22}.
					\begin{theorem}[Criterion for weak joint ergodicity {\cite[Theorem 2.5]{FrKu22}}]\label{T: criteria for joint ergodicity}
				The collection of  sequences  $a_1, \ldots, a_\ell\colon \N\to\Z$ is weakly jointly ergodic for the system  $(X, \CX, \mu, T_1, \ldots, T_\ell)$,     if and only if the following two properties hold:
					\begin{enumerate}
						\item \label{i:11} there exists $s\in\N$ such that for every $m\in[\ell]$, we have
						\begin{align*}
						    \lim_{N\to\infty}\norm{\E_{n\in[N]}T_1^{a_1(n)}f_1\cdots T_\ell^{a_\ell(n)}f_\ell}_{L^2(\mu)} = 0
						\end{align*}
						for all $f_1, \ldots, f_\ell\in L^\infty(\mu)$ with $f_j\in\CE(T_j)$ for $j\in\{m+1, \ldots, \ell\}$, whenever $\nnorm{f_m}_{s, T_m} = 0$;
						
						\item \label{i:12} for all non-ergodic eigenfunctions $\chi_j\in\CE(T_j)$, $j\in[\ell]$, we have
						\begin{equation}\label{E:liminv}
						\lim_{N\to\infty}\E_{n\in[N]} \,  T_1^{a_1(n)}\chi_1\cdots T_\ell^{a_\ell(n)}\chi_\ell =\E(\chi_1|\CI(T_1))\cdots \E(\chi_\ell|\CI(T_\ell))
						\end{equation}
						in $L^2(\mu)$.
					\end{enumerate}
				\end{theorem}
				When $T_1, \ldots, T_\ell$ are ergodic, the condition \eqref{E:liminv} can be restated as follows:
					\begin{equation*}
						\lim_{N\to\infty}\E_{n\in[N]} \,  e(\alpha_1 a_1(n) + \cdots + \alpha_\ell a_\ell(n)) = 0
					\end{equation*}
					for all $\alpha_j\in\Spec(T_j)$, $j\in[\ell]$, not all zero. Here,
					$$\Spec(T) := \{\alpha\in\T:\ T\chi = e(\alpha) \chi \textrm{ for some } \chi\in\CE(T)\}.$$
				
				We will apply Theorem \ref{T: criteria for joint ergodicity} in the proof of Theorem \ref{T: weak joint ergodicity}. The first condition will be satisfied thanks to the stronger result proved in Theorem \ref{T:Host Kra characteristic}.

\section{Preliminary results}\label{S:lemmas}
In this section, we gather auxiliary results needed in the proof of Theorem \ref{T:Host Kra characteristic}. We start with the following simple lemma from \cite[Lemma~5.2]{FrKu22}, which allows us to pass from averages of sequences $a_{\uh-\uh'}$ to averages of sequences $a_{\uh}$.
\begin{lemma}\label{difference sequences}
Let $(a_\uh)_{\uh\in\N^s}$ be a sequence of nonnegative real numbers. Then
\begin{align*}
    \E_{\uh, \uh'\in[H]^s}a_{\uh-\uh'} \leq \E_{\uh\in[H]^s}a_{\uh}
\end{align*}
for every $H\in\N$.
\end{lemma}
Subsequently, we state a result that allows us to replace a function $f_m$ in the original average by a more structured averaged term $\tilde{f}_m$ that encodes the information about the original average. This idea originates in the finitary works on the polynomial Szemer\'edi theorem by Peluse and Prendiville \cite{P19, P20, PP19}, and it has been successfully applied in the ergodic theoretic setting in \cite{Fr21,  BF22, FrKu22}. The version below differs from earlier formulations because we additionally show that if the to-be-replaced function $f_m$ is measurable with respect to some sub-$\sigma$-algebra $\CA$, then the same can be assumed about the function that replaces it. In our applications, $\CA$ will always be either the full $\sigma$-algebra $\CX$ or the invariant sub-$\sigma$-algebra of some measure preserving transformation.

\begin{lemma}[Introducing averaged functions]\label{P:dual}
	Let $a_1,\ldots, a_\ell\colon \N\to \Z$  be sequences, $(X,$ $\CX, \mu, T_1, \ldots, T_\ell)$ be a system, and 1-bounded functions $f_1,\ldots, f_\ell\in L^\infty(\mu)$ be such that
	\begin{equation*}
	\limsup_{N\to\infty}\norm{\E_{n\in [N]} \, T_1^{a_1(n)}f_1\cdots T_\ell^{a_\ell(n)}f_\ell}_{L^2(\mu)} \geq \delta
	\end{equation*}
	for some $\delta>0$. Let $m\in [\ell]$. Suppose moreover that $f_m$ is $\CA$-measurable for some sub-$\sigma$-algebra $\CA\subseteq\CX$. Then  there exist $N_k\to\infty$ and $g_k\in L^\infty(\mu)$, with  $\norm{g_k}_{L^\infty(\mu)}\leq 1$,  $k\in\N$, such that for
	\begin{equation*}
	\tilde{f}_m:=\lim_{k\to\infty} \E_{n\in [N_k]}
	 \, T_m^{-a_m(n)}g_k\cdot \prod_{j\in [\ell], j\neq m}T_m^{-a_m(n)} T_j^{a_j(n)}\overline{f}_j,
	\end{equation*}
	where the limit is a weak limit, we have
	\begin{equation*}
	\limsup_{k\to\infty}\norm{\E_{n\in[N_k]} \, T_m^{a_m(n)}\E(\tilde{f}_m|\CA)\cdot  \prod_{j\in [\ell], j\neq m}T_j^{a_j(n)}f_j}_{L^2(\mu)} \geq \delta^4.
	\end{equation*}
\end{lemma}

\begin{proof}
Let $\{\tilde{N}_k\}_{k\in\N}$ be an increasing sequence of integers for which
\begin{align}\label{lower bound in dual}
    \norm{\E_{n\in [\tilde{N}_k]} \, T_1^{a_1(n)}f_1\cdots T_\ell^{a_\ell(n)}f_\ell}_{L^2(\mu)}\geq \delta.
\end{align}
We set
\begin{align*}
    \tilde{g}_N := \E_{n\in [N]} \, T_1^{a_1(n)}{f_1}\cdots T_\ell^{a_\ell(n)}{f_\ell}
\end{align*}
for every $N\in\N$.
 The weak compactness of $L^2(\mu)$ implies  that there exists a subsequence $(N_k)_{k\in\N}$ of $(\tilde{N}_k)_{k\in\N}$ for which the sequence
\begin{align*}
    F_k:=\E_{n\in [N_k]}
	 \, T_m^{-a_m(n)}g_k\cdot \prod_{j\in [\ell], j\neq m}T_m^{-a_m(n)} T_j^{a_j(n)}\overline{f}_j,
\end{align*}
where $g_k := \tilde{g}_{\tilde{N}_k}$, $k\in\N$,  converges weakly to a 1-bounded function $\tilde{f}_m$.

We observe from \eqref{lower bound in dual} that
\begin{align*}
    \delta^2 &\leq \int {g_k} \cdot \E_{n\in [{N}_k]} \, T_1^{a_1(n)}\overline{f}_1\cdots T_\ell^{a_\ell(n)}\overline{f}_\ell\, d\mu\\
    &= \int {f}_m \cdot \E_{n\in [{N}_k]} \, T_m^{-a_m(n)} g_k \cdot  \prod_{j\in [\ell], j\neq m}T_m^{-a_m(n)} T_j^{a_j(n)}\overline{f}_j\, d\mu.
\end{align*}
Taking $k\to\infty$, using the $\CA$-measurability of the 1-bounded function  $f_m$, and applying the Cauchy-Schwarz inequality, we get
\begin{align*}
    \delta^2 \leq \int f_m \cdot \tilde{f}_m\, d\mu = \int f_m \cdot \E(\tilde{f}_m|\CA)\, d\mu \leq \norm{\E(\tilde{f}_m|\CA)}_{L^2(\mu)}.
\end{align*}
Hence,
\begin{align*}
    \delta^4 &\leq \norm{\E(\tilde{f}_m|\CA)}_{L^2(\mu)}^2 = \int \E(\tilde{f}_m|\CA) \cdot \overline{\tilde{f}}_m\, d\mu\\
    &= \lim_{k\to\infty}\int \E(\tilde{f}_m|\CA) \cdot \E_{n\in [N_k]}
	 \, T_m^{-a_m(n)}\overline{g}_k\cdot \prod_{j\in [\ell], j\neq m}T_m^{-a_m(n)} T_j^{a_j(n)}{f}_j\\
	 &= \lim_{k\to\infty}\int  \overline{g}_k \cdot \E_{n\in [N_k]} T_m^{a_m(n)}\E(\tilde{f}_m|\CA)
	 \cdot \prod_{j\in [\ell], j\neq m}T_j^{a_j(n)}{f}_j\, d\mu.
\end{align*}
An application of the Cauchy-Schwarz inequality gives the result.
\end{proof}

We now present two different versions of the dual-difference interchange result that we use in our smoothing argument. While the second version in the proposition below has already been used in \cite{FrKu22}, the first one is novel since the extra information that it provides has not been required in earlier arguments.

\begin{proposition}[Dual-difference interchange]\label{dual-difference interchange}
	Let $(X, \CX, \mu,T_1, \ldots, T_\ell)$ be a system,  $s,$ $s'\in \N$, $\b_1, \ldots, \b_{s+1}\in\Z^\ell$ be vectors,
	$(f_{n,k})_{n,k\in\N}\subseteq L^\infty(\mu)$ be 1-bounded, and $f\in L^\infty(\mu)$ be defined by
	\begin{align*}
        f:=\, \lim_{k\to\infty}\E_{n\in [N_k]}\, f_{n,k},
	\end{align*}
	for some $N_k\to\infty$, where the average is assumed to converge weakly.
	\begin{enumerate}
	    \item If
	\begin{align}\label{seminorm of conditional expectation}
	    \nnorm{\E(f|\CI(T^\bc))}_{\b_1, \ldots, \b_{s+1}}>0,
	\end{align}
	then there exist 1-bounded functions $u_{\uh, \uh'}$, invariant under both $T^{\b_{s+1}}$ and $T^{\bc}$, for which the inequality
    \begin{align*}
        \liminf_{H\to\infty}\E_{\uh, \uh'\in [H]^s}\limsup_{k\to\infty} \E_{n\in[N_k]}
        \int\Delta_{\b_1, \ldots, \b_s; \uh-\uh'}f_{n,k} \cdot u_{\uh, \uh'} \, d\mu >0
    \end{align*}
    holds.
    \item 	If
	\begin{align*}
	    \nnorm{f}_{\b_1, \ldots, \b_s, \b_{s+1}^{\times s'}}>0,
	\end{align*}
	then
	\begin{align*}
            \liminf_{H\to\infty}\E_{\uh, \uh'\in [H]^s}\limsup_{k\to\infty} \E_{n\in[N_k]} \int\Delta_{\b_1, \ldots, \b_s; \uh-\uh'}f_{n,k} \cdot \prod_{j=1}^{2^s}\CD_{j,\uh,\uh'} \, d\mu >0
    \end{align*}
    for some 1-bounded  dual functions $\CD_{j, \uh, \uh'}$ of $T$ of level $s'$.
	\end{enumerate}

\end{proposition}
For the proof of Proposition \ref{dual-difference interchange}, we need the following version of the Gowers-Cauchy-Schwarz inequality from \cite{FrKu22}.
\begin{lemma}[Twisted Gowers-Cauchy-Schwarz inequality]\label{L:GCS} Let $(X, \CX, \mu,T_1, \ldots, T_\ell)$ be a system, $s\in \N$, $\b_1, \ldots, \b_s\in\Z^\ell$ be vectors, and the functions $(f_\ueps)_{\ueps\in\{0,1\}^s}, (u_\uh)_{\uh\in\N^s} \subseteq L^\infty(\mu)$ be 1-bounded. Then for every $H\in \N$, we have
\begin{align*}
&\Big|\E_{\uh\in [H]^s}\,  \int \prod_{\ueps\in \{0,1\}^s} T^{\b_1 \eps_1 h_1+\cdots+\b_s \eps_s h_s} f_\ueps\cdot u_{\uh}\, d\mu\Big|^{2^s} \leq\\
& \E_{\uh,\uh'\in [H]^s}\,  \int \Delta_{\b_1, \ldots, \b_s; \uh-\uh'}f_{\underline{1}} \cdot T^{-(\b_1 h_1'+\cdots+\b_s h'_s)}\left(\prod_{\ueps\in \{0,1\}^s}\mathcal{C}^{|\ueps|}u_{\uh^\ueps}\right)\, d\mu.
\end{align*}
\end{lemma}

We also record two simple observations.
\begin{lemma}\label{composing invariance}
Let $(X, \CX, \mu,T_1, \ldots, T_\ell)$ be a system, $\ba, \b\in\Z^\ell$ be vectors, and $f\in L^\infty(\mu)$ be a function. Then the following two properties hold.
\begin{enumerate}
    \item If $f$ is invariant under $T^\ba T^{-\b}$, then $\nnorm{f}_{s, T^\ba} = \nnorm{f}_{s, T^\b}$ for any $s\in\N$.
    \item If $f$ in invariant under $T^\ba$, then so is the function $\E(f|\CI(T^\b))$.
\end{enumerate}
\end{lemma}
\begin{proof}
For part (i), we notice that
\begin{align*}
    \nnorm{f}_{s, T^\ba} &= \lim_{H\to\infty}\E_{\uh\in[H]^s}\int \prod_{\ueps\in\{0,1\}^s} \CC^{|\ueps|}T^{\ba(\ueps\cdot \uh)} f\, d\mu\\
    &= \lim_{H\to\infty}\E_{\uh\in[H]^s}\int \prod_{\ueps\in\{0,1\}^s} \CC^{|\ueps|}T^{\b(\ueps\cdot \uh)} f\, d\mu = \nnorm{f}_{s, T^\b}.
\end{align*}
For part (ii), we use the fact that $T^\ba, T^{\b}$ commute and the $T^\ba$-invariance of $f$ to observe that $T^\ba T^{h\b}f =  T^{h\b} T^\ba f = T^{h\b} f$. From this and the mean ergodic theorem it follows that
\begin{align*}
    T^{\ba}\E(f|\CI(T^\b)) = \lim_{H\to\infty}\E_{h\in[H]}T^\ba T^{h\b}f = \lim_{H\to\infty}\E_{h\in[H]}T^{h\b} f = \E(f|\CI(T^\b)),
\end{align*}
and so $\E(f|\CI(T^\b))$ is invariant under $T^\ba$.
\end{proof}

\begin{proof}[Proof of Proposition \ref{dual-difference interchange}]
Part (ii) follows from \cite[Proposition 5.7]{FrKu22}, and so we only prove (i). Letting $u_\uh := \overline{\E(\Delta_{{\b_1}, \ldots, {\b_{s}}; \uh}\E(f|\CI(T^\bc))|\CI(T^{\b_{s+1}}))}$, we deduce from \eqref{seminorm of conditional expectation} that
\begin{align*}
    \lim_{H\to\infty}\E_{\uh\in [H]^s} \int  \Delta_{{\b_1}, \ldots, {\b_{s}}; \uh}\E(f|\CI(T^\bc)) \cdot u_\uh\, d\mu > 0.
\end{align*}
The $T^\bc$-invariance of $\E(f|\CI(T^\bc))$ implies the $T^\bc$-invariance of $\Delta_{{\b_1}, \ldots, {\b_{s}}; \uh}\E(f|\CI(T^\bc))$, so the functions $u_\uh$ are invariant under $T^\bc$ by Lemma \ref{composing invariance} (their $T^{\b_{s+1}}$-invariance is trivial). Using the $T^\bc$-invariance of $u_\uh$ and the properties of conditional expectations, we deduce that

\begin{align*}
&\lim_{H\to\infty}\E_{\uh\in [H]^s} \int  \Delta_{{\b_1}, \ldots, {\b_{s}}; \uh}\E(f|\CI(T^\bc)) \cdot u_\uh\, d\mu=\\
& \lim_{H\to\infty}\E_{\uh\in [H]^s} \int  \prod\limits_{{\ueps}\in \{0,1\}^s\setminus\{\underline{1}\}}\mathcal{C}^{|{\ueps}|} T^{\b_1 \eps_1 h_1+\cdots+ \b_s \eps_s h_s}\E(f|\CI(T^\bc))\cdot T^{\b_1 h_1+\cdots+\b_s h_s}f \cdot \E(u_\uh|\CI(T^\bc))\, d\mu.
\end{align*}

For $\ueps\in\{0,1\}^s\setminus\{1\}$, we let $f_\ueps := \CC^{|\ueps|}\E(f|\CI(T^\bc))$.  We deduce from the previous identity and the fact    $f=\lim_{k\to\infty}\E_{n\in [N_k]}\, f_{n,k}$, where convergence is  in the weak sense, that
	$$
    \lim_{H\to\infty} \lim_{k\to\infty}\E_{n\in [N_k]}\, \E_{\uh\in  [H]^s}\,  \int \prod_{\ueps\in \{0,1\}^s\setminus \{\underline{1}\}} T^{\b_1 \eps_1 h_1+\cdots+ \b_s \eps_s h_s} f_\ueps \cdot T^{\b_1 h_1+\cdots+\b_s h_s} f_{n,k}\cdot u_\uh\, d\mu > 0.
	$$
For fixed $k,n, H\in \N$, we apply  Lemma~\ref{L:GCS} with $f_{\underline{1}}:= f_{n,k}$, obtaining
$$
    \liminf_{H\to\infty} \limsup_{k\to\infty}\E_{n\in [N_k]}\, \E_{\uh,\uh'\in  [H]^s}\,  \int\Delta_{\b_1, \ldots, \b_s; \uh-\uh'}f_{n,k} \cdot u_{\uh, \uh'} \, d\mu > 0
$$	
where
\begin{align*}
     u_{\uh, \uh'} := T^{-(\b_1 h_1'+\cdots+\b_s h'_s)}\left(\prod_{\ueps\in \{0,1\}^s}\mathcal{C}^{|\ueps|}u_{\uh^\ueps}\right).
\end{align*}
The functions $u_{\uh, \uh'}$ are both $T^{\b_{s+1}}$ and $T^{\bc}$ invariant given that each $(u_\uh)_{\uh\in\N^s}$ is and the transformations $T_1,\ldots, T_\ell$ commute. The result follows from the fact that the limsup of a sum is at most the sum  of the limsups.
\end{proof}
 The proposition below enables a transition between qualitative and soft quantitative results. Its proof  uses  rather abstract functional analytic arguments and  the mean convergence result of Walsh~\cite{Wal12}. If we instead use the mean convergence result of Zorin-Kranich~\cite{Zo15a} we can also get a variant that deals with averages over an arbitrary F\o lner sequence on $\Z^D$.
\begin{proposition}[Soft quantitative control {\cite[Proposition A.1]{FrKu22}}]\label{P:Us}
	Let $m, \ell, s\in\N$ with $m\in[\ell]$, $p_1, \ldots, p_\ell\in\Z[n]$ be polynomials,  and $(X, \CX, \mu, T_1,\ldots, T_\ell)$ be a system. Let $\CY_1, \ldots, \CY_\ell\subseteq\CX$ be sub-$\sigma$-algebras. Suppose that for all $f_j\in L^\infty(\CY_j, \mu)$, $j\in[\ell]$, the seminorm $\nnorm{f_m}_{s}$ controls   the average
	\begin{equation}\label{E:averages}
		\E_{n\in[N]}  \, T_1^{p_1(n)}f_1\cdots T_\ell^{p_\ell(n)}f_\ell
	\end{equation}
	in that \eqref{E:averages} converges to 0 in $L^2(\mu)$ whenever $\nnorm{f_m}_{s} = 0$.
	Then for every	$\varepsilon>0$	there exists  $\delta>0$ such that if
	$f_j\in L^\infty(\CY_j, \mu)$, $j\in[\ell]$,  are $1$-bounded and  $\nnorm{f_m}_{s}\leq \delta$, then
			$$
			\lim_{N\to\infty} \norm{\E_{n\in[N]} \,  T_1^{p_1(n)}f_1\cdots T_\ell^{p_\ell(n)}f_\ell}_{L^2(\mu)}\leq \varepsilon.
			$$
\end{proposition}

 Finally, we need the following PET result  that gives box seminorm control for averages with extra terms involving dual functions. It extends  \cite[Theorem~2.5]{DFMKS21} that did not involve dual functions. We remark that these arguments  are proved by combining a complicated variant of Bergelson's original PET technique \cite{Be87a}  with  concatenation results of Tao and Ziegler \cite{TZ16}.
\begin{proposition}[Box seminorm control {\cite[Proposition B.1]{FrKu22}}]\label{strong PET bound}
Let $d, \ell, L\in\N$, $\eta\in[\ell]^\ell,$ and $p_1, \ldots, p_\ell, q_1, \ldots, q_L\in\Z[n]$ be polynomials with representation $p_j(n) = \sum_{i=0}^d a_{ji} n^i$. Suppose that $\deg p_\ell = d$ and $d_{\ell j} := \deg(p_\ell \be_{\eta_\ell} - p_j \be_{\eta_j})>0$ for every $j=0, \ldots, \ell-1$. Then there exist $s\in\N$ and nonzero vectors
\begin{align}\label{E:coefficientvectors}
    \b_1, \ldots, \b_s \in \{a_{\ell d_{\ell j}}\be_{\eta_\ell}- a_{j d_{\ell j}}\be_{\eta_j}:\; j = 0, \ldots, \ell-1\},
\end{align}
with the following property: for every system $(X, \CX, \mu, T_1, \ldots, T_\ell)$,  functions $f_1, \ldots, f_\ell\in L^\infty(\mu)$, and sequences of functions $\CD_1, \ldots, \CD_L\in\FD_d$, we have
\begin{align*}
     \lim_{N\to\infty}\norm{\E_{n\in[N]}\prod_{j\in[\ell]}T_{\eta_j}^{p_j(n)}f_j\cdot\prod_{j\in[L]}\CD_{j}(q_j(n))}_{L^2(\mu)}=0
\end{align*}
whenever $\nnorm{f_\ell}_{{\b_1}, \ldots, {\b_s}}=0$.
\end{proposition}

\section{Two motivating examples}\label{S:n^2, n^2, n^2+n}
 In this section, we prove  Theorem \ref{T:Host Kra characteristic} for the family $n^2, n^2, n^2+n$ and sketch the changes needed to handle the family   $n^2, n^2, 2n^2+n$. These two  cases illustrate some (but not all) key ideas needed in the proof of  Theorem \ref{T:Host Kra characteristic} in a simple setting.   Additional complications arise for more general families and the ideas needed to overcome them will be illustrated with  examples given on subsequent sections.

\begin{example}[Seminorm control for a monic family of length 3]\label{Ex: n^2, n^2, n^2+n}
Our goal is to prove the following result.
\begin{proposition}[Seminorm control for $n^2,n^2,n^2+n$]\label{Seminorm control of n^2, n^2, n^2+n}
    There exists $s\in\N$ such that  for every system $(X, \CX, \mu, T_1, T_2, T_3)$ satisfying $\CI(T_1 T_2^{-1})=\CI(T_1)\cap\CI(T_2)$ and all   functions $f_1, f_2, f_3\in L^\infty(\mu)$, the average
    \begin{align}\label{E:n^2, n^2, n^2+n}
    \E_{n\in[N]}T_1^{n^2} f_1 \cdot T_2^{n^2}f_2 \cdot T_3^{n^2+n}f_3
\end{align}
    converges to 0 in $L^2(\mu)$ whenever $\nnorm{f_3}_{s, T_3} = 0$.
\end{proposition}
We subsequently show in Corollary	\ref{Seminorm control of n^2, n^2, n^2+n other terms} that the same conclusion holds if we assume $\nnorm{f_i}_{s, T_i} = 0$ for $i=1,2,$ instead.

By Proposition \ref{strong PET bound}, there exist vectors $\b_1, \ldots, \b_{s+1}\in\{\be_3, \be_3 - \be_1, \be_3 - \be_2\}$ such that
\begin{align}\label{n^2, n^2, n^2+n vanishes}
    \lim_{N\to\infty} \E_{n\in[N]}T_1^{n^2} f_1 \cdot T_2^{n^2}f_2 \cdot T_3^{n^2+n}f_3 = 0
\end{align}
whenever $\nnorm{f_3}_{\b_1, \ldots, \b_{s+1}} = 0$. The goal is to inductively replace all the vectors $\b_1, \ldots, \b_{s+1}$ in the seminorm different from $\be_3$ by $\be_3^{\times s'}$ for some $s'\in\N$, which is achieved in the following proposition.
\begin{proposition}[Box seminorm smoothing]\label{Smoothing of n^2, n^2, n^2+n}
    Let $\b_1, \ldots, \b_{s+1}\in\{\be_3, \be_3 - \be_1, \be_3 - \be_2\}$   and $(X, \CX, \mu, T_1, T_2, T_3)$ be a system satisfying $\CI(T_1 T_2^{-1})=\CI(T_1)\cap\CI(T_2)$. Suppose that \eqref{n^2, n^2, n^2+n vanishes} holds for all  functions $f_1, f_2, f_3\in L^\infty(\mu)$ whenever $\nnorm{f_3}_{\b_1, \ldots, \b_{s+1}} = 0$. Then  there exists $s'\in\N$ independent of the system and the functions, such that \eqref{n^2, n^2, n^2+n vanishes} holds whenever $\nnorm{f_3}_{\b_1, \ldots, \b_s, \be_3^{\times s'}}=0$.
\end{proposition}
Proposition \ref{Seminorm control of n^2, n^2, n^2+n} follows from Proposition \ref{strong PET bound} and an iterative application of Proposition \ref{Smoothing of n^2, n^2, n^2+n}.

Passing from a control by $\nnorm{f_3}_{\b_1, \ldots, \b_{s+1}}$ in Proposition \ref{Smoothing of n^2, n^2, n^2+n} to a control by $\nnorm{f_3}_{\b_1, \ldots, \b_s, \be_\ell^{\times s'}}$ follows a two-step ping-pong strategy similar to the one used in \cite{FrKu22}. Using the control by $\nnorm{f_3}_{\b_1, \ldots, \b_{s+1}}$, we first pass to an auxiliary control by some seminorm $\nnorm{f_i}_{\b_1, \ldots, \b_s, \be_i^{\times s_1}}$ for some $i=1,2$ and $s_1\in\N$, and then we use this auxiliary control to go back and control the average \eqref{original average} by $\nnorm{f_3}_{\b_1, \ldots, \b_s, \be_\ell^{\times s'}}$ for some $s'\in\N$. We call the two steps outlined above \emph{ping} and \emph{pong}.

The assumption $\CI(T_1 T_2^{-1})=\CI(T_1)\cap\CI(T_2)$ is crucial for the following special case that will be invoked in the proof of the general case of Proposition \ref{Seminorm control of n^2, n^2, n^2+n}.
\begin{proposition}\label{Smoothing of n^2, n^2, n^2+n special 1}
    There exists $s\in\N$ such that  for every system $(X, \CX, \mu, T_1, T_2)$ satisfying $\CI(T_1 T_2^{-1})=\CI(T_1)\cap\CI(T_2)$ and all functions $f_1, f_2, f_3\in L^\infty(\mu)$, the average
\begin{align}\label{n^2, n^2, n^2+n special 1}
    \E_{n\in[N]}T_1^{n^2} f_1 \cdot T_2^{n^2}f_2 \cdot T_2^{n^2+n}f_3
\end{align}
converges to 0 in $L^2(\mu)$ whenever one of $\nnorm{f_1}_{s, T_1}$, $\nnorm{f_2}_{s, T_2}$, $\nnorm{f_3}_{s, T_2}$ is 0.
\end{proposition}
\begin{proof}
We prove that $\nnorm{f_3}_{s, T_2} = 0$ implies
\begin{align}\label{n^2, n^2, n^2+n special vanishes}
    \lim_{N\to\infty}\E_{n\in[N]}T_1^{n^2} f_1 \cdot T_2^{n^2}f_2 \cdot T_2^{n^2+n}f_3    = 0
\end{align}
for some $s\in\N$; the other cases follow similarly. By Proposition \ref{strong PET bound}, the equality \eqref{n^2, n^2, n^2+n special vanishes} holds under the assumption that $\nnorm{f_3}_{\be_2^{\times s_1}, (\be_2 - \be_1)^{\times s_2}} = 0$ for some $s_1, s_2\in\N_0$. The assumption $\CI(T_1 T_2^{-1})=\CI(T_1)\cap\CI(T_2)$ implies that $\CI(T_1 T_2\inv)\subseteq\CI(T_2)$. Together with Lemma \ref{bounding seminorms}, this gives $$\nnorm{f_3}_{\be_2^{\times s_1}, (\be_2 - \be_1)^{\times s_2}}\leq \nnorm{f_3}_{\be_2^{\times (s_1+s_2)}} = \nnorm{f_3}_{s_1+s_2, T_2},$$ and so \eqref{n^2, n^2, n^2+n special vanishes} holds whenever $\nnorm{f_3}_{s, T_2} = 0$ for $s=s_1+s_2$.
\end{proof}

Proposition \ref{Smoothing of n^2, n^2, n^2+n special 1} is invoked in the \textit{ping} step of the proof of Proposition \ref{Smoothing of n^2, n^2, n^2+n}; in the \textit{pong} step, we invoke the result below.
\begin{proposition}\label{Smoothing of n^2, n^2, n^2+n special 2}
    Let $d, L\in \N$. Then there exists $s\in\N$ such that for all systems $(X, \CX, \mu, T_1, T_2, T_3)$,  functions $f_1, f_3\in L^\infty(\mu)$, and sequences $\CD_1, \ldots, \CD_L\in\FD_d$, the average
\begin{align}\label{n^2, n^2, n^2+n special 2}
    \E_{n\in[N]}T_1^{n^2} f_1 \cdot \prod_{j=1}^L \CD_{j}(n^2) \cdot T_3^{n^2+n}f_3
\end{align}
converges to 0 in $L^2(\mu)$  whenever $\nnorm{f_3}_{s, T_3} = 0$.
\end{proposition}
Proposition \ref{Smoothing of n^2, n^2, n^2+n special 2}
follows from \cite[Proposition 8.4]{FrKu22} since the two non-dual terms \eqref{n^2, n^2, n^2+n special 2} involve the pairwise independent polynomials $n^2$ and $n^2+n$.

Having stated all the needed auxiliary results, we are finally in the position to prove Proposition \ref{Smoothing of n^2, n^2, n^2+n}.
\begin{proof}[Proof of Proposition \ref{Smoothing of n^2, n^2, n^2+n}]

Suppose that \eqref{n^2, n^2, n^2+n vanishes} fails. Then $\nnorm{f_3}_{\b_1, \ldots, \b_{s+1}}>0$. If $\b_{s+1} = \be_3$, then $\nnorm{f_3}_{\b_1, \ldots, \b_{s}, \be_3}>0$, so we assume without loss of generality that $\b_{s+1} = \be_3 - \be_2$, the case $\b_{s+1} = \be_3 - \be_1$ being identical.

\smallskip

\textbf{Step 1 (\textit{ping}): Obtaining auxiliary control by a seminorm of $f_2$.}

\smallskip

By Proposition \ref{P:dual} (applied with $\CA = \CX$),  we replace $f_3$ by $\tilde{f}_3$ so that
\begin{align*}
    \lim_{N\to\infty}\norm{\E_{n\in[N]}T_1^{n^2} f_1 \cdot T_2^{n^2}f_2 \cdot T_2^{n^2+n}\tilde{f}_3}_{L^2(\mu)}>0.
\end{align*}
We set $f_{j,\uh, \uh'} := \Delta_{{\b_1}, \ldots, {\b_{s}}; \uh-\uh'} f_j$ for $j\in[3]$ and apply Proposition \ref{dual-difference interchange}(i) (with $\bc = \bzero$) to conclude that
$$
\liminf_{H\to\infty}\E_{\uh,\uh'\in [H]^{s}} \lim_{N\to\infty}\norm{ \E_{n\in[N]} \, T_1^{n^2} f_{1, \uh, \uh'}\cdot T_2^{n^2} f_{2, \uh, \uh'}\cdot T_3^{n^2+n}u_{\uh,\uh'}}_{L^2(\mu)}>0
$$
for some 1-bounded functions $u_{\uh, \uh'}$ invariant under $T_3 T_2^{-1}$. The invariance property implies that
\begin{align*}
    T_3^{n^2+n} u_{\uh, \uh'} = T_2^{n^2+n} u_{\uh, \uh'},
\end{align*}
and hence
$$
\liminf_{H\to\infty}\E_{\uh,\uh'\in [H]^{s}} \lim_{N\to\infty}\norm{ \E_{n\in[N]} \, T_1^{n^2} f_{1, \uh, \uh'}\cdot T_2^{n^2} f_{2, \uh, \uh'}\cdot T_2^{n^2+n}u_{\uh,\uh'}}_{L^2(\mu)}>0.
$$
Consequently, there exist a set $B\subseteq\N^{2s}$ of positive lower density and $\veps>0$ such that
\begin{align}\label{E:n^2, n^2, n^2+n bdd from below}
    \lim_{N\to\infty}\norm{ \E_{n\in[N]} \, T_1^{n^2} f_{1, \uh, \uh'}\cdot T_2^{n^2} f_{2, \uh, \uh'}\cdot T_2^{n^2+n}u_{\uh,\uh'}}_{L^2(\mu)}>\veps
\end{align}
for all $(\uh, \uh')\in B$. Each of the averages in \eqref{E:n^2, n^2, n^2+n bdd from below} takes the form \eqref{n^2, n^2, n^2+n special 1}; we therefore apply Propositions \ref{Smoothing of n^2, n^2, n^2+n special 1} and \ref{P:Us} to obtain $s_1\in\N$ and $\delta>0$ such that
\begin{align*}
\nnorm{f_{2, \uh, \uh'}}_{s_1, T_2} > \delta
\end{align*}
for all $(\uh, \uh')\in B$. Hence,
\begin{align}\label{positivity of differences}
    \liminf_{H \to\infty}\E_{\uh,\uh'\in [H]^{s}} \nnorm{\Delta_{{\b_1}, \ldots, {\b_{s}}; \uh-\uh'} f_2}_{s_1, T_2} >0
\end{align}
for some $s_1\in\N$. Together with Lemma \ref{difference sequences}, the inductive formula for seminorms \eqref{inductive formula}  and H\"older inequality, the inequality \eqref{positivity of differences}  implies that
$$
\nnorm{f_2}_{{\b_1}, \ldots, {\b_{s}}, \be_2^{\times s_1}}>0.
$$
We deduce from this that the seminorm $\nnorm{f_2}_{{\b_1}, \ldots, {\b_{s}}, \be_2^{\times s_1}}$ controls the average \eqref{E:n^2, n^2, n^2+n}.

This seminorm control is not particularly useful as an independent result since the vectors $\b_1, \ldots, \b_s$ may not involve the transformation $T_2$ in any way. However, it is of great importance as an intermediate result applied in the next step of our argument.

\smallskip
\textbf{Step 2 (\textit{pong}): Obtaining control by a seminorm of $f_3$.}
\smallskip

Using our assumption that \eqref{n^2, n^2, n^2+n vanishes} fails, we now replace $f_2$ by $\tilde{f}_2$ and deduce from Proposition \ref{P:dual} (applied again with $\CA = \CX$) that
\begin{align*}
    \lim_{N\to\infty}\norm{\E_{n\in[N]}T_1^{n^2} f_1 \cdot T_2^{n^2}\tilde{f}_2 \cdot T_2^{n^2+n}f_3}_{L^2(\mu)}>0.
\end{align*}
From Proposition \ref{dual-difference interchange} (with $\bc = \bzero$) it follows that
$$
\liminf_{H\to\infty}\E_{\uh,\uh'\in [H]^{s}}\lim_{N\to\infty}\norm{ \E_{n\in[N]} \, T_1^{n^2} f_{1, \uh, \uh'}\cdot \CD_{\uh,\uh'}(n^2) \cdot T_3^{n^2+n} f_{3, \uh, \uh'}}_{L^2(\mu)}>0,
$$
where $\CD_{\uh,\uh'}$ is a product of $2^s$ elements of $\FD_{s_1}$. Once again, there exists a set $B'\subseteq\N^{2s}$ of positive lower density and $\veps>0$ such that
\begin{align}\label{E:n^2, n^2, n^2+n bdd from below 2}
    \norm{ \E_{n\in[N]} \, T_1^{n^2} f_{1, \uh, \uh'}\cdot \CD_{\uh,\uh'}(n^2) \cdot T_3^{n^2+n} f_{3, \uh, \uh'}}_{L^2(\mu)} > \veps
\end{align}
for $(\uh, \uh')\in B'$.
Each average indexed in \eqref{E:n^2, n^2, n^2+n bdd from below 2} is of the form
\begin{align}\label{n^2, n^2+n}
    \E_{n\in[N]} \, T_1^{n^2} g_1  \cdot \prod_{j=1}^{2^s}\CD_{j}(n^2)\cdot  T_3^{n^2+n} g_3
\end{align}
for $\CD_1, \ldots, \CD_{2^s}\in\FD_{s_1}$. By Proposition \ref{Smoothing of n^2, n^2, n^2+n special 2}, the averages \eqref{n^2, n^2+n} are controlled by $\nnorm{g_3}_{s', T_3}$ for some $s'\in\N$, and so Proposition \ref{P:Us} gives $\delta>0$ such that
\begin{align*}
    \nnorm{f_{3, \uh, \uh'}}_{s', T_3} > \delta
\end{align*}
for all $(\uh, \uh')\in B'$\footnote{Specifically, we invoke Proposition \ref{P:Us} for averages of the form $\E_{n\in[N]}\,  T_1^{n^2}g_1 \cdot \prod_{j=1}^{2^{s}} T_2^{n^2} g'_j \cdot T_3^{n^2+n}g_3$, where $g'_j$ is $\CZ_{s_1}(T_2)$-measurable for each $j\in[2^{s}]$. One can show that an average like this is qualitatively controlled by $\nnorm{g_3}_{s', T_3}$ by approximating functions $g'_1, \ldots, g'_{2^{s}}$ by linear combinations of dual functions using Proposition \ref{dual decomposition} and then applying Proposition \ref{strong PET bound}.}. Hence,
\begin{align*}
    \liminf_{H \to\infty}\E_{\uh,\uh'\in [H]^{s}} \nnorm{\Delta_{{\b_1}, \ldots, {\b_{s}}; \uh-\uh'} f_3}_{s', T_3} >0.
\end{align*}
Together with Lemma \ref{difference sequences}, the H\"older inequality and the inductive formula \eqref{inductive formula} for the seminorms, this implies that $\nnorm{f_3}_{\b_1, \ldots, \b_s, \be_3^{\times s'}}>0$, which is what we claim.
\end{proof}

Finally, we show how we can use Proposition \ref{Seminorm control of n^2, n^2, n^2+n} to obtain control of the average \eqref{E:n^2, n^2, n^2+n} by seminorms of other terms.
\begin{corollary}\label{Seminorm control of n^2, n^2, n^2+n other terms}
    There exists $s\in\N$ such that for every system $(X, \CX, \mu, T_1, T_2, T_3)$ satisfying $\CI(T_1 T_2^{-1})=\CI(T_1)\cap\CI(T_2)$ and all  functions $f_1, f_2, f_3\in L^\infty(\mu)$, the average \eqref{E:n^2, n^2, n^2+n} converges to 0 in $L^2(\mu)$ whenever one of $\nnorm{f_1}_{s, T_1}, \nnorm{f_2}_{s, T_2}, \nnorm{f_3}_{s, T_3}$ is 0.
\end{corollary}
\begin{proof}
The statement that $\nnorm{f_3}_{s, T_3} = 0$ implies the vanishing of the $L^2(\mu)$ limit of \eqref{E:n^2, n^2, n^2+n} follows from Proposition \ref{strong PET bound} and Proposition \ref{Seminorm control of n^2, n^2, n^2+n}, so the content of Corollary \ref{Seminorm control of n^2, n^2, n^2+n other terms} is to show control by other terms. Suppose that
\begin{align*}
    \lim_{N\to\infty}     \E_{n\in[N]}T_1^{n^2} f_1 \cdot T_2^{n^2}f_2 \cdot T_3^{n^2+n}f_3  \neq 0.
\end{align*}
We then apply Proposition \ref{dual decomposition} and the pigeonhole principle to find a 1-bounded dual function $\CD_{s, T_3}g$ such that
\begin{align*}
    \lim_{N\to\infty}     \E_{n\in[N]}T_1^{n^2} f_1 \cdot T_2^{n^2}f_2 \cdot \CD(n^2+n),
\end{align*}
where $\CD(n) := T_3^n \CD_{s, T_3} g$. By Proposition \ref{strong PET bound}, we have $\nnorm{f_2}_{\be_2^{\times s'}, (\be_2-\be_1)^{\times s'}}>0$ for some $s'\in\N$. The ergodicity assumption of $T_1 T_2^{-1}$ and Lemma \ref{bounding seminorms} imply that $\nnorm{f_2}_{\be_2^{\times 2 s'}}>0$, and an analogous argument gives $\nnorm{f_1}_{\be_1^{\times 2 s'}}>0$.
\end{proof}
\end{example}

The argument in Example~\ref{Ex: n^2, n^2, n^2+n}
 is relatively clean because the leading coefficients of the polynomials are all 1. When this is not the case, minor modifications are required as explained in the next example.
\begin{example}[Seminorm control for a non-monic family of length 3]\label{Ex: n^2, n^2, 2n^2+n}
Consider the average
    \begin{align}\label{E:n^2, n^2, 2n^2+n}
    \E_{n\in[N]}T_1^{n^2} f_1 \cdot T_2^{n^2}f_2 \cdot T_3^{2n^2+n}f_3.
\end{align}
By Proposition \ref{strong PET bound}, this average is controlled by $\nnorm{f_3}_{\b_1, \ldots, \b_{s+1}}$ for some $s\in\N$ and $\b_1, \ldots, \b_{s+1}\in\{2\be_3, 2\be_3 - \be_2\}$, and we want to pass towards the control by $\nnorm{f_3}_{\b_1, \ldots, \b_s, \be_3^{\times s'}}$. Suppose that the $L^2(\mu)$ limit of \eqref{E:n^2, n^2, 2n^2+n} does not vanish, and suppose moreover that $\b_{s+1} = 2\be_3 - \be_2$. Arguing as in the proof of Proposition \ref{Smoothing of n^2, n^2, n^2+n}, we arrive at the inequality
\begin{align*}
    \liminf_{H\to\infty}\E_{\uh,\uh'\in [H]^{s}} \lim_{N\to\infty}\norm{ \E_{n\in[N]} \, T_1^{n^2} f_{1, \uh, \uh'}\cdot T_2^{n^2} f_{2, \uh, \uh'}\cdot T_3^{2n^2+n}u_{\uh,\uh'}}_{L^2(\mu)}>0
\end{align*}
where $f_{j, \uh, \uh'}:=\Delta_{\b_1, \ldots, \b_s; \uh-\uh'}f_j$ for $j\in[2]$ and the functions $u_{\uh, \uh'}$ are $T_2 T_3^{-2}$ invariant. We can no longer apply the invariance property in the same way as before since the polynomial $2n^2+n$ is not divisible by 2. Instead, we first split $\N$ into the odd and even part and then apply the triangle inequality to deduce that
\begin{multline*}
    \E_{r\in\{0,1\}}\limsup_{H\to\infty}\E_{\uh,\uh'\in [H]^{s}}\\
    \lim_{N\to\infty}\norm{ \E_{n\in[N]} \, T_1^{(2n+r)^2} f_{1, \uh, \uh'}\cdot T_2^{(2n+r)^2} f_{2, \uh, \uh'}\cdot T_3^{2(2n+r)^2+(2n+r)}u_{\uh,\uh'}}_{L^2(\mu)}>0.
\end{multline*}
Only then can we apply the $T_2 T_3^{-2}$ invariance of $u_{\uh, \uh'}$ to obtain the identity
\begin{align*}
    T_3^{2(2n+r)^2+(2n+r)}u_{\uh,\uh'} = T_2^{4n^2 + (4r+1)n}T_3^{2r^2 + r}u_{\uh, \uh'}.
\end{align*}
It follows that for some $r_0\in \{0,1\}$, we have
\begin{multline*}
    \limsup_{H\to\infty}\E_{\uh,\uh'\in [H]^{s}} \\ \lim_{N\to\infty}\norm{ \E_{n\in[N]} \, T_1^{4n^2 + 4r_0 n} f'_{1, \uh, \uh'}\cdot T_2^{4n^2 + 4r_0 n} f'_{2, \uh, \uh'}\cdot T_2^{4n^2 + (4r_0+1)n}u'_{\uh,\uh'}}_{L^2(\mu)}>0
\end{multline*}
upon setting $f'_{j, \uh, \uh'} := T_j^{r_0^2} f_{j, \uh, \uh'}$ for $j\in[2]$ and $u'_{\uh, \uh'} := T_3^{2r_0^2 + r_0} u_{\uh, \uh'}$. The rest of the argument proceeds analogously except that we invoke an analogue of Proposition \ref{Smoothing of n^2, n^2, n^2+n special 1} for the tuple $(T_1^{4n^2 + 4r_0 n}, T_2^{4n^2 + 4r_0 n}, T_2^{4n^2 + (4r_0+1)n})$. The important part about this new tuple is that the first two polynomials are again pairwise dependent while the last one is pairwise independent with any of the first two, and that the new tuple retains the good ergodicity property.
\end{example}

\section{Formalism and general strategy for longer families}\label{S: longer families}

We move on towards deriving Theorem \ref{T:Host Kra characteristic} for longer families. To prove it for averages
\begin{align}\label{E: av Sec 6}
    \E_{n\in [N]} \,\prod_{j\in[\ell]}T_{j}^{p_j(n)}f_j,
\end{align}
 we need to analyse more complicated averages of the form
\begin{align}\label{general average}
    \E_{n\in [N]} \,\prod_{j\in[\ell]}T_{\eta_j}^{\rho_j(n)}f_j \cdot \prod_{j\in[L]}\CD_{j}(q_j(n))
\end{align}
that  appear at the intermediate steps of the proof of Theorem \ref{T:Host Kra characteristic}, much like averages \eqref{n^2, n^2, n^2+n special 1} and \eqref{n^2, n^2, n^2+n special 2} show up at the intermediate steps of the proof of Proposition \ref{Seminorm control of n^2, n^2, n^2+n}, a special case of Theorem \ref{T:Host Kra characteristic} for the family $n^2, n^2, n^2+n$. In \eqref{E: av Sec 6} and \eqref{general average}, $p_1, \ldots, p_\ell, \rho_1, \ldots, \rho_\ell,  q_1, \ldots, q_L$ are (not necessarily distinct) polynomials with integer coefficients and  zero constant terms, $(X, \CX, \mu, T_1, \ldots, T_\ell)$ is a system, $f_1, \ldots, f_\ell\in L^\infty(\mu)$ are 1-bounded functions, and $\CD_1, \ldots, \CD_L\in\FD$ are sequences of functions. Since $\CD_{j}(q_j(n))$ has the form $T_{\pi_j}^{q_j(n)}g_j$ for some $\pi_j\in[\ell]$ and $g_j\in L^\infty(\mu)$, the averages \eqref{general average} converge in $L^2(\mu)$ by \cite{Wal12}. The same comment applies for all limits involving dual sequences that appear in the rest of the paper.

The purpose of this section is to introduce a formalism that helps us meaningfully discuss averages \eqref{general average}. This will be done in Section \ref{SS:formalism}. Subsequently, we give in Section \ref{SS: strategy} an overview of the strategy used to prove Theorem \ref{T:Host Kra characteristic}. The details of various moves discussed in Section \ref{SS: strategy} will be presented in Section \ref{S:maneuvers}.

While discussing various examples in this and the next sections, we often say informally that for $j\in[\ell],$ the average \eqref{general average} is \textit{controlled} by a $T_{\eta_j}$-seminorm of $f_j$ (or that we have \textit{seminorm control} of \eqref{general average} by a $T_{\eta_j}$-seminorm of $f_j$) if for all $d, L\in\N$, there exists $s\in\N$ such that if the $T_{\eta_j}$-seminorm of $f_j$ vanishes, then  the $L^2(\mu)$ limit of \eqref{general average} is 0 for all sequences $\CD_1, \ldots, \CD_L\in\FD_d$ and all functions $f_1, \ldots, f_\ell\in L^\infty(\mu)$ satisfying some explicitly stated invariance properties. We also say informally that we have \textit{seminorm control over the average} \eqref{general average} if we have seminorm control by a $T_{\eta_j}$-seminorm of $f_j$ for every $j\in[\ell]$.


\subsection{The formalism behind the induction scheme}\label{SS:formalism}
We start by introducing a handy formalism used for the induction scheme in the proof of Theorem \ref{T:Host Kra characteristic}. We often associate the average \eqref{general average} with the tuple $\brac{T_{\eta_j}^{\rho_j(n)}}_{j\in[\ell]}$. This tuple does not contain any information about the polynomials $q_1, \ldots, q_L$, but this is not necessary. These terms play no role in our inductive argument and can be easily disposed of using Proposition \ref{strong PET bound}. The only thing they do influence is the degree $s$ of the seminorm with which we end up controlling the average \eqref{general average}.

\begin{definition}[Indexing data] For an average \eqref{general average} or the associated tuple $\brac{T_{\eta_j}^{\rho_j(n)}}_{j\in[\ell]}$, we let $\ell$ be their \emph{length}, $d:= \max\limits_{j\in[\ell]}\deg \rho_j$ be its \emph{degree}, and $\eta$ be its \emph{indexing tuple}. For $j\in[\ell]$, we set $d_j := \deg \rho_j$.  We let $K_1$ be the maximum number of pairwise independent polynomials within the family $\rho_1,\ldots, \rho_\ell$ (we set $K_1:=1$ if $\ell=1$ or every two polynomials are pairwise dependent). We partition $[\ell]=\bigcup_{t\in[K_1]} \FI_t$, where $j_1, j_2$ belong to the same $\FI_t$ if and only if $\rho_{j_1}$, $\rho_{j_2}$ are linearly dependent. Thus, $\FI_1, \ldots, \FI_{K_1}$ partitions $[\ell]$ into index sets corresponding to families of pairwise dependent polynomials.
Furthermore, we define
\begin{align*}
    \FL := \{j\in[\ell]:\ \deg \rho_j = d\}
\end{align*}
to be the set of indices corresponding to polynomials of maximum degree, and we rearrange $\FI_1, \ldots, \FI_{K_1}$ so that $\FL =\bigcup_{t\in[K_2]}\FI_t$ for some $K_2\leq K_1$. We also let $K_3 := |\FL|$ be the number of maximum degree polynomials, and we notice that $K_2\leq K_3\leq \ell$.
\end{definition}
Sometimes, we denote $\FL = \FL(\rho_1, \ldots, \rho_\ell), \FI_t=\FI_t(\rho_1, \ldots, \rho_\ell)$, and $K_i = K_i(\rho_1, \ldots, \rho_\ell)$ to emphasise the dependence on a specific family of polynomials.

\begin{example}\label{Ex: 7 tuple}
Consider the tuple
\begin{align}\label{tuple example}
    (T_1^{n^2}, T_2^{n^2}, T_3^{n^2+n}, T_4^{2n^2+2n}, T_3^{n^2+2n}, T_6^n, T_2^{n^2+3n}).
\end{align}
It has length 7, degree 2, indexing tuple $\eta = (1, 2, 3, 4, 3, 6, 2)$, $K_1 = 5$ (corresponding to five pairwise independent polynomials $n^2, n^2+n, n^2 + 2n, n^2 + 3n, n$), $K_2 = 4$ (corresponding to four quadratic pairwise independent polynomials $n^2, n^2 + n, n^2 + 2n, n^2+3n$), $K_3 = 6$ (corresponding to six quadratic polynomials appearing in \eqref{tuple example}), $\FL = \{1, 2, 3, 4, 5, 7\}$, and partition $\FI_1 = \{1, 2\}, \FI_2 = \{3, 4\}, \FI_3 = \{5\}, \FI_4 = \{7\}, \FI_5 = \{6\}$.
\end{example}

The rationale behind introducing the partition $\FI_1, \ldots, \FI_{K_1}$ and the indexing tuple $\eta$ is as follows. As part of our assumptions, we know that  $T_i^{\beta_i} T_j^{-\beta_j}$ is ergodic whenever the polynomials $\rho_i, \rho_j$ are dependent, $b_i, b_j$ are their leading coefficients and $\beta_i := b_i/\gcd(b_i, b_j), \beta_j := b_j/\gcd(b_i, b_j)$. Introducing the partition $\FI_1, \ldots, \FI_{K_1}$ allows us to keep track of pairs $(i,j)$ for which we have these ergodicity properties. The reason for introducing the indexing tuple $\eta$ is that  in our induction procedure, we gradually replace a transformation $T_{\eta_m}$ in the tuple $(T_{\eta_j}^{\rho_j(n)})_{j\in[\ell]}$ by a different transformation $T_{\eta_i}$, and so $\eta$ keeps track of these substitutions.

\begin{definition}[Good ergodicity property along $\eta$] Let $\eta\in[\ell]^\ell$ be an indexing tuple and $b_1, \ldots, b_\ell$ be the leading coefficients of the polynomials $\rho_1, \ldots, \rho_\ell$. We say that  the system  $(X,\CX, \mu,T_1, \ldots, T_\ell)$ has the \emph{good ergodicity property along $\eta$ for the polynomials} $\rho_1, \ldots, \rho_\ell$ if for every distinct $\eta_{j_1}, \eta_{j_2}$ belonging to the same $\FI_t$ with $t\in[K_1]$,  we have
\begin{align*}
    \CI(T_{\eta_{j_1}}^{\beta_{j_1}}T_{\eta_{j_2}}^{-\beta_{j_2}}) = \CI(T_{\eta_{j_1}})\cap \CI(T_{\eta_{j_2}}),
\end{align*}
where $\beta_{j} := b_{j}/\gcd(b_{j_1}, b_{j_2})$ for $j=j_1, j_2$. In other words, the only functions invariant under $T_{\eta_{j_1}}^{\beta_{j_1}}T_{\eta_{j_2}}^{-\beta_{j_2}}$ are those simultaneously invariant under $T_{\eta_{j_1}}$ and $T_{\eta_{j_2}}$.
In particular, having the good ergodicity property corresponds to having the good ergodicity property for the polynomials $p_1,\ldots, p_\ell$ along the identity indexing tuple $\eta_0 := (1, \ldots, \ell)$. We similarly say that the tuple $(T_{\eta_j}^{\rho_j(n)})_{j\in[\ell]}$ has the \emph{good ergodicity property} if $(X, \CX, \mu, T_1, \ldots, T_\ell)$ has the good ergodicity property along $\eta$ for $\rho_1, \ldots, \rho_\ell$. For instance, the tuple \eqref{tuple example} has the good ergodicity property along $\eta$ precisely when
\begin{align*}
    \CI(T_1 T_2\inv) = \CI(T_1)\cap\CI(T_2)\quad \textrm{and}\quad \CI(T_3 T_4^{-2}) = \CI(T_3)\cap \CI(T_4).
\end{align*}
\end{definition}

The guiding principle behind our arguments is that we derive seminorm control of the average \eqref{general average} by inductively applying seminorm control of an average that is ``simpler'' than the original average in an appropriate sense. For instance, in Proposition \ref{Smoothing of n^2, n^2, n^2+n} and Corollary \ref{Seminorm control of n^2, n^2, n^2+n other terms}, we obtained seminorm control for the tuple $(T_1^{n^2}, T_2^{n^2}, T_3^{n^2+n})$ from Example \ref{Ex: n^2, n^2, n^2+n} by invoking seminorm control for the following tuples:
\begin{itemize}
    \item $(T_1^{n^2}, T_2^{n^2}, T_2^{n^2+n})$ in the \textit{ping} step of the smoothing argument in Proposition \ref{Smoothing of n^2, n^2, n^2+n};
    \item $(T_1^{n^2}, *, T_3^{n^2+n})$ in the \textit{pong} step of the smoothing argument in Proposition \ref{Smoothing of n^2, n^2, n^2+n} (the asterisk is introduced purely for convenience; it denotes the term replaced by a product of dual functions);
    \item $(T_1^{n^2}, T_2^{n^2}, *)$ in Corollary \ref{Seminorm control of n^2, n^2, n^2+n other terms}.
\end{itemize}

The relative complexity of a tuple or an average is captured by the following notion.
\begin{definition}[Type]
The \emph{type} of $\brac{T_{\eta_j}^{\rho_j(n)}}_{j\in[\ell]}$ is the tuple $w = (w_1, \ldots, w_{K_2})$, where
\begin{align*}
    w_t := |\{j\in \FL:\ {\eta_j}\in \FI_t\}| = |\{j\in[\ell]:\ \deg \rho_j = d,\; {\eta_j}\in \FI_t\}|
\end{align*}
counts the number of times the transformations $(T_j)_{j\in \FI_t}$ appear in $\brac{T_{\eta_j}^{\rho_j(n)}}_{j\in[\ell]}$ with a polynomial iterate of maximum degree\footnote{We note here that the type of $\brac{T_{\eta_j}^{\rho_j(n)}}_{j\in[\ell]}$ depends not just on $\eta$ and the polynomials $\rho_1, \ldots, \rho_\ell$, but also on the ordering of the sets $\FI_1, \ldots, \FI_{K_1}$. We do not record this dependence explicitly, instead fixing some ordering of $\FI_1, \ldots, \FI_{K_1}$ a priori.}. We note that $|w| := w_1 +\cdots + w_{K_2} = K_3$. We say that the type $w$
is \textit{basic} if it has  the form  $w=(K_3,0,\ldots, 0)$.
\end{definition}

For instance, the tuple \eqref{tuple example} has type $(3,3,0,0)$: this is because $T_1, T_2$ corresponding to $\FI_1$ occur thrice, as do the transformations $T_3, T_4$ corresponding to $\FI_2$, while the transformations $T_5, T_7$ corresponding to $\FI_3$ and $\FI_4$ do not occur at all. We do not care about the occurrence of $T_6$ since it has a linear iterate.

It is instructive to see what happens when the polynomials $\rho_1, \ldots, \rho_\ell$ are pairwise independent. In that case, $\FI_t = \{j_t\}$ for every $t\in[\ell]$ and $w_t$ counts the number of times the transformation $T_{j_t}$ appears among $(T_{\eta_j})_{j\in \FL}$, or equivalently the number of times that $T_{j_t}$ attains a polynomial iterate of maximal degree. So for pairwise independence polynomials, this notion of type recovers the concept of type from \cite[Section 8.2]{FrKu22} (up to permuting $\FI_1, \ldots, \FI_{K_1}$).

For the set of tuples in $\N_0^{K_2}$ of length $K_3$, we say $w'<w$ if there exists $\kappa\in[K_2-1]$ such that for all $t\in[\kappa]$, we have $w'_t = w_t$, and either $w'_{\kappa+1} = 0 < w_{\kappa+1}$ or $w'_{\kappa+1}>w_{\kappa+1}>0$. For instance, we have the following chain of type inequalities
\begin{align*}
    (4, 0, 0) &< (3, 0, 1) < (3, 1, 0) < (2, 0, 2) < (2, 2, 0)\\
    &< (2, 1, 1) < (1, 0, 3) < (1, 2, 1) < (1, 1, 2).
\end{align*}
The first, third, fifth, sixth, and eighth inequality follow from the condition $w'_{\kappa+1}>w_{\kappa+1}>0$ while the second, fourth, and seventh inequality are consequences of the condition $w'_{\kappa+1} = 0 < w_{\kappa+1}$.
This is a rather atypical ordering, but it turns out to determine well which of the tuples  $\brac{T_{\eta_j}^{\rho_j(n)}}_{j\in[\ell]}$, $\brac{T_{\eta'_j}^{\rho'_j(n)}}_{j\in[\ell]}$ is ``simpler'' than the other. The motivation for this particular choice of ordering is that in the arguments to come, we will be passing from a tuple $\brac{T_{\eta_j}^{\rho_j(n)}}_{j\in[\ell]}$ of type $w$ to another tuple $\brac{T_{\eta'_j}^{\rho'_j(n)}}_{j\in[\ell]}$ of type $w'<w$ in two ways. In one of them, the type $w'$ will meet the condition $w'_{\kappa+1}>w_{\kappa+1}>0$ while in the other one, it will satisfy the condition $w'_{\kappa+1} = 0 < w_{\kappa+1}$. Arguing this way, we arrive in finitely many steps at a tuple of a basic type $w=(K_3, 0, \ldots, 0)$, which constitutes the base case of our induction.  This transition will be explained in greater detail at the end of Section \ref{SS: strategy} and illustrated in Example \ref{Ex: induction scheme ex}.
\begin{lemma}\label{L: type order}
For fixed $K_2,K_3\in\N$, let  $A:=\{w\in\N_0^{K_2}:\ w_1+\cdots+w_{K_2} = K_3\}$. Then $<$ defines a strict partial order on $A$.
\end{lemma}
\begin{proof}
It is clear that $<$ is asymmetric and irreflexive, so it remains to show that it is transitive.
Suppose that $w'' < w'$, $w' < w$, and let $\kappa_1, \kappa_2 \in [K_2 - 1]$ be indices such that
$w''_t = w'_t$ for all $t \in [\kappa_2]$ but not for $t = \kappa_2 + 1$, and
$w'_t = w_t$ for all $t \in [\kappa_1]$ but not for $t = \kappa_1 + 1$.
Let $\kappa = \min(\kappa_1, \kappa_2)$; it suffices to compare
$w_{\kappa+1}$, $w'_{\kappa+1}$, and $w''_{\kappa+1}$.

If $\kappa_1 > \kappa_2$, then either
$0 = w''_{\kappa+1} < w'_{\kappa+1} = w_{\kappa+1}$
or
$w''_{\kappa+1} > w'_{\kappa+1} = w_{\kappa+1} > 0$,
and so $w'' < w$.
Similarly, if $\kappa_1 < \kappa_2$, then either
$0 = w''_{\kappa+1} = w'_{\kappa+1} < w_{\kappa+1}$
or
$w''_{\kappa+1} = w'_{\kappa+1} > w_{\kappa+1} > 0$,
giving $w'' < w$ again.

Lastly, suppose that $\kappa_1 = \kappa_2$.
Then $w'_{\kappa+1}$ is necessarily nonzero (for otherwise $\kappa = \kappa_2$ forces $w''_{\kappa+1} > w'_{\kappa+1} = 0$, and hence $w''>w'$).
Thus we must have
$w'_{\kappa+1} > w_{\kappa+1} > 0$.
Now either $w''_{\kappa+1} = 0$, in which case
$0 = w''_{\kappa+1} < w_{\kappa+1}$,
or $w''_{\kappa+1} > 0$, in which case
$w''_{\kappa+1} > w'_{\kappa+1} > w_{\kappa+1} > 0$.
Either way, we have $w'' < w$.
\end{proof}

\subsection{The general strategy}\label{SS: strategy}
In this section, we outline how to obtain a seminorm control of a given tuple using seminorm control for tuples of lower type or shorter length.
\begin{definition}[Controllable and uncontrollable tuples] Let $t_w := \max\{t:\ w_t>0\}$ be the \textit{last nonzero index} of $w$. We call a tuple $\brac{T_{\eta_j}^{\rho_j(n)}}_{j\in[\ell]}$ of a non-basic type $w$ (or the corresponding average) \textit{controllable}, if there exists an index $m\in[\ell]$ such that:
\begin{itemize}
    \item $\eta_m\in\FI_{t_w}$;
    \item for every other $i\in[\ell]$ with $\eta_i=\eta_m$, we have $\rho_i\neq \rho_m$.
\end{itemize}
If $m$ satisfies the aforementioned assumption, we say that it satisfies the \textit{controllability condition}; in this case, Proposition \ref{strong PET bound} guarantees that the average \eqref{general average} is controlled by $\nnorm{f_m}_{\b_1, \ldots, \b_{s}}$ for \emph{nonzero} vectors $\b_1, \ldots, \b_{s}$. If no such index $m$ exists, we call the tuple $\brac{T_{\eta_j}^{\rho_j(n)}}_{j\in[\ell]}$ \textit{uncontrollable}.
\end{definition}
The previous notions of controllability are supposed to capture whether Proposition~\ref{strong PET bound} is applicable to the relevant tuples in our setting.
\begin{example}[Controllable vs. uncontrollable tuples]\label{Ex: controllable}
Consider the following two tuples
\begin{align}
    \label{controllable} &\brac{T_1^{n^2}, T_2^{n^2}, T_3^{n^2}, T_4^{n^2}, T_5^{n^2+n}, T_1^{n^2+n}, T_5^{n^2+2n}, T_5^{n^2+2n}}\\
    \label{uncontrollable} &\brac{T_1^{n^2}, T_2^{n^2}, T_3^{n^2}, T_4^{n^2}, T_1^{n^2+n}, T_1^{n^2+n}, T_5^{n^2+2n}, T_5^{n^2+2n}}.
\end{align}
Defining the partitions $\CI_1 = \{1, 2, 3, 4\}$, $\CI_2 = \{5, 6\}$, $\CI_3 = \{7, 8\}$ corresponding to the independent polynomials $n^2, n^2+n, n^2+2n$ respectively, the first tuple has type $(5, 3, 0)$ while the second one has type $(6, 2, 0)$, and for both tuples  we have $t_w=2$. The  first one is controllable because for the index  $m = 5$, the only values $i\neq m$ such that $\eta_i = 5$ are $i=7,8$ corresponding to the polynomial $n^2+2n$, which is distinct from $n^2+n$.

By contrast, the tuple \eqref{uncontrollable} does not possess an index satisfying the controllability condition: the only indices $m\in[8]$ with $\eta_m\in\FI_2$ are $m=7,8$, and we have both $\eta_7 = \eta_8$ and $\rho_7(n) = n^2+2n = \rho_8(n)$. Hence, this tuple is uncontrollable.
\end{example}

Our strategy for proving seminorm control will work rather differently for controllable and uncontrollable tuples. Suppose first that the average \eqref{general average} with tuple $\brac{T_{\eta_j}^{\rho_j(n)}}_{j\in[\ell]}$ of a non-basic type is controllable, and that an index $m$ satisfies the controllability condition. Then Proposition \ref{strong PET bound} guarantees that the average \eqref{general average} is controlled by $\nnorm{f_m}_{\b_1, \ldots, \b_{s+1}}$ for some vectors $\b_1, \ldots, \b_{s+1}$ from \eqref{E:coefficientvectors}, and the controllability implies that these vectors are indeed nonzero. We want to show that this average is also controlled by $\nnorm{f_m}_{\b_1, \ldots, \b_{s}, \be_{\eta_m}^{\times s'}}$ for some $s'\in\N$ via a seminorm smoothing argument that generalises Proposition \ref{Smoothing of n^2, n^2, n^2+n}. We then iterate this result $s$ more times to get control by a $T_{\eta_m}$-seminorm of $f_m$. The seminorm smoothing argument follows a ping-pong strategy much like in the proof of Proposition \ref{Smoothing of n^2, n^2, n^2+n}. We first show that control by $\nnorm{f_m}_{\b_1, \ldots, \b_{s+1}}$ implies control by $\nnorm{f_i}_{\b_1, \ldots, \b_{s}, \be_{\eta_i}^{\times s_1}}$ for some $s_1\in\N$ and $i\neq m$, and then we use this auxiliary result to obtain control by $\nnorm{f_m}_{\b_1, \ldots, \b_{s}, \be_{\eta_m}^{\times s'}}$.

The main idea behind the \textit{ping} step of the seminorm smoothing argument is to show that a seminorm control of the average \eqref{general average} can be deduced from a seminorm control of a family of averages of the form
\begin{align}\label{intermediate average}
    \lim_{N\to\infty}\E_{n\in [N]} \,\prod_{j\in[\ell]}T_{{\eta_j'}}^{\rho'_j(n)}f'_j \cdot \prod_{j\in[L']}\CD'_{j}(q'_j(n))
\end{align}
for some polynomials $\rho'_1, \ldots, \rho'_\ell, q'_1, \ldots, q'_{L'}\in\Z[n]$, 1-bounded functions $f'_1, \ldots, f'_\ell\in L^\infty(\mu)$  and sequences of functions $\CD'_1, \ldots, \CD'_{L'}\in\FD$. Moreover, the indexing tuple $\eta'\in[\ell]^\ell$ is obtained from $\eta$ by changing $\eta_m$ into $\eta_i$ for some $i\neq m$, i.e. the passage from \eqref{general average} to \eqref{intermediate average} goes by replacing $T_{\eta_m}$ at index $m$ with $T_{\eta_i}$. Importantly, the new average \eqref{intermediate average} satisfies several key properties:
\begin{enumerate}
    \item it has a lower type than the original average, so that we can argue by induction;
    \item the new average retains the good ergodicity property of the original average;
    \item the functions $f'_j$ in the new average satisfy some invariance properties;
    \item  as long as the aforementioned invariance properties are satisfied, the new average \eqref{intermediate average} is controlled by the seminorm $\nnorm{f_j'}_{s, T_{\eta'_j}}$  for each $j\in[\ell]$ and some $s\in\N$.
\end{enumerate}
Proposition \ref{P: type reduction} explains the exact way in which we pass from averages \eqref{general average} to \eqref{intermediate average} so that the property (i) is satisfied, and Proposition \ref{P: properties of descendants} establishes  the property (ii). Proposition \ref{propagation of invariance} then ensures that the functions $f_j'$ in \eqref{intermediate average} satisfy needed invariance properties.

We note though that the new average \eqref{intermediate average} need not be controllable. For instance, if we take the average corresponding to the tuple \eqref{controllable}, then in the \textit{ping} step we replace $T_5^{n^2+n}$ by the same iterate of one of $T_1, T_2, T_3, T_4$. The new average is then uncontrollable, as is the tuple \eqref{uncontrollable}, corresponding to replacing $T_5^{n^2+n}$ by $T_1^{n^2+n}$. Hence, controllability may not be preserved while performing the procedure outlined above.

In the \textit{pong} step of the smoothing argument for \eqref{general average}, we deal with averages of the form
\begin{align}\label{intermediate average 2}
    \E_{n\in [N]} \,\prod_{\substack{j\in[\ell],\\ j\neq i}}T_{\eta_j}^{\rho_j(n)}f''_j \cdot \prod_{j\in[L'']}\CD''_{j}(q''_j(n)).
\end{align}
Crucially, each function $f''_j$ is invariant under some composition of $T_{\eta_j}$ and $T_j$. This allows us to replace (some iterate of) $T_{\eta_j}$ in \eqref{intermediate average 2} by (some iterate of) $T_j$, a procedure that we call \emph{flipping}, and show that an average \eqref{intermediate average 2} essentially equals an average of the form
\begin{align}\label{original average}
    \E_{n\in [N]} \,\prod_{\substack{j\in[\ell],\\ j\neq i}}T_j^{\rho'''_j(n)}f'''_j \cdot \prod_{j\in[L'']}\CD''_{j}(q'''_j(n))
\end{align}
of length $\ell -1$. The details of how flipping is performed are presented in Proposition \ref{P:flipping}. An inductive application of a suitable modification of Theorem \ref{T:Host Kra characteristic} then gives a control of \eqref{original average} by a $T_j$-seminorm of $f'''_j$ for each $j\neq i$, and the invariance property of $f''_j$ translates it into a control of \eqref{intermediate average 2} by a $T_{\eta_j}$-seminorm of $f''_j$ for each $j\neq i$. A straightforward argument analogous to one at the end of the proof of Proposition \ref{Smoothing of n^2, n^2, n^2+n} gives a control of \eqref{general average}  by a $T_{\eta_m}$-seminorm of $f_m$.

If the average \eqref{general average} is uncontrollable, then we proceed rather differently. The previous strategy breaks right at the start since there is no index $m$ satisfying the controllability condition. Consequently, whichever index $m$ with $\eta_m\in\FI_{t_w}$ we take, we cannot employ Proposition \ref{strong PET bound} to bound the seminorm by $\nnorm{f_m}_{\b_1, \ldots, \b_{s+1}}$ for nonzero vectors $\b_1, \ldots, \b_{s+1}$. What we use instead is the inductive assumption that the functions $f_j$ at indices $j$ with $\eta_j\in\FI_{t_w}$ are invariant under a composition of (some power of) $T_{\eta_j}$ and (some power of) $T_j\inv$. Using this invariance property, we perform flipping once more to replace the original average \eqref{general average} by a new average
\begin{align}\label{intermediate average 3}
    \lim_{N\to\infty}\E_{n\in [N]} \,\prod_{j\in[\ell]}T_{{\eta_j''''}}^{\rho''''_j(n)}f''''_j \cdot \prod_{j\in[L]}\CD_{j}(q''''_j(n)),
\end{align}
where $\eta''''_j = j$ whenever $\eta_j\in \FI_{t_w}$; the details are provided in Corollary \ref{C: Flipping uncontrollable tuples}. This new  average has the good ergodicity property and is controllable. Importantly, it has a lower type, which is established in Proposition \ref{P: induction scheme}. We can then obtain seminorm control of \eqref{general average} by inductively invoking the seminorm control of \eqref{intermediate average 3}. The seminorm control of \eqref{intermediate average 3} is proved in turn by the smoothing argument for controllable averages described above.

Thus, whether the tuple $\brac{T_{\eta_j}^{\rho_j(n)}}_{j\in[\ell]}$ of a non-basic type is controllable or not, the idea is to control it by a Gowers-Host-Kra seminorm by invoking seminorm control for tuples of lower type or smaller length that naturally appear when examining $\brac{T_{\eta_j}^{\rho_j(n)}}_{j\in[\ell]}$. If the tuple $\brac{T_{\eta_j}^{\rho_j(n)}}_{j\in[\ell]}$ of type $w$ is controllable, we will invoke seminorm control for tuples of type $w'$ satisfying $w'_t = w_t$ for $t\in[\kappa]$ and $w'_{\kappa+1}>w_{\kappa+1}>0$ for some $\kappa\in[K_2-1]$ in the \textit{ping} step of the smoothing argument. Specifically, the new type $w'$ is obtained from $w$ by the type operation defined in \eqref{definition sigma}. In the \textit{pong} step of the smoothing argument, we will use seminorm control for tuples of length $\ell-1$. If the tuple $\brac{T_{\eta_j}^{\rho_j(n)}}_{j\in[\ell]}$ is uncontrollable, we will invoke seminorm control for tuples of type $w'$ satisfying $w'_t = w_t$ for $t\in[\kappa]$ and $w'_{\kappa+1}=0<w_{\kappa+1}$ for some $\kappa\in [K_2-1]$; the details are given in Proposition \ref{P: induction scheme}(v). The way in which we apply seminorm control for tuples of lower type motivates the choice of our somewhat weird ordering on types.

Reducing to tuples of lower type this way and noting that the tuples of length $\ell$ can have at most $(\ell+1)^\ell$ distinct types, we arrive after finitely many steps at tuples of basic type $w = (K_3, 0, \ldots, 0)$, i.e. those in which all the transformations come from the same class $\FI_1$. Tuples of basic type will serve as the basis for our induction procedure. For instance, the tuple $(T_1^{n^2}, T_2^{n^2}, T_2^{n^2+n})$ from Example \ref{Ex: n^2, n^2, n^2+n} has basic type $(3, 0)$ because it only involves the transformations $T_1, T_2$ whose indices belong to the set $\FI_1 = \{1, 2\}$ (corresponding to the polynomial $n^2$); however the type $(2, 1)$ of $(T_1^{n^2}, T_2^{n^2}, T_3^{n^2+n})$ is not basic because this tuple involves both the transformations $T_1, T_2$ and the transformation $T_3$ with an index from the set $\FI_2 = \{3\}$ (corresponding to the polynomial $n^2 + n$).

\section{Further maneuvers and obstructions for longer families}\label{S:maneuvers}
Having presented the general strategy for proving Theorem \ref{T:Host Kra characteristic} for longer families, we move on to discuss in detail the specific maneuvers outlined in Section \ref{SS: strategy}. In this section, we state and prove various partial results that give substance to the moves discussed in Section \ref{SS: strategy}. We also discuss a number of obstructions that appear in the process and have to be overcome before we can give a complete proof of Theorem \ref{T:Host Kra characteristic}. All of the above is illustrated with examples that will hopefully make the abstract statements in this and the next section more comprehensible to the reader. We then move on in Section \ref{S:smoothing} to prove Theorem \ref{T:Host Kra characteristic}.

The plan for this section is as follows. In Section \ref{SS:ping}, we discuss how to obtain a tuple of lower type in the \textit{ping} step of the seminorm smoothing argument for controllable tuples. Section \ref{SS:flipping} exhibits the necessity of assuming that the functions appearing in the averages \eqref{intermediate average} have some invariance properties. In particular, we show on examples how these properties are essentially used to tackle tuples of basic type and to perform the \textit{pong} step of the seminorm smoothing argument for controllable tuples. We also give details of the flipping procedure that relies on these invariance properties. Subsequently, we discuss in Section \ref{SS:uncontrollable} how flipping can be used to reduce an uncontrollable tuple to a controllable tuple of a lower type. Finally, we combine the details of the aforementioned moves in Section \ref{SS:recapitulation} and show how we can reach a tuple of a basic type in a finite number of steps.

\subsection{Reducing controllable tuples to tuples of lower type in the \textit{ping} step}\label{SS:ping}
As explained in Section \ref{SS: strategy}, in the \textit{ping} part of the smoothing argument, we will replace the original tuple $\brac{T_{\eta_j}^{\rho_j(n)}}_{j\in[\ell]}$ by a new tuple $\brac{T_{\eta'_j}^{\rho'_j(n)}}_{j\in[\ell]}$. The new indexing tuple $\eta'$ will be defined via the operation
\begin{align}\label{definition tau}
    (\tau_{mi}\eta)_j := \begin{cases} \eta_j,\; &j\neq m\\ \eta_i,\; &j = m
    \end{cases}
\end{align}
for some distinct values $m, i\in\FL$. This indexing tuple corresponds to replacing the term $T_{\eta_m}^{\rho_m(n)}$ in $\brac{T_{\eta_j}^{\rho_j(n)}}_{j\in[\ell]}$ by $T_{\eta_i}^{\rho'_i(n)}$, and all the other terms $T_{\eta_j}^{\rho_j(n)}$ by $T_{\eta_j}^{\rho'_j(n)}$. The new tuple $\brac{T_{\eta'_j}^{\rho'_j(n)}}_{j\in[\ell]}$ has to be chosen carefully: it must preserve the good ergodicity property and allow for seminorm control. Lastly, it must have a lower type. For this reason, we let $\supp(w):= \{t\in[K_2]:\ w_t>0\}$, and if there exist distinct integers $t_1,t_2\in \supp(w)$, we define the type operation $\sigma_{t_1 t_2}w$ by the formula
\begin{align}\label{definition sigma}
    (\sigma_{t_1 t_2}w)_t := \begin{cases} w_t,\; &t \neq t_1, t_2\\
    w_{t_1}-1,\; &t = t_1\\
    w_{t_2} + 1,\; &t = t_2.
    \end{cases}.
\end{align}
For instance, $\sigma_{32}(3, 2, 2) = (3, 3, 1)$. As a consequence of our ordering on types, we have $\sigma_{t_1 t_2}w < w$ whenever $t_2 < t_1$ (the assumption $w_{t_2}>0$ is crucial here), so in particular $(3, 3, 1) < (3, 2, 2)$.

Proposition \ref{P: type reduction}, which we are about to state now, specifies how these tuples of lower type are picked, what form they take, and what properties they enjoy. It will be used in our smoothing argument in Proposition \ref{P:smoothing} in that the tuple of lower type for which we invoke the induction hypothesis in the \textit{ping} step is constructed in Proposition \ref{P: type reduction}.

\begin{proposition}[Type reduction for controllable tuples]\label{P: type reduction}
Let $\ell\in\N$, $\eta\in[\ell]^\ell$ be an indexing tuple, $\rho_1, \ldots, \rho_\ell\in\Z[n]$ be polynomials with leading coefficients $b_1, \ldots, b_\ell$. Let also  $(X, \CX, \mu, T_1, \ldots, T_\ell)$ be a system and $\brac{T_{\eta_j}^{\rho_j(n)}}_{j\in[\ell]}$ be a tuple of a non-basic type $w$ whose last nonzero index is $t_w$. Suppose that $\brac{T_{\eta_j}^{\rho_j(n)}}_{j\in[\ell]}$ is controllable.
 Then there exists $\lambda\in\N$ such that for every $r\in\{0, \ldots, \lambda-1\}$, there exist an index $t'_w\in\supp(w)$ distinct from $t_w$, an index $i\in[\ell]$ with $\eta_i\in\FI_{t'_w}$, and a tuple $\brac{T_{\eta'_j}^{\rho'_j(n)}}_{j\in[\ell]}$ satisfying the following properties.
\begin{enumerate}
    \item The type $w'$ of the tuple $\brac{T_{\eta'_j}^{\rho'_j(n)}}_{j\in[\ell]}$ satisfies $w' = \sigma_{t_w t'_w}w < w$.
    \item The indexing tuple $\eta'$ is given by $\eta' := \tau_{m i} \eta$ for some $m\in[\ell]$ satisfying the controllability condition (recall that $\tau_{mi}$ is defined in \eqref{definition tau}).
    \item The polynomials $\rho'_{1}, \ldots, \rho'_{\ell}$ have integer coefficients and zero constant terms, and they take the form
\begin{align*}
    \rho'_{j}(n) := \begin{cases} \rho_{j}(\lambda n + r) - \rho_{j}(r),\; &j\neq m\\ \frac{b_{i}}{b_{m}}(\rho_j(\lambda n + r) - \rho_j(r)),\; &j= m. \end{cases}
\end{align*}
\end{enumerate}

\end{proposition}

We remark that when the leading coefficients of $\rho_1, \ldots, \rho_\ell$ are 1, then $\rho'_j  = \rho_j$ for every $j\in[\ell]$, so the property (iii) becomes trivial.

\begin{proof}
Let $t_w$ be the last nonzero index of the type $w$ of $\brac{T_{\eta_j}^{\rho_j(n)}}_{j\in[\ell]}$. By the controllability of the tuple, there exists $m\in[\ell]$ with $\eta_m\in\FI_{t_w}$ such that $\eta_{m'}=\eta_m$ implies that $\rho_{m'}$ and $\rho_m$ are distinct. Let $t'_w\in\supp(w)$ be an index different from $t_w$ (it exists since the type $w$ is non-basic) and $i\in[\ell]$ be an index with $\eta_i\in \FI_{t'_w}$. We define $\eta':=\tau_{mi}\eta$, meaning that we replace $T_{\eta_m}$ by $T_{\eta_i}$ and keep the other transformations the same.



We let $\lambda\in\N$ be the smallest number for which $\frac{\lambda}{b_m}\rho_m\in\Z[n]$ (equivalently, $\lambda$ is the smallest number such that $b_m$ divides the coefficients of the polynomials $\rho_m(\lambda n + r) - \rho_m(r)$ for $r\in\Z$). We also fix an arbitrary $r\in\{0, \ldots, \lambda-1\}$. We then define the new polynomials $\rho'_1, \ldots, \rho'_\ell$ by the formula
\begin{align*}
    \rho'_{j}(n) := \begin{cases} \rho_{j}(\lambda n + r) - \rho_{j}(r),\; &j\neq m\\ \frac{b_{i}}{b_{m}}(\rho_j(\lambda n + r) - \rho_j(r)),\; &j= m. \end{cases}
\end{align*}
The new polynomials are in $\Z[n]$ by the choice of $\lambda$, and it is not hard to check that they have zero constant terms.
Lastly, the new tuple $\brac{T_{\eta'_j}^{\rho'_j(n)}}_{j\in[\ell]}$ has the type $w' = \sigma_{t_w t'_w}w$, which is strictly  smaller than $w$ by the assumption that $t'_w<t_w$ (which follows from $t'_w\neq t_w$ and the assumption that $t_w$ is the last nonzero index).
\end{proof}

\begin{example}[Examples of type reduction]\label{Ex: type reduction for earlier examples}
We show how Proposition \ref{P: type reduction} has been implicitly applied to  the two tuples from Section \ref{S:n^2, n^2, n^2+n}.
\begin{enumerate}
    \item The tuple $(T_1^{n^2}, T_2^{n^2}, T_3^{n^2+n})$ discussed at length in Section \ref{S:n^2, n^2, n^2+n} has type $(2, 1)$ corresponding to the partition $\FI_1 = \{1,2\}, \FI_2 = \{3\}$, and its good ergodicity property means that $\CI(T_1 T_2\inv)=\CI(T_1)\cap\CI(T_2)$, i.e. the only $T_1 T_2\inv$-invariant functions are those invariant simultaneously under $T_1$ and $T_2$. In the proof of Proposition \ref{Smoothing of n^2, n^2, n^2+n}, we applied the type reduction once (with the operation $\tau_{32}$) to obtain the new tuple $(T_1^{n^2}, T_2^{n^2}, T_2^{n^2+n})$ with indexing tuple $(1,2,2)$ and basic type $(3,0)$.
    \item The tuple $(T_1^{n^2}, T_2^{n^2}, T_3^{2n^2+n})$ presented at the end of Section \ref{S:n^2, n^2, n^2+n} also has type $(2, 1)$ corresponding to the partition $\FI_1 = \{1,2\}, \FI_2 = \{3\}$, and its good ergodicity property also states that $\CI(T_1 T_2\inv)=\CI(T_1)\cap\CI(T_2)$. In the \textit{ping} step of the smoothing argument, we obtained  (upon taking $r_0 = 1$) the new tuple $(T_1^{4n^2+4n}, T_2^{4n^2+4}, T_2^{4n^2+5n})$ by performing the operation $\tau_{32}$. This new tuple also has the indexing tuple $(1,2,2)$ and basic type $(3,0)$, and its ergodicity property is the same as for the original tuple.
\end{enumerate}
\end{example}

\begin{definition}[Descendants]
    Let $p_1, \ldots, p_\ell\in\Z[n]$ be polynomials with leading coefficients $a_1, \ldots, a_\ell$ and $\eta\in[\ell]^\ell$ be an indexing tuple. We say that the polynomials $\rho_1, \ldots, \rho_\ell$ are \textit{descendants of $p_1, \ldots, p_\ell$ along $\eta$}, if there exists $\lambda\in\N$ and $r\in\{0, \ldots, \lambda-1\}$ such that $\rho_j(n) = \frac{a_{\eta_j}}{a_j}(p_j(\lambda n+r) - p_j(r))$. If this is the case, we also say the tuple $\brac{T_{\eta_j}^{\rho_j(n)}}_{j\in[\ell]}$ is a \textit{descendant} of the tuple $\brac{T_j^{p_j(n)}}_{j\in[\ell]}$.
\end{definition}
 Descendancy enjoys the following transitivity property.
\begin{lemma}[Descendancy is transitive]\label{L: descendancy transitive}
Suppose that the polynomials  $\rho_1, \ldots, \rho_\ell\in \Z[n]$ are descendants of $p_1, \ldots, p_\ell\in \Z[n]$ along $\eta$, and let $\rho'_j(n) := \frac{a_{\eta'_j}}{a_{\eta_j}}(\rho_j(\lambda' n + r') - \rho_j(r'))$ for all $j\in[\ell]$, where $\eta'\in[\ell]^\ell$, $\lambda'\in\N$, $r'\in\{0, \ldots, \lambda-1\}$ and $a_1, \ldots, a_\ell$ are the leading coefficients of $p_1, \ldots, p_\ell$. Then $\rho'_1, \ldots, \rho'_\ell$ are descendants of $p_1, \ldots, p_\ell$ along $\eta'$.
\end{lemma}
\begin{proof}
Let $\lambda\in\N$ and $r\in\{0, \ldots, \lambda-1\}$ be such that $\rho_j(n) = \frac{a_{\eta_j}}{a_j}(p_j(\lambda n+r) - p_j(r))$. Then a direct computation gives that
\begin{align*}
    \rho'_j(n) = \frac{a_{\eta'_j}}{a_j}(p_j(\lambda\lambda' n + \lambda r' + r) - p_j(\lambda r' + r)),
\end{align*}
giving the claim.
\end{proof}
In particular, we get the following corollary of interest to us that follows from a straightforward combination of Proposition \ref{P: type reduction} and Lemma \ref{L: descendancy transitive}.
\begin{corollary}[Type reduction preserves descendancy]\label{C: ping reduction transitive}
    Let $\ell\in\N$, $\eta\in[\ell]^\ell$ be an indexing tuple, $p_1, \ldots, p_\ell, \rho_1, \ldots, \rho_\ell\in\Z[n]$ be polynomials, and $(X, \CX, \mu, T_1, \ldots, T_\ell)$ be a system. Suppose that $\brac{T_{\eta_j}^{\rho_j(n)}}_{j\in[\ell]}$ is a tuple of a non-basic type $w$ that is a descendant of $\brac{T_j^{p_j(n)}}_{j\in[\ell]}.$ Then the tuple $\brac{T_{\eta'_j}^{\rho'_j(n)}}_{j\in[\ell]}$ constructed from $\brac{T_{\eta_j}^{\rho_j(n)}}_{j\in[\ell]}$ in Proposition \ref{P: type reduction} is also a descendant of $\brac{T_j^{p_j(n)}}_{j\in[\ell]}.$
\end{corollary}
\begin{proof}
Suppose that $\rho_1, \ldots, \rho_\ell$ are descendants of $p_1, \ldots, p_\ell$ along $\eta$ by assumption.
Letting $a_j, b_j$ be the leading coefficients of $p_j, \rho_j$ respectively and $\rho'_j$ be as defined in Proposition \ref{P: type reduction}, we have $b_j = a_{\eta_j} (\lambda')^{d_j}$ for some $\lambda\in \Z$, where $d_j := \deg p_j = \deg \rho_j = \deg \rho'_j$. Thus,  $$\rho'_j(n) = \rho_j(\lambda n + r) - \rho_j(r) =\frac{a_{\eta'_j}}{a_{\eta_j}}(\rho_j(\lambda n + r) - \rho_j(r))$$ for $j\neq m$ and $$\rho'_j(n) = \frac{b_i}{b_m}(\rho_j(\lambda n + r) - \rho_j(r)) = \frac{a_{\eta_i}}{a_{\eta_m}}(\rho_j(\lambda n + r) - \rho_j(r)) = \frac{a_{\eta_m'}}{a_{\eta_m}}(\rho_j(\lambda n + r) - \rho_j(r))$$ for $j = m$, where we use $\eta'_m = \eta_i$ and $\eta'_j = \eta_j$ for $j\neq m$. Hence, the polynomials $\rho'_1, \ldots, \rho'_\ell$ satisfy the condition of Lemma \ref{L: descendancy transitive}, implying the claim.
\end{proof}

Descendant tuples enjoy the following important properties.
\begin{proposition}[Properties of descendants]\label{P: properties of descendants}
    Let $\ell\in\N$, $\eta\in[\ell]^\ell$ be an indexing tuple, $p_1, \ldots, p_\ell, \rho_1, \ldots, \rho_\ell\in\Z[n]$ be polynomials, and $(X, \CX, \mu, T_1, \ldots, T_\ell)$ be a system. Suppose that $\brac{T_j^{p_j(n)}}_{j\in[\ell]}$ has the good ergodicity property and $\brac{T_{\eta_j}^{\rho_j(n)}}_{j\in[\ell]}$  is its descendant. Then the following holds:
    \begin{enumerate}
        \item We have
        \begin{align*}
    \FL(\rho_1, \ldots, \rho_\ell) &= \FL(p_1, \ldots, p_\ell), \\
    K_i(\rho_1, \ldots, \rho_\ell) &=K_i(p_1, \ldots, p_\ell),\quad  i\in[3],\\
    \FI_t(\rho_1, \ldots, \rho_\ell)&=\FI_t(p_1, \ldots, p_\ell), \quad t\in[K_1].
        \end{align*}
        \item The tuple $\brac{T_{\eta_j}^{\rho_j(n)}}_{j\in[\ell]}$ has the good ergodicity property.
    \end{enumerate}
\end{proposition}
Property (i) ensures that when passing to descendants, we do not need to redefine the partition $\FI_1, \ldots, \FI_{K_1}$. Property (ii) is crucial because it shows that descendants retain the essential ergodicity properties of the original tuple.
\begin{proof}
Part (i) follows from the fact that for every $j\in[\ell]$, the polynomials $p_j$ and $\rho_j$ have the same degree, and that $p_{j_1}, p_{j_2}$ are linearly dependent if and only if $\rho_{j_1}, \rho_{j_2}$ are. We therefore move on to proving part (ii). Let $b_j$ be the leading coefficient of $\rho_j$, $a_j$ be the leading coefficient of $p_j$, and $d_j:= \deg p_j = \deg \rho_j$ for every $j\in[\ell]$. To check that the tuple $\brac{T_{\eta_j}^{\rho_j(n)}}_{j\in[\ell]}$ has the good ergodicity property, we need to show that if $\eta_{j_1}, \eta_{j_2}$ are distinct elements of the same set $\FI_{t}$, then
\begin{align}\label{E:erg con in type reduction}
    \CI(T_{\eta_{j_1}}^{\beta_{j_1}} T_{\eta_{j_2}}^{-\beta_{j_2}})=\CI(T_{\eta_{j_1}})\cap\CI(T_{\eta_{j_2}}),
\end{align}
where $$\beta_{j} := b_{j}/\gcd(b_{j_1}, b_{j_2})$$ for $j = j_1, j_2$.  By construction, $b_j = a_{\eta_j}\lambda^{d_j}$ for some $\lambda\in\N$, and so
\begin{align}\label{E: beta=alpha}
    \beta_{j} = a_{\eta_j}/\gcd(a_{\eta_{j_1}}, a_{\eta_{j_2}})=:\alpha_{\eta_j}
\end{align}
for $j= j_1, j_2$. The assumption $\eta_{j_1}, \eta_{j_2}\in\FI_t$ for some fixed $t$ implies that $p_{\eta_{j_1}}, p_{\eta_{j_2}}$ are linearly dependent, and additionally $p_{\eta_{j_1}}/\alpha_{\eta_{j_1}} = p_{\eta_{j_2}}/\alpha_{\eta_{j_2}}$. Since $\alpha_{\eta_{j_1}}, \alpha_{\eta_{j_2}}$ are coprime, the good ergodicity property of $\brac{T_j^{p_j(n)}}_{j\in[\ell]}$ implies that
\begin{align*}
    \CI(T_{\eta_{j_1}}^{\alpha_{\eta_{j_1}}} T_{\eta_{j_2}}^{-\alpha_{\eta_{j_2}}})=\CI(T_{\eta_{j_1}})\cap\CI(T_{\eta_{j_2}}).
\end{align*}
The equality \eqref{E:erg con in type reduction} follows from this and the identification \eqref{E: beta=alpha}.
\end{proof}

\begin{example}[Type reduction for non-monic polynomials]\label{Ex:type reduction example}
We present one more example to show how Proposition \ref{P: type reduction} is applied iteratively for more complicated tuples, and how properties listed in Proposition \ref{P: properties of descendants} are retained when passing to lower-type descendant tuples.
Consider the tuple
\begin{align}\label{example tuple 20}
    (T_1^{n^2}, T_2^{3n^2}, T_3^{2n^2}, T_4^{2n^2+n}, T_5^{n^2+n}, T_6^{n^2+n}, T_7^n),
\end{align}
and assume that it has the good ergodicity property, i.e.
\begin{align*}
    \CI(T_1 T_2^{-3}) = \CI(T_1)\cap\CI(T_2),\quad \CI(T_1 T_3^{-2}) = \CI(T_1)\cap\CI(T_3),\\
    \CI(T_2^3 T_3^{-2}) = \CI(T_2)\cap\CI(T_3),\quad \CI(T_5 T_6^{-1}) = \CI(T_5)\cap\CI(T_6).
\end{align*}
This tuple has length 7, degree 2, $K_1 = 5$, $K_2 = 4$, $K_3=6$  and $\FL = \{1, 2, 3, 4, 5, 6\}$. If we define the partition\footnote{Perhaps a more natural way to define the partition would be to have $\FI_1=\{1, 2, 3\}$, $\FI_2 = \{4\}$, $\FI_3 = \{5,6\}$, $\FI_4 = \{7\}$. But then the tuple would have type $(3, 1, 2)$, and reducing to the tuple of basic type by iteratively applying Proposition \ref{P: type reduction} would take more steps. This shows that choosing the partition strategically can save on the number of iterations of Proposition \ref{P: type reduction} needed to reach a tuple of basic type.} $\FI_1=\{1, 2, 3\}$, $\FI_2 = \{5,6\}$, $\FI_3 = \{4\}$, $\FI_4 = \{7\}$, then the tuple has type $w_0 = (3,2, 1)$; we recall that the term $T_7^n$ plays no part in the type consideration since the polynomial $\rho_{07}(n)=n$ has a lower degree. The tuple \eqref{example tuple 20} has the basic indexing tuple $\eta_0 = (1, 2, 3, 4, 5, 6,7)$. The tuple is controllable, and 4 satisfies the controllability condition, so in the first step we replace $T_4$ (this corresponds to us wanting to first get a $T_4$-seminorm control over the tuple \eqref{example tuple 20}). We are then provided with an index $i_0\in\FI_1\cup\FI_2$ (say, $i_0 = 1$), and we get the new indexing tuple
$$\eta_1 := \tau_{m_0 i_0} \eta_0 = \tau_{41}\eta_0 = (1, 2, 3, 1, 5, 6, 7)$$
The leading coefficient 2 of $2n^2+n$ does not divide the linear coefficient, and the smallest $\lambda_0\in\N$ such that $2$ divides the coefficients of $\lambda_0(2n^2+n)$ is $\lambda_0 = 2$. In performing the \textit{ping} step of the seminorm smoothing argument for the tuple \eqref{example tuple 20}, we will want to apply the $T_1 T_4^{-2}$ invariance of some function $u$ to replace $T_4^{2n^2+n}u$ by $T_1^{q(n)}u'$ for some $q\in\Z[n]$ and a function $u'$ related in some way to $u$. We cannot do this directly since $\frac{1}{2}(2n^2+n)\notin\Z[n]$, but we can do this ``piecewise'' by splitting $\N$ into arithmetic progressions $(2\N + r)_{r=0, 1}$ and considering the two cases separately (see the sketch of the seminorm smoothing argument for $n^2, n^2, 2n^2+n$ at the end of Section \ref{S:n^2, n^2, n^2+n} to see how this was done for that family). We therefore replace the original polynomials $\rho_{01}, \ldots, \rho_{07}$ by new polynomials
\begin{align*}
    \rho_{1j}(n) := \begin{cases} \rho_{0j}(2 n + r_0) - \rho_{0j}(r_0),\; &j \neq 4\\
    \frac{1}{2}(\rho_{04}(2 n + r_0) - \rho_{04}(r_0)),\; &j =4
    \end{cases}
\end{align*}
for some $r_0\in\{0,1\}$ (the choice of $r_0$ is not ours).
Assuming that $r_0=1$, we obtain the new tuple
\begin{align}\label{example tuple 21}
    (T_1^{4n^2 + 4n}, T_2^{12n^2+12n}, T_3^{8n^2+8n}, T_1^{4n^2+5n}, T_5^{4n^2+6n}, T_6^{4n^2+6n}, T_7^{2n}).
\end{align}
The type of the new tuple is $w_1 = \sigma_{31}w_0 = (4, 2,0)$ since we now have four transformations with indices coming from $\FI_1$ and two transformations coming from $\FI_2$. This type is lower than the original type $w_0$, and so we have successfully obtained a tuple of lower type. The new tuple is controllable, with $m=5,6$ both satisfying the controllability condition.

Although we replaced the polynomials $\rho_{01}, \ldots, \rho_{07}$ by new ones, we note that for any $j_1, j_2\in[7]$, the polynomials $\rho_{1j_1}, \rho_{1j_2}$ are pairwise dependent if and only if $\rho_{0j_1}, \rho_{0j_2}$ are, and not only that: if they are pairwise dependent, then $\rho_{1j_1}/c_1 = \rho_{1j_2}/c_2$ if and only if $\rho_{0j_1}/c_1 = \rho_{0j_2}/c_2$ for  any nonzero integers $c_1, c_2$. Moreover, if $\eta_{1j_1} = \eta_{1j_2}$, then the leading coefficients of $\rho_{1j_1}$ and $\rho_{1j_2}$ are identical. These two observations ensure that the ergodicity conditions on $T_1 T_2^{-3}, T_1 T_3^{-2}, T_2^3 T_3^{-2}, T_5 T_6\inv$, which constitute the assumption that the original tuple \eqref{example tuple 20} has the good ergodicity property, carry on to the new tuple \eqref{example tuple 21}, implying that it also enjoys the good ergodicity property. This exemplifies the claim from Proposition \ref{P: properties of descendants} that descendants of tuples with the good ergodicity property inherit the property.

The type $w_1$ is not basic, and so we continue the procedure. This time, we pick some $m_1\in\FI_2$, say $m_1 = 5$ (it satisfies the controllability condition, as does 6, the other possible choice). We are then handed an index $i_1\in\FI_1$ (say, $i_1 = 3$), so that
\begin{align*}
    \eta_2 := \tau_{m_1 i_1}\eta_1= \tau_{53}\eta_1 = (1, 2, 3, 1, 3, 6, 7).
\end{align*}
The leading coefficient $4$ of $4n^2+6n$ does not divide the linear term, and so we replace the polynomials $\rho_{11}, \ldots, \rho_{17}$ by new polynomials of the form
\begin{align*}
    \rho_{2j}(n) := \begin{cases} \rho_{1j}(2 n + r_1) - \rho_{1j}(r_1),\; &j \neq 5\\
    \frac{8}{4}(\rho_{15}(2 n + r_1) - \rho_{15}(r_1)),\; &j =5
    \end{cases}
\end{align*}
for some $r_1\in\{0,1\}$ (we pass from $n$ to $2n+r_1$ because 2 is the smallest natural number $\lambda_1$ such that the leading coefficient 4 of $\rho_{15}$ divides the coefficients of $\lambda_1 \rho_{15}$). Hence, the new tuple takes the form (upon assuming $r_1 = 0$)
\begin{align}\label{example tuple 22}
    (T_1^{16n^2 + 8n}, T_2^{48n^2+24n}, T_3^{32n^2+16n}, T_1^{16n^2+10n}, T_3^{32n^2+24n}, T_6^{16n^2+12n}, T_7^{4n}).
\end{align}
The tuple \eqref{example tuple 22} still has the good ergodicity property; this is once again a consequence of two facts:
\begin{itemize}
    \item for any $j_1, j_2\in[7]$ and nonzero $c_1, c_2\in\Z$, we have $\rho_{2j_1}/c_1 = \rho_{2j_2}/c_2$ if and only if $\rho_{0j_1}/c_1 = \rho_{0j_2}/c_2$;
    \item for any $j_1, j_2\in[7]$, if $\eta_{2j_1} = \eta_{2j_2}$, then $\rho_{2j_1}, \rho_{2j_2}$ have identical leading coefficients.
\end{itemize}
We remark though that to ensure the ergodicity property of \eqref{example tuple 22}, we no longer need the original assumption $\CI(T_5 T_6^{-1}) = \CI(T_5)\cap\CI(T_6)$ because the transformation $T_5$ is not present.

The new tuple \eqref{example tuple 22} has type $w_2 = \sigma_{21}w_1 = (5,1, 0)$, which is still not basic, and so we continue the procedure one more time. The only index left in $\FI_2$ is $6$, and it satisfies the controllability assumption, so we replace $T_6$ this time. We are given an index $i_2\in\FI_1$ (say, $i_2 = 1$), so that
\begin{align*}
    \eta_3 := \tau_{m_2 i_2}\eta_1 = (1, 2, 3, 1, 3, 1, 7).
\end{align*}
Since the leading coefficient 16 of $\rho_{21}$ does not divide the coefficients of the polynomial $\rho_{26}(n) = 16n^2 + 12 n$, and the smallest $\lambda_2\in\N$ for which 16 divides the coefficients of $\lambda_2 \rho_{26}(n) = \lambda_2(16n^2+12n)$ is $\lambda_2 = 4$, we define the new polynomials to be
\begin{align*}
    \rho_{3j}(n) := \begin{cases} \rho_{2j}(4 n + r_2) - \rho_{2j}(r_2),\; &j \neq 6\\
    \frac{16}{16}(\rho_{26}(4 n + r_2) - \rho_{26}(r_2)),\; &j =6
    \end{cases}
\end{align*}
for some $r_2 \in\{0, 1, 2, 3\}$.
Assuming, say, $r_2 = 3$, we get the new tuple
\begin{align*}
    (T_1^{256n^2 + 416n}, T_2^{768n^2+1248n}, T_3^{512n^2+832n}, T_1^{256n^2+424n}, T_3^{512n^2+864n}, T_1^{256n^2+432n}, T_7^{16n}).
\end{align*}
This tuple has the basic type $w_2 = \sigma_{21}w_2 = (6,0,0)$, and so the procedure halts. A similar argument as before shows also that it enjoys the good ergodicity property.

Lastly, we observe that $\eta_3|_{\FI_1} = \eta_3|_{\{1, 2, 3\}}$ is constant, i.e. while performing the type reduction procedure, we did not replace the transformations at indices from $\FI_1$. This is a special case of property (vi) from Proposition \ref{P: induction scheme}, which will play an important role in the proof of Proposition \ref{control of basic types}, a seminorm control argument for tuples of basic type.

\end{example}

\subsection{The role of invariance properties}\label{SS:flipping}
Proposition \ref{P: type reduction} ensures that the lower type tuples to which we pass in the \textit{ping} step of the smoothing argument have the good ergodicity property. But this is not enough. For more complicated tuples, we also need to assume that the functions appearing in the associated average have some invariance properties, otherwise the induction breaks. The example that we present now displays the necessity of this extra information. We sketch how - reducing the original tuple to tuples of shorter length or lower type - we eventually arrive at averages for which we cannot obtain seminorm control unless the functions appearing in the averages satisfy certain invariance properties.  We emphasise that our goal in this example is not to give a complete proof of seminorm control, but rather to point out the necessity of the invariance assumptions. Therefore, we assume without proof when convenient that we have seminorm control over certain tuples of lower type or shorter length.


\begin{example}[The necessity of invariance properties]\label{Ex: n^2, n^2, n^2+n, n^2+n}
Consider the average
\begin{align}\label{1234}
     \E_{n\in[N]}T_1^{n^2}f_1 \cdot T_2^{n^2}f_2 \cdot T_3^{n^2+n}f_3 \cdot T_4^{n^2 + n} f_4.
\end{align}
It has length 4, degree 2, and type $w=(2,2)$, corresponding to the partition $\FI_1 = \{1, 2\}, \FI_2 = \{3, 4\}$. Suppose that \eqref{1234} has the good ergodicity property, i.e.
\begin{align*}
    \CI(T_1 T_2^{-1}) = \CI(T_1)\cap \CI(T_2)\quad \textrm{and}\quad \CI(T_3 T_4\inv) = \CI(T_3)\cap\CI(T_4).
\end{align*}
We illustrate the steps that need to be taken in order to show that this average is controlled by $\nnorm{f_4}_{s, T_4}$ for some $s\in\N$.

By Proposition \ref{strong PET bound}, the average \eqref{1234} is controlled by the seminorm $\nnorm{f_4}_{{\b_1}, \ldots, {\b_{s+1}}}$ for some vectors
\begin{align*}
    \b_1, \ldots, \b_{s+1}\in\{\be_4, \be_4-\be_3, \be_4-\be_2, \be_4-\be_1\}.
\end{align*}
We want to replace the vector $\b_{s+1}$ by (multiple copies of) $\be_4$. Iterating this procedure gives a control of \eqref{1234} by a $T_4$-seminorm of $f_4$.

If $\b_{s+1}=\be_4$, then this follows easily from Lemma \ref{L:seminorm of power}. If $\b_{s+1} = \be_4-\be_3$, then this is the consequence of the good ergodicity property of the average and Lemma \ref{bounding seminorms}. So the only cases to check are when $\b_{s+1}$ equals $\be_4-\be_2$ or $\be_4-\be_1$. Without loss of generality, we assume that $\b_{s+1} = \be_4-\be_2$.

Suppose that the limit of \eqref{1234} is nonzero. Arguing as in the proof of Proposition \ref{Seminorm control of n^2, n^2, n^2+n}, we deduce  that
\begin{align*}
    \liminf_{H\to\infty}\E_{\uh,\uh'\in [H]^{s}}\lim_{N\to\infty}\norm{ \E_{n\in[N]} \, T_1^{n^2} f_{1, \uh, \uh'}\cdot T_2^{n^2} f_{2, \uh, \uh'}\cdot T_3^{n^2+n}f_{3, \uh, \uh'}\cdot T_4^{n^2+n}u_{\uh,\uh'}}_{L^2(\mu)}>0
\end{align*}
for some $T_4 T_2\inv$-invariant functions $u_{\uh, \uh'}$ as well as functions $f_{j, \uh,\uh'} := \Delta_{{\b_1}, \ldots, {\b_{s}}; \uh-\uh'} f_j$ for $j\in[4]$. The invariance property of the functions $u_{\uh, \uh'}$ implies that
\begin{align}\label{1232 2}
    \liminf_{H\to\infty}\E_{\uh,\uh'\in [H]^{s}}\lim_{N\to\infty}\norm{ \E_{n\in[N]} \, T_1^{n^2} f_{1, \uh, \uh'}\cdot T_2^{n^2} f_{2, \uh, \uh'}\cdot T_3^{n^2+n}f_{3, \uh, \uh'}\cdot T_2^{n^2+n}u_{\uh,\uh'}}_{L^2(\mu)}>0.
\end{align}
Each of the averages inside the liminf above is of the form
\begin{align}\label{1232}
     \E_{n\in[N]}T_1^{n^2} g_1\cdot T_2^{n^2} g_2\cdot T_3^{n^2+n}g_3\cdot T_2^{n^2+n}g_4
\end{align}
for 1-bounded functions $g_1, g_2, g_3, g_4\in L^\infty(\mu)$ of which $g_4$ is $T_4 T_2\inv$ invariant. The averages \eqref{1232} are controllable, with $3$ satisfying the controllability condition, and they have type $(3, 1)$, which is lower than the type $(2,2)$ of the original average \eqref{1234}. Assuming inductively that we have the seminorm control of averages \eqref{1232} by a $T_3$-seminorm of $f_3$\footnote{While we only use this particular control, our inductive assumption will guarantee that we control averages \eqref{1232} by a relevant seminorm of other functions, too.}, we can deduce from \eqref{1232 2} (like in the proof of Proposition \ref{Smoothing of n^2, n^2, n^2+n}) that
\begin{align*}
    \nnorm{f_3}_{\b_1, \ldots, \b_s, \be_3^{\times s_1}}>0
\end{align*}
for some $s_1\in\N$, and similarly for other terms. This completes the \textit{ping} step. For the \textit{pong} step, this auxiliary control and Proposition \ref{dual-difference interchange} imply that
\begin{align}\label{1204}
    \liminf_{H\to\infty}\E_{\uh,\uh'\in [H]^{s}}\lim_{N\to\infty}\norm{ \E_{n\in[N]} \, T_1^{n^2} f_{1, \uh, \uh'}\cdot T_2^{n^2} f_{2, \uh, \uh'} \cdot \prod_{j=1}^{2^{s}}\CD_{j}(n^2+n)\cdot T_4^{n^2+n}f_{4, \uh, \uh'}}_{L^2(\mu)}>0
\end{align}
for some $\CD_j\in \FD_{s_1}$.  Each average in \eqref{1204} takes the form
\begin{align}\label{1204 2}
    \E_{n\in[N]} \, T_1^{n^2} g_1\cdot T_2^{n^2} g_2 \cdot \prod_{j=1}^{2^{s}}\CD_{j}(n^2+n)\cdot T_4^{n^2+n} g_4.
\end{align}
Assuming inductively that averages of the form \eqref{1204 2} are controlled by a $T_4$-seminorm of the last term, we get the desired claim $\nnorm{f_4}_{\b_1, \ldots, \b_s, \be_4^{\times s'}}>0$ for some $s'\in\N$ using an argument similar to one in the proof of Proposition \ref{Smoothing of n^2, n^2, n^2+n}.

We have showed how a seminorm control of the original average \eqref{1234} by a $T_4$-seminorm of $f_4$ follows from the seminorm control of the averages \eqref{1232} and \eqref{1204 2}. We have not proved, however, that these auxiliary averages are indeed controlled by Gowers-Host-Kra seminorms, assuming instead that this follows by induction. It turns out that obtaining a seminorm control of the averages \eqref{1232} involves an interesting twist in that the argument makes essential use of the assumption that the function $g_4$ is $T_4 T_2\inv$-invariant. We sketch the steps taken in the seminorm smoothing argument for this average under the extra invariance assumption to show where this invariance property comes up and why it is necessary.

In proving the seminorm control of \eqref{1232}, we first prove that the average is controlled by a $T_3$-seminorm of $g_3$ since $T_3$ is the only transformation with index in $\FI_2$. Arguing as above (using Proposition \ref{strong PET bound} for \eqref{1232}, assuming that the $L^2(\mu)$ limit of \eqref{1232} is nonzero and mimicking the proof of Proposition \ref{Smoothing of n^2, n^2, n^2+n}), we deduce that
\begin{align*}
    \liminf_{H\to\infty}\E_{\uh,\uh'\in [H]^{s}}\lim_{N\to\infty}\norm{ \E_{n\in[N]} \, T_1^{n^2} g_{1, \uh, \uh'}\cdot T_2^{n^2} g_{2, \uh, \uh'}\cdot T_3^{n^2+n}u_{\uh, \uh'}\cdot T_2^{n^2+n}g_{4, \uh,\uh'}}_{L^2(\mu)}>0
\end{align*}
for 1-bounded functions $u_{\uh, \uh'}$ that are all invariant either under $T_3 T_1\inv$ or under $T_3 T_2\inv$. Then we use the relevant invariance property to replace each $T_3^{n^2+n} u_{\uh, \uh'}$ by $T_1^{n^2+n} u_{\uh, \uh'}$ or $T_2^{n^2+n} u_{\uh, \uh'}$. Hence, in the \textit{ping} step of the seminorm smoothing argument for \eqref{1232}, we need to invoke seminorm control of averages of the form
\begin{align}
\nonumber    &\E_{n\in[N]}T_1^{n^2} g_1'\cdot T_2^{n^2} g_2'\cdot T_1^{n^2+n}g_3'\cdot T_2^{n^2+n}g_4'\\
\label{1222}    \textrm{and}\quad & \E_{n\in[N]}T_1^{n^2} g_1'\cdot T_2^{n^2} g_2'\cdot T_2^{n^2+n}g_3'\cdot T_2^{n^2+n}g_4',
\end{align}
where $g'_4$ is $T_4 T_2\inv$-invariant while $g'_3$ is invariant under $T_3 T_1\inv$ and $T_3 T_2\inv$ respectively. Both of them have basic type.

We show that for arbitrary $g_1', g_2', g_3', g_4'$, without the aforementioned invariance assumptions, we would not be able to control the average \eqref{1222} by Gowers-Host-Kra seminorms; specifically, we could not control it by a $T_2$-seminorm of $g'_4$. Conversely, this is achievable if $g'_3, g'_4$ are invariant under $T_3T_2\inv$, $T_4T_2\inv$ respectively. Assuming for simplicity that $g_1'=g_2':= 1$, we have that the second average equals
\begin{align*}
    \E_{n\in[N]}T_1^{n^2} g_1'\cdot T_2^{n^2} g_2'\cdot T_2^{n^2+n}g_3'\cdot T_2^{n^2+n}g_4' = \E_{n\in[N]}T_2^{n^2+n}(g_3'\cdot g_4'),
\end{align*}
and so without any additional assumptions, the average \eqref{1222} is in general not controlled by a $T_2$-seminorm of $g_3'$ or a $T_2$-seminorm of $g'_4$.
However, the invariance assumptions on $g'_3, g'_4$ give us
\begin{align}\label{1222 1}
    \E_{n\in[N]}T_2^{n^2+n}(g_3'\cdot g_4') = \E_{n\in[N]} T_3^{n^2+n}g_3'\cdot T_4^{n^2+n}g_4',
\end{align}
and by Proposition \ref{strong PET bound} these averages  can be controlled by $\nnorm{g_4'}_{\be_4^{\times s}, (\be_4-\be_3)^{\times s}}$ for some $s\in\N$.  Then the assumption $\CI(T_4 T_3\inv)\subseteq\CI(T_4)$ and Lemma \ref{bounding seminorms} give $\nnorm{g_4'}_{\be_4^{\times s}, (\be_4-\be_3)^{\times s}} \leq \nnorm{g'_4}_{2s, T_4}$, and so a $T_4$-seminorm of $g_4'$ does control the average \eqref{1222 1}. Using the invariance property once again, this time alongside with Lemma \ref{composing invariance}, we get $\nnorm{g_4'}_{s', T_4} =\nnorm{g_4'}_{s', T_2}$,
so a $T_2$-seminorm of $g'_4$ controls \eqref{1222 1} and hence \eqref{1222}.

To get control over \eqref{1222} by a $T_2$-seminorm of $g_4'$ without any simplifying assumptions on $g_1', g_2'$, we have to run a more complicated argument. Combining Proposition \ref{strong PET bound}, the ergodic condition on $T_1T_2\inv$ and Lemma \ref{bounding seminorms}, we first obtain control of \eqref{1222} by a $T_1$-seminorm of $g'_1$ and a $T_2$-seminorm of $g'_2$. Assuming that the $L^2(\mu)$ limit of the average \eqref{1222} is positive,  we use this newly established control, decompose $g'_1$ using Proposition \ref{dual decomposition} and apply the pigeonhole principle to show that the average
\begin{align}\label{0222}
    \E_{n\in[N]}\CD(n^2)\cdot T_2^{n^2} g_2'\cdot T_2^{n^2+n}g_3'\cdot T_2^{n^2+n}g_4'
\end{align}
has a nonvanishing limit. The invariance properties of $g'_3$ and $g'_4$ imply that the average \eqref{0222} equals \begin{align}\label{0234 0}
    \E_{n\in[N]}\CD(n^2)\cdot T_2^{n^2} g_2'\cdot T_3^{n^2+n}g_3'\cdot T_4^{n^2+n}g_4',
\end{align}
for which a seminorm control by a $T_4$-seminorm of $g'_4$  follows by inductively invoking seminorm control for averages of length 3. The invariance property of $g'_4$ implies once again that $\nnorm{g_4'}_{s', T_4} =\nnorm{g_4'}_{s', T_2}$ for any $s'\in\N$. It follows that a $T_2$-seminorm of $g'_4$ controls \eqref{0234 0}, and hence also \eqref{0222} and \eqref{1222}.


The invariance properties also come up in the \textit{pong} step of the smoothing argument for \eqref{1232}. In this part, we encounter averages of the form
\begin{align}\label{0232}
    \E_{n\in[N]}\prod_{j=1}^L \CD_{j}(n^2)\cdot T_2^{n^2} g_2''\cdot T_3^{n^2+n}g_3''\cdot T_2^{n^2+n}g_4''.
\end{align}
Moreover, the function $g_4''$ is $T_4 T_2\inv$-invariant because it is essentially a multiplicative derivative of $g_4'$. By a similar reason as before, such averages could not be controlled by Gowers-Host-Kra seminorms for arbitrary $g_2'', g_3'', g_4''$ without the invariance assumption. However, thanks to the invariance assumption, the average \eqref{0232} equals
\begin{align}\label{0234}
        \lim_{N\to\infty} \E_{n\in[N]}\prod_{j=1}^L \CD_{j}(n^2)\cdot T_2^{n^2} g_2''\cdot T_3^{n^2+n}g_3''\cdot T_4^{n^2+n}g_4'',
\end{align}
which is controlled\footnote{We can assume this inductively, or we can prove that a $T_4$-seminorm of $g_4''$ controls \eqref{0234} in essentially the same way as we argued in Proposition \ref{Smoothing of n^2, n^2, n^2+n}.} by $\nnorm{g''_4}_{s', T_4}$ for some $s'\in\N$. Using the invariance property of $g''_4$ again together with Lemma \ref{composing invariance}, we deduce that $\nnorm{g_4''}_{s', T_4} = \nnorm{g_4''}_{s', T_2}$,
and so a $T_2$-seminorm of $g''_4$ does control \eqref{0232}. An argument similar to the one at the end of the proof of Proposition \ref{Smoothing of n^2, n^2, n^2+n} implies that a $T_2$-seminorm of $g_4$ controls \eqref{1232}.
\end{example}
The example above shows that it is crucial to keep track of the invariance properties of the functions appearing in our averages; these invariance properties turn out to be indispensable for applying Proposition \ref{strong PET bound} to averages of basic type, while obtaining seminorm control in the \textit{pong} step of the argument, or - as we see later on - for handling uncontrollable averages. Using the invariance property to replace an average like \eqref{0222} and \eqref{0232} for which we cannot have seminorm control, by an average like \eqref{0234 0} and \eqref{0234} respectively,  which is controlled by Gowers-Host-Kra seminorms exemplifies the \textit{flipping} technique that will be presented in detail in Proposition \ref{P:flipping}.

Recalling how we have performed the \textit{ping} and \textit{pong} steps in Examples \ref{Ex: n^2, n^2, n^2+n}, \ref{Ex: n^2, n^2, 2n^2+n}, and \ref{Ex: n^2, n^2, n^2+n, n^2+n}, we observe that in the \textit{ping} step, we pass from the average
\begin{align}\label{gen average 2}
        \E_{n\in [N]} \,\prod_{j\in[\ell]}T_{\eta_j}^{\rho_j(n)}f_j \cdot \prod_{j\in[L]}\CD_{j}(q_j(n))
\end{align}
to averages
\begin{align}\label{ping average}
  \E_{n\in[N]}\,\prod_{\substack{j\in[\ell],\\ j\neq m}}T_{\eta_j}^{\rho'_j(n)}f_{j, \uh, \uh'} \cdot T_{\eta_i}^{\rho'_m(n)}u_{\uh, \uh'} \cdot \prod_{j\in[L']}\CD'_{j}(q'_j(n)),
\end{align}
where $f_{j, \uh, \uh'}:= \Delta_{\b_1, \ldots, \b_s; \uh-\uh'}$ for some vectors $\b_1, \ldots, \b_s\in\Z^\ell$. In particular, the functions $f_{j, \uh, \uh'}$ are invariant under whatever transformations the functions $f_j$ are invariant. Moreover, the functions $u_{\uh, \uh'}$ are invariant under $T_{\eta_m}^{a_m}T_{\eta_i}^{-a_i}$, where $a_m$ and $a_i$ are the leading coefficients of $\rho_m$ and $\rho_i$, but they also retain whatever invariance property $f_m$ has. Thus, by passing from \eqref{gen average 2} to \eqref{ping average}, we do not lose any invariance properties of the original functions, but rather gain new ones.

Similarly, in the \textit{pong} step, we pass to averages
\begin{align*}
  \E_{n\in[N]}\,\prod_{\substack{j\in[\ell],\\ j\neq i}}T_{\eta_j}^{\rho_j(n)}f_{j, \uh, \uh'} \cdot \prod_{j\in[L'']}\CD''_{j}(q_j''(n)),
\end{align*}
and the functions $f_{j, \uh, \uh'}$ retain whatever invariance properties $f_j$ have.

Thus, the functions $(f_1, \ldots, f_\ell)$ get replaced by
\begin{align}\label{E: function tuple 1}
    (f_{1, \uh, \uh'}, \ldots, f_{m-1, \uh, \uh'}, u_{\uh, \uh'}, f_{m+1, \uh, \uh'}, \ldots, f_{\ell, \uh, \uh'})
\end{align}
in the \textit{ping} step and
\begin{align}\label{E: function tuple 2}
    (f_{1, \uh, \uh'}, \ldots, f_{i-1, \uh, \uh'}, 1, f_{i+1, \uh, \uh'}, \ldots, f_{\ell, \uh, \uh'})
\end{align}
in the \textit{pong} step. We now formalise the idea that these new families of functions retain the original invariance properties and gain new ones.

\begin{definition}[Good invariance property]
    Let $\gamma\in\N$. We say that the tuple of functions $(f_1, \ldots, f_\ell)$ has  the \emph{$\gamma$-invariance property along $\eta$ with respect to  polynomials $p_1, \ldots, p_\ell$} with leading coefficients $a_1, \ldots, a_\ell$,  if for every $j\in[\ell]$, the function $f_j$ is invariant under $\brac{T_{\eta_j}^{a_{\eta_j}}T_j^{-a_j}}^{\gamma}$.  Let $I$ be a (possibly infinite) indexing set.
    We say that a collection $(f_{i1}, \ldots, f_{i\ell})_{i\in I}$ has the \emph{good invariance property along $\eta$ with respect to  polynomials $p_1, \ldots, p_\ell$} if there exists $\gamma\in\N$ such that $(f_{i1}, \ldots, f_{i\ell})_{i\in I}$ has the $\gamma$-invariance property along $\eta$ with respect to $p_1, \ldots, p_\ell$ for every $i\in I$.
\end{definition}

If $\eta=(1, \ldots, \ell)$ is the identity tuple, there is nothing to check and any collection of functions has the $1$-invariance property with respect to any polynomial family. The property only becomes nontrivial when $\eta$ is not the identity tuple.

In our arguments, we will ensure that the functions $f_j$ in the average
\begin{align*}
    \E_{n\in[N]}\prod_{j\in[\ell]}T_{\eta_{j}}^{\rho_j(n)} f_j\cdot\prod_{j\in[L]}\CD_{j}(q_j(n))
\end{align*}
 obtained by a sequence of reductions from the original average $\E_{n\in[N]}\prod_{j\in[\ell]}T_{j}^{p_j(n)}$ have the good invariance property with respect to the original polynomials $p_1, \ldots, p_\ell$, i.e. there exists $\gamma\in\N$ such that for every $j\in[\ell]$, the function $f_j$ is invariant under $\brac{T_{\eta_j}^{a_{\eta_{j}}}T_j^{-a_{j}}}^{\gamma}$,  where $a_j$ is the leading coefficient of the polynomial $p_j$. 
The need to keep track of the invariance property with respect to the original polynomials is explained in Example \ref{Ex: 6 tuple} below. Before we state this example, however, we prove that the invariance properties get preserved when passing from the tuple of functions $(f_1, \ldots, f_\ell)$ to the tuples \eqref{E: function tuple 1} and \eqref{E: function tuple 2}.



\begin{proposition}[Propagation of invariance properties]\label{propagation of invariance}
Let $\ell, s\in\N$, $\eta\in[\ell]^\ell$ be an indexing tuple, $\eta' := \tau_{mi}\eta$ for distinct $m,i\in[\ell]$ be another indexing tuple, $p_1, \ldots, p_\ell\in\Z[n]$ be polynomials with leading coefficients $a_1, \ldots, a_\ell$, $(X, \CX, \mu, T_1, \ldots, T_\ell)$ be a system, and $\b_1, \ldots,\b_s\in\Z^\ell$ be vectors. Suppose that for some $\gamma\in\N$, the functions $(f_1, \ldots, f_\ell)$ have the $\gamma$-invariance property along $\eta$ with respect to the polynomials $p_1, \ldots, p_\ell$. Consider the functions $(f'_{1, \uh}, \ldots, f'_{\ell, \uh})_{\uh \in \Z^s}$, where $f'_{j, \uh} := \Delta_{\b_1, \ldots, \b_s; \uh}f_j$ for $j\neq m$, and $f_{m, \uh}'$ is a function invariant under both
$S_1 :=\brac{T_{\eta_m}^{a_{\eta_m}}T_m^{-a_m}}^{\gamma_1}$ and $S_2 := \brac{T_{\eta_i}^{a_{\eta_i}}T_{\eta_m}^{-a_{\eta_m}}}^{\gamma_2}$ for some $\gamma_1, \gamma_2\in\N$ independent of $\uh$.
Then $(f'_{1, \uh}, \ldots, f'_{\ell, \uh})_{\uh \in \Z^s}$ has the good invariance property along $\eta'$ with respect to $p_1, \ldots, p_\ell$.
\end{proposition}

\begin{proof}[Proof of Proposition \ref{propagation of invariance}]
For $j\neq m$, the functions $f_{j}$ are invariant under $\brac{T_{\eta_j}^{a_{\eta_j}}T_j^{-a_j}}^\gamma$ for some nonzero $\gamma\in\Z$ independent of $\uh\in\Z^s$, and so are their translations $$\CC^{|\ueps|}T^{\eps_1 h_1 \b_1 + \cdots + \eps_{s} h_s \b_s}f_j.$$ The identity $\eta'_j = \eta_j$, which holds for $j\neq m$,  and the fact that $f'_{j, \uh}$ is a product of $\brac{T_{\eta_j}^{a_{\eta_j}}T_j^{-a_j}}^\gamma$-invariant functions, implies that $f'_{j, \uh}$ is itself invariant under $\brac{T_{\eta'_j}^{a_{\eta'_j}}T_j^{-a_j}}^\gamma$. For $j=m$, the functions $f'_{j, \uh}$ are invariant under
\begin{align*}
    S_1^{\gamma_2}S_2^{\gamma_1} = \brac{T_{\eta_m}^{a_{\eta_m}}T_m^{-a_m}T_{\eta_i}^{a_{\eta_i}}T_{\eta_m}^{-a_{\eta_m}}}^{\gamma_1\gamma_2} = \brac{T_{\eta_i}^{a_{\eta_i}}T_m^{-a_m}}^{\gamma_1\gamma_2}=\brac{T_{\eta_m'}^{a_{\eta_m'}}T_m^{-a_m}}^{\gamma_1\gamma_2}
\end{align*}
by noting $\eta'_m = \eta_i$ and combining the two invariance properties that these functions enjoy. Letting $\gamma' := \textrm{lcm}(\gamma, \gamma_1\gamma_2)$, it follows that for every $\uh\in\Z^s$, the collection $(f'_{1, \uh}, \ldots, f'_{\ell, \uh})$ has the $\gamma'$-invariance property along $\eta'$ with respect to $p_1, \ldots, p_\ell$.
\end{proof}

To get desirable seminorm control over the intermediate tuples encountered in Proposition \ref{P: type reduction}, it is not sufficient to keep track of the most immediate invariance properties. This is illustrated by the example below.

\begin{example}[The necessity of composed invariance properties]\label{Ex: 6 tuple}
Consider the average
\begin{align}\label{123456}
    \E_{n\in[N]}T_1^{n^2}f_1 \cdot T_2^{n^2}f_2\cdot T_3^{n^2+n}f_3\cdot T_4^{n^2+n}f_4 \cdot T_5^{n^2+2n}f_5\cdot T_6^{n^2+2n}f_6.
\end{align}
It has length $6$, degree 2 and type $(2,2,2)$ corresponding to the partition
\begin{align*}
    \FI_1 =\{1, 2\},\quad \FI_2 = \{3, 4\}, \quad \FI_3 = \{5, 6\}.
\end{align*}
We assume that it has the good ergodicity property, i.e.
\begin{align*}
    \CI(T_1 T_2\inv) = \CI(T_1)\cap\CI(T_2),\quad \CI(T_3 T_4\inv) = \CI(T_3)\cap \CI(T_4), \quad  \CI(T_5 T_6\inv) =\CI(T_5)\cap \CI(T_6).
\end{align*}
Suppose we want to perform the seminorm smoothing argument to obtain a control of the associated average by the $T_6$-seminorm of $f_6$. We iteratively pass to averages of lower type as in Proposition \ref{P: type reduction} (all of which turn out to be controllable), and we show that to get seminorm control for the average of basic type at which we arrive, we need to keep track of not just the latest invariance properties that the functions in the intermediate averages enjoy, but of \textit{all} the invariance properties that the functions in earlier intermediate averages enjoyed.

\smallskip
\textbf{Step 1: Reducing to an average of basic type.}
\smallskip

In the \textit{ping} part of the seminorm smoothing argument for \eqref{123456}, we replace $T_6$ in the original average \eqref{123456} by some $T_i$ with $i\in[4]=\FI_1\cup\FI_2$, arriving at, say, averages
\begin{align}\label{123454}
    \E_{n\in[N]}T_1^{n^2}f_{11} \cdot T_2^{n^2}f_{12}\cdot T_3^{n^2+n}f_{13}\cdot T_4^{n^2+n}f_{14} \cdot T_5^{n^2+2n}f_{15}\cdot T_4^{n^2+2n}f_{16}.
\end{align}
The new tuple \eqref{123454} has type $(2,3,1)$ and is controllable, with the index 5 satisfying the controllability condition. The functions inside take the form
\begin{align*}
    f_{1j} := \begin{cases}
    \Delta_{\b_{11}, \ldots, \b_{1s_1}; \uh-\uh'}f_{j},\; & j \neq 6\\
    u_{1,\uh, \uh'},\; &j=6,
    \end{cases}
\end{align*}
where the functions $u_{1, \uh, \uh'}$ are $T_6 T_4^{-1}$-invariant.

To obtain seminorm control of the average \eqref{123454}, we need to perform the seminorm smoothing argument for this tuple. We aim to control it first by a $T_5$-seminorm of $f_{15}$ since $5$ satisfies the controllability condition and is the only index left in $\FI_3$. As guided by Proposition \ref{P: type reduction}, in the \textit{ping} step of the smoothing argument, we replace $T_5$ in \eqref{123454} by some $T_i$ with $i\in[4]$. When $i=1$, for instance, we end up with averages
\begin{align}\label{123414}
    \E_{n\in[N]}T_1^{n^2}f_{21} \cdot T_2^{n^2}f_{22}\cdot T_3^{n^2+n}f_{23}\cdot T_4^{n^2+n}f_{24} \cdot T_1^{n^2+2n}f_{25}\cdot T_4^{n^2+2n}f_{26}
\end{align}
of type $(3, 3, 0)$. The functions in \eqref{123414} take the form
\begin{align*}
    f_{2j} := \begin{cases}
    \Delta_{\b_{21}, \ldots, \b_{2s_2}; \uh-\uh'}f_{1j},\; & j \neq 5\\
    u_{2,\uh, \uh'},\; &j=5,
    \end{cases}
\end{align*}
where $u_{2, \uh, \uh'}$ are $T_5 T_1\inv$-invariant. We note by Proposition \ref{propagation of invariance} that the functions $f_{26}$ retain the $T_6 T_4\inv$-invariance of $f_{16}$.

The indices $3, 4, 6$ in the average \eqref{123414} all satisfy the controllability condition, so if we want to get seminorm control of this average, we should first control it by a relevant seminorm of one of the functions $f_{23}, f_{24}, f_{26}$. Suppose that we choose to obtain a seminorm control of the tuple \eqref{123414} with respect to $f_{26}$ first.
Then we would replace $T_4$ at index 6 by $T_i$ for any $i\in[2] = \FI_1$ in the \textit{ping} step, getting, say, averages
\begin{align}\label{123411}
    \E_{n\in[N]}T_1^{n^2}f_{31} \cdot T_2^{n^2}f_{32}\cdot T_3^{n^2+n}f_{33}\cdot T_4^{n^2+n}f_{34} \cdot T_1^{n^2+2n}f_{35}\cdot T_1^{n^2+2n}f_{36}
\end{align}
of type $(4, 2, 0)$. The functions in \eqref{123411} take the form
\begin{align*}
    f_{3j} := \begin{cases}
    \Delta_{\b_{31}, \ldots, \b_{3s_3}; \uh-\uh'}f_{2j},\; & j \neq 6\\
    u_{3,\uh, \uh'},\; &j=6.
    \end{cases}
\end{align*}
The function $f_{35}$, being a multiplicative derivative of a $T_5 T_1\inv$-invariant function, is itself invariant under $T_5 T_1\inv$. The function $f_{36}$ is invariant not only under $T_4T_1\inv$, but also under $T_6 T_4\inv$ thanks to Proposition \ref{propagation of invariance}. It is crucial that  $f_{36}$ retains the $T_6 T_4\inv$-invariance of $f_{26}$, and we shall return to this point shortly.

The average \eqref{123411} is controllable, and so to arrive at an average of a basic type, we need to perform this procedure two more times. Both the indices 3 and 4 satisfy the controllability condition, so we want to get seminorm control in terms of one of $f_{33}, f_{34}$ - say, we choose $f_{34}$. To obtain control of the tuple \eqref{123411} by a $T_4$-seminorm of $f_{34}$, we replace $T_4$ by $T_i$ for any $i\in[2] = \FI_1$ in the \textit{ping} step, arriving at, say, averages
\begin{align}\label{123211}
    \E_{n\in[N]}T_1^{n^2}f_{41} \cdot T_2^{n^2}f_{42}\cdot T_3^{n^2+n}f_{43}\cdot T_2^{n^2+n}f_{44} \cdot T_1^{n^2+2n}f_{45}\cdot T_1^{n^2+2n}f_{46}
\end{align}
of type $(5, 1, 0)$ if $i=2$. The functions in \eqref{123211} take the form
\begin{align*}
    f_{4j} := \begin{cases}
    \Delta_{\b_{41}, \ldots, \b_{4s_4}; \uh-\uh'}f_{3j},\; & j \neq 4\\
    u_{4,\uh, \uh'},\; &j=4.
    \end{cases}
\end{align*}
The functions $f_{45}$ and $f_{46}$ retain respectively the invariance under $T_5 T_1\inv$ of $f_{35}$ and the invariance under $T_4T_1\inv$ and $T_6 T_4\inv$ 
of $f_{36}$. Moreover, the function $f_{44}$ is $T_4 T_1\inv$-invariant.

Finally, the only index in the average \eqref{123211} satisfying the controllability condition is $3$, so if we want to obtain control of the tuple \eqref{123211}, we first want to get this in terms of a $T_3$-seminorm of $f_{43}$. Applying Proposition \ref{P: type reduction}, we end up replacing $T_3$ by, say, $T_2$, getting averages
\begin{align}\label{122211}
    \E_{n\in[N]}T_1^{n^2}f_{51} \cdot T_2^{n^2}f_{52}\cdot T_2^{n^2+n}f_{53}\cdot T_2^{n^2+n}f_{54} \cdot T_1^{n^2+2n}f_{55}\cdot T_1^{n^2+2n}f_{56}.
\end{align}
 The functions in \eqref{122211} take the form
\begin{align*}
    f_{5j} := \begin{cases}
    \Delta_{\b_{51}, \ldots, \b_{5s_5}; \uh-\uh'}f_{4j},\; & j \neq 3\\
    u_{5,\uh, \uh'},\; &j=3,
    \end{cases}
\end{align*}
in particular, the functions $f_{54}, f_{55}, f_{56}$ retain the invariance properties of $f_{44}, f_{45}, f_{46}$ and $f_{53}$ is $T_3T_2\inv$-invariant.

\smallskip
\textbf{Step 2: Handling an average of basic type.}
\smallskip

The average \eqref{122211} has basic type $(6, 0, 0)$, and so we want to control it by appropriate seminorms using Proposition \ref{strong PET bound}.  We show that without the assumption that $f_{56}$ is invariant under both $T_4T_1\inv$ \textit{and} $T_6T_4\inv$, we cannot control this average by a $T_1$-seminorm of $f_{56}$, and conversely - that this goal can be achieved with both of these assumptions.

We first note that Proposition \ref{strong PET bound}, the ergodicity condition $\CI(T_1 T_2\inv) = \CI(T_1)\cap\CI(T_2)$ and Lemma \ref{bounding seminorms} allow us to control it by a $T_1$-seminorm of $f_{51}$ and by a $T_2$-seminorm of $f_{52}$\footnote{At the same time, Proposition \ref{strong PET bound}, our ergodicity assumptions, and Lemma \ref{bounding seminorms} alone cannot be used to control \eqref{122211} by $T_2$-seminorms of $f_{53}$ or $f_{54}$ or by $T_1$-seminorms of $f_{55}$ or $f_{56}$. Without additional information about the invariance properties of the functions, Proposition \ref{strong PET bound}, our ergodicity assumptions and Lemma \ref{bounding seminorms} could only give control of \eqref{122211} by a $T_2$-seminorms of $f_{53}f_{54}$ and by a $T_1$-seminorm of $f_{55}f_{56}$, which is insufficient for our purposes.}. Suppose that the $L^2(\mu)$ limit of \eqref{122211} is positive. Decomposing $f_{51}$ using Proposition \ref{dual decomposition} and then applying the pigeonhole principle, we deduce the existence of $\CD\in\FD$ such that
\begin{align}\label{122211 2}
    \lim_{N\to\infty}\norm{\E_{n\in[N]}\CD(n^2)\cdot T_2^{n^2}f_{52}\cdot T_2^{n^2+n}f_{53}\cdot T_2^{n^2+n}f_{54} \cdot T_1^{n^2+2n}f_{55}\cdot T_1^{n^2+2n}f_{56}}_{L^2(\mu)}>0.
\end{align}

We now want to obtain a seminorm control of the average in \eqref{122211 2} by inductively invoking seminorm control for some average of length 5. To this end, we attempt to proceed like in Example \ref{Ex: n^2, n^2, n^2+n, n^2+n}. That is, we use the invariance of $f_{53}, f_{54}, f_{55}, f_{56}$ under $T_3T_2\inv, T_4T_2\inv, T_5T_1\inv, T_4 T_1\inv$ respectively to conclude that the average in \eqref{122211 2} equals
\begin{align}\label{122211 3}
    \lim_{N\to\infty}\E_{n\in[N]}\CD(n^2)\cdot T_2^{n^2}f_{52}\cdot T_3^{n^2+n}f_{53}\cdot T_4^{n^2+n}f_{54} \cdot T_5^{n^2+2n}f_{55}\cdot T_4^{n^2+2n}f_{56}.
\end{align}

However, without any extra information, we could not control the average \eqref{122211 3} using a $T_2$-seminorm of $f_{55}$ or $T_1$-seminorm of $f_{56}$. Suppose for simplicity that $\CD$ is a constant sequence and $f_{52}=f_{53}=f_{54} = 1$. Then \eqref{122211 3} reduces to
\begin{align}\label{122211 4}
    \lim_{N\to\infty}\E_{n\in[N]}T_5^{n^2+2n}f_{55}\cdot T_4^{n^2+2n}f_{56}.
\end{align}
We know nothing about the composition $T_5 T_4\inv$, and so without additional input, we cannot control \eqref{122211 4} by a Gowers-Host-Kra seminorm.

This is the moment when we have to use the additional $T_6T_4\inv$-invariance of $f_{56}$. Since $T_4 f_{56} = T_6 f_{56}$, we can replace $T_4$ in \eqref{122211 4} by $T_6$. Then Proposition \ref{strong PET bound} gives us control over \eqref{122211 4} by $\nnorm{f_{56}}_{\be_6^{\times s}, (\be_6-\be_5)^{\times s}}$ for some $s\in\N$. The ergodicity condition on $T_6 T_5\inv$ and the $T_6 T_1\inv$-invariance of $f_{56}$ then give  $\nnorm{f_{56}}_{\be_6^{\times s}, (\be_6-\be_5)^{\times s}}\leq \nnorm{f_{56}}_{2s, T_6}= \nnorm{f_{56}}_{2s, T_1}$, and so this latter seminorm controls \eqref{122211 2}, and hence also \eqref{122211}.

If we want to control \eqref{122211 2} by a $T_1$-seminorm of $f_{56}$ without the simplifying assumptions on $\CD$ and $f_{52}, f_{53}, f_{54}$, we proceed similarly\footnote{There is no special reason why we would want to control \eqref{122211 2} by a seminorm of $f_{56}$ instead of other functions. We just aim to illustrate that  using the extra $T_6T_4\inv$-invariance of $f_{56}$, this can be done.}. Using the $T_6T_4\inv$-invariance of $f_{56}$, we rewrite \eqref{122211 3} as
\begin{align}\label{122211 5}
    \lim_{N\to\infty}\E_{n\in[N]}\CD(n^2)\cdot T_2^{n^2}f_{52}\cdot T_3^{n^2+n}f_{53}\cdot T_4^{n^2+n}f_{54} \cdot T_5^{n^2+2n}f_{55}\cdot T_6^{n^2+2n}f_{56}.
\end{align}
Then we inductively apply the fact that we have seminorm control for averages of length 5 of the form \eqref{122211 5} to control this average, and hence also \eqref{122211}, by a $T_6$-seminorm of $f_{56}$. Subsequently, the $T_6 T_1\inv$-invariance of $f_{56}$ (resulting from its $T_6 T_4\inv$- and $T_4T_1\inv$-invariance) and Lemma \ref{composing invariance} give control of \eqref{122211} by a $T_1$-seminorm of $f_{56}$.

We note that the argument above would not work if we only used the ``new'' property of $f_{56}$ of being invariant under $T_4 T_1\inv$ rather than the ``combined'' property of being invariant under $T_6 T_1\inv$. The important point that this example shows is that it is not enough to keep track of the new invariance properties that we obtain at each stage of the \textit{ping} argument and forget the old ones. Rather, we need to keep track of the invariance property with respect to a composition of the original and the most recent transformations, which in our example are $T_6$ and $T_1$.
\end{example}

Examples \ref{Ex: n^2, n^2, n^2+n, n^2+n} and \ref{Ex: 6 tuple} reveal that to obtain seminorm control of a controllable average of basic type and in the \textit{pong} step of the seminorm smoothing argument, we need to substitute the transformations $T_j$ for $T_{\eta_j}$ to arrive at an average for which we have seminorm control. The following proposition allows us to do just that.
\begin{proposition}[Flipping]\label{P:flipping}
Let $\gamma, \ell, L\in\N$, $(X, \CX, \mu, T_1, \ldots, T_\ell)$ be a system, $\eta\in[\ell]^\ell$ be an indexing tuple, $A\subset[\ell]$ be a subset of indices and $p_1, \ldots, p_\ell, \rho_1, \ldots, \rho_\ell, q_1, \ldots, q_L\in\Z[n]$ be polynomials. Suppose that
\begin{enumerate}
    \item the tuple $\brac{T_{\eta_j}^{\rho_j(n)}}_{j\in[\ell]}$ is a descendant of $\brac{T_j^{p_j(n)}}_{j\in[\ell]}$;
    \item $f_1, \ldots, f_{\ell}\in L^\infty(\mu)$ are 1-bounded functions having the $\gamma$-invariance property along $\eta$ with respect to $p_1, \ldots, p_\ell$.
\end{enumerate}

Then there exist 1-bounded functions $f_1', \ldots, f'_{\ell}$, polynomials $\rho'_1, \ldots, \rho'_{\ell}, q_1', \ldots, q_L'\in\Z[n]$, and an indexing tuple $\eta'$ with the following properties:
\begin{enumerate}
    \item the tuple $\eta'$ takes the form
    \begin{align*}
        \eta'_j = \begin{cases} j,\; &j\in A\\ \eta_j,\; &j\notin A;
        \end{cases}
    \end{align*}
    \item  $\brac{T_{\eta'_j}^{\rho'_j(n)}}_{j\in[\ell]}$ is a descendant of $\brac{T_j^{p_j(n)}}_{j\in[\ell]}$;
    \item $f'_1, \ldots, f'_{\ell}\in L^\infty(\mu)$ are 1-bounded, have the $\gamma$-invariance property along $\eta'$ with respect to $p_1, \ldots, p_\ell$, satisfy the bound $\nnorm{f'_j}_{s, T_{\eta'_j}}\ll_\gamma \nnorm{f_j}_{s, T_{\eta_{j}}}$ for every $s\geq 2$ and $j\in[\ell]$, and moreover $f'_j = 1$ whenever $f_j = 1$; 
    \item we have the inequality
    \begin{align}\label{E:flipping}
    &\lim_{N\to\infty}\norm{\E_{n\in [N]} \,\prod_{j\in[\ell]}T_{{\eta_{j}}}^{\rho_{j}(n)}f_{j} \cdot \prod_{j\in[L]}\CD_{j}(q_{j}(n))}_{L^2(\mu)}\leq \\
    \nonumber
    & \lim_{N\to\infty}\norm{\E_{n\in [N]} \,\prod_{j\in[\ell]}T_{{\eta'_j}}^{\rho'_j(n)}f'_j \cdot\prod_{j\in[L]}\CD_{j}(q'_j(n))}_{L^2(\mu)}.
\end{align}
\end{enumerate}
\end{proposition}
We note that if the leading coefficients of $p_1, \ldots, p_\ell$ are all 1, and the good invariance property takes the form of $f_j$ being invariant under $T_{\eta_j}T_j^{-1}$ for all $j\in[\ell]$, then Proposition \ref{P:flipping} is straightforward, and in fact for every $N\in\N$, we have
\begin{align*}
    \E_{n\in [N]} \,\prod_{j\in[\ell]}T_{{\eta_{j}}}^{\rho_{j}(n)}f_{j} \prod_{j\in[L]}\CD_{j}(q_{j}(n))=\E_{n\in [N]} \,\prod_{j\in[\ell]}T_{\eta'_j}^{\rho_{j}(n)}f_{j} \prod_{j\in[L]}\CD_{j}(q_{j}(n))
\end{align*}
and $\nnorm{f_j}_{s, T_{\eta_j}} = \nnorm{f_j}_{s, T_{\eta'_j}}$ for all $j\in[\ell], s\in\N$.
The need for the more complicated statement of Proposition \ref{P:flipping} comes from tedious but uninspiring technicalities that appear when the polynomials $p_1, \ldots, p_\ell$ have  leading coefficients distinct from 1.

We have used the flipping technique twice in Example \ref{Ex: n^2, n^2, n^2+n, n^2+n}: in \eqref{1222 1} and when passing from \eqref{0232} to \eqref{0234} in order to obtain a seminorm control on the former using the seminorm control on the latter. We also used it in Example \ref{Ex: 6 tuple} to get seminorm control of \eqref{122211}. We will also use it shortly to handle uncontrollable tuples.


\begin{proof}[Proof of Proposition \ref{P:flipping}]
For each $j\in[\ell]$, let $a_j$ be the leading coefficient of $p_j$ and  $\gamma_j\in\N$ be the smallest natural number such that $f_j$ is invariant under $\brac{T_{\eta_j}^{a_{\eta_j}}T_j^{-a_j}}^{\gamma_j}$ (in particular, $\gamma_j = 1$ if $\eta_j = j$). Let $\gamma\in\N$ be the smallest natural number such that $a_{\eta_j} \gamma_j$ divides the coefficients of $\gamma \rho_j$ for every $j\in A$.
We then define
\begin{align*}
    f'_j := T_{\eta_j}^{\rho_j(r)}f_j,\quad q'_j(n):=q_j(\gamma n + r),\quad \rho'_{j}(n) := \frac{a_{\eta'_j}}{a_{\eta_{j}}}\brac{\rho_{j}(\gamma n + r) - \rho_{j}(r)}
\end{align*}
for some $r\in\{0, \ldots, \lambda-1\}$ to be chosen later, and we observe that $\rho'_j\in\Z[n]$ for every $j\in [\ell]$. By definition of $\eta'$ and the $\gamma$-invariance of $f_1, \ldots, f_\ell$ along $\eta$, the functions $f_1', \ldots, f_\ell'$ are $\gamma$-invariant along $\eta'$. Moreover, if $f_j =1$, then so is $f'_j$. Lastly, they satisfy the bound $\nnorm{f'_j}_{s, T_{\eta'_j}}\ll_\gamma \nnorm{f_j}_{s, T_{\eta_{j}}}$ for every $s\geq 2$ and $j\in[\ell]$; this is trivial for $j\notin A$, and if $j\in A$, then
\begin{align*}
    \nnorm{f'_j}_{s, T_{\eta'_j}} = \nnorm{f'_j}_{s, T_{j}}= \nnorm{f_j}_{s, T_{j}} \leq \nnorm{f_j}_{s, T_{j}^{a_{j}\gamma}} = \nnorm{f_j}_{s, T_{\eta_{j}}^{a_{\eta_{j}}\gamma}}\ll_{a_{\eta_j} \gamma}\nnorm{f_j}_{s, T_{\eta_j}},
\end{align*}
where we use the fact that $f_j$ is a composition of $f'_j$ with respect to a measure preserving transformation, the invariance property of $f_j$, Lemma \ref{composing invariance}, and both directions of Lemma \ref{L:seminorm of power}.

We move on to prove the inequality \eqref{E:flipping}, For $j\notin A$, where $\eta'_j = \eta_j$, we simply have $T_{\eta_j}^{\rho_j(\gamma n+r)}f_j = T_{\eta'_j}^{\rho'_j(n)}f_j'$. For $j\in A$, the invariance property of $f_j$ gives us the identity
\begin{align*}
    T_{\eta_j}^{\rho_j(\gamma n+r)}f_j = \brac{T_{\eta_j}^{a_{\eta_j}\gamma_j}}^{\frac{\rho_j(\gamma n+r)-\rho_j(r)}{a_{\eta_j}\gamma_j}} T_{\eta_j}^{\rho_j(r)} f_j
    =  T_{\eta'_j}^{\rho'_j(n)}f_j'.
\end{align*}
Splitting $\N$ into $(\gamma\cdot\N + r)_{r\in\{0, \ldots, \gamma-1\}}$, we deduce from the pigeonhole principle that there exists $r\in\{0, \ldots, \gamma-1\}$ for which \eqref{E:flipping} holds.

From the construction of the polynomials $\rho_1', \ldots, \rho'_\ell$, the assumption that $\brac{T_{\eta_j}^{\rho_j(n)}}_{j\in[\ell]}$ is a descendant of $\brac{T_j^{p_j(n)}}_{j\in[\ell]}$ and Lemma \ref{L: descendancy transitive}, it follows that $\brac{T_{\eta'_j}^{\rho'_j(n)}}_{j\in[\ell]}$ is also a descendant of $\brac{T_j^{p_j(n)}}_{j\in[\ell]}$.
\end{proof}

\subsection{Handling uncontrollable tuples}\label{SS:uncontrollable}
We have explained in the previous sections that if a tuple $\brac{T_{\eta_j}^{\rho_j(n)}}_{j\in[\ell]}$ is controllable, then we control it by a Gowers-Host-Kra seminorm using a seminorm smoothing argument. If the tuple is uncontrollable, however, we use the following variant of the flipping technique from Proposition \ref{P:flipping} to bound the $L^2(\mu)$ norm of the associated average by an $L^2(\mu)$ norm of a controllable average.
\begin{corollary}[Flipping uncontrollable tuples]\label{C: Flipping uncontrollable tuples}
Let $\gamma, \ell, L\in\N$, $(X, \CX, \mu, T_1, \ldots, T_\ell)$ be a system, $\eta\in[\ell]^\ell$ be an indexing tuple, and $p_1, \ldots, p_\ell, \rho_1, \ldots, \rho_\ell, q_1, \ldots, q_L\in\Z[n]$ be polynomials. Suppose that
\begin{enumerate}
    \item the tuple $\brac{T_{\eta_j}^{\rho_j(n)}}_{j\in[\ell]}$ is uncontrollable of type $w$ with the last nonzero index $t_w$, and it is a descendant of $\brac{T_j^{p_j(n)}}_{j\in[\ell]}$;
    \item $f_1, \ldots, f_{\ell}\in L^\infty(\mu)$ are 1-bounded functions having the $\gamma$-invariance property along $\eta$ with respect to $p_1, \ldots, p_\ell$.
\end{enumerate}

Then there exist 1-bounded functions $f_1', \ldots, f'_{\ell}$, polynomials $\rho'_1, \ldots, \rho'_{\ell}, q_1', \ldots, q_L'\in\Z[n]$, and an indexing tuple $\eta'$ with the following properties:
\begin{enumerate}
    \item the tuple $\eta'$ takes the form
    \begin{align*}
        \eta'_j = \begin{cases} j,\; &\eta_j\in \FI_{t_w}\\ \eta_j,\; &\eta_j\notin \FI_{t_w};
        \end{cases}
    \end{align*}
    \item  $\brac{T_{\eta'_j}^{\rho'_j(n)}}_{j\in[\ell]}$ is a descendant of $\brac{T_j^{p_j(n)}}_{j\in[\ell]}$;
    \item $f'_1, \ldots, f'_{\ell}\in L^\infty(\mu)$ are 1-bounded, have $\gamma$-invariance property along $\eta'$ with respect to $p_1, \ldots, p_\ell$, satisfy the bound $\nnorm{f'_j}_{s, T_{\eta'_j}}\ll_\gamma \nnorm{f_j}_{s, T_{\eta_{j}}}$ for every $s\geq 2$ and $j\in[\ell]$, and moreover $f'_j = 1$ whenever $f_j = 1$;
    \item we have the inequality
    \begin{align*}
    &\lim_{N\to\infty}\norm{\E_{n\in [N]} \,\prod_{j\in[\ell]}T_{{\eta_{j}}}^{\rho_{j}(n)}f_{j} \cdot \prod_{j\in[L]}\CD_{j}(q_{j}(n))}_{L^2(\mu)}\leq \\
    \nonumber
    & \lim_{N\to\infty}\norm{\E_{n\in [N]} \,\prod_{j\in[\ell]}T_{{\eta'_j}}^{\rho'_j(n)}f'_j \cdot\prod_{j\in[L]}\CD_{j}(q'_j(n))}_{L^2(\mu)}.
\end{align*}
\end{enumerate}
\end{corollary}
Corollary \ref{C: Flipping uncontrollable tuples} follows from Proposition \ref{P:flipping} by taking $A=\{j\in[\ell]: \eta_j\in \FI_{t_w}\}$.

We emphasise that Corollary \ref{C: Flipping uncontrollable tuples} by itself does not guarantee that the tuple $\brac{T_{\eta'_j}^{\rho'_j(n)}}_{j\in[\ell]}$ has a lower type than the tuple $\brac{T_{\eta_j}^{\rho_j(n)}}_{j\in[\ell]}$. However, this will be the case when we apply it to all the tuples $\brac{T_{\eta_j}^{\rho_j(n)}}_{j\in[\ell]}$ that appear in our inductive procedure. The crucial ingredient in achieving this type reduction will be the property (iv) in Proposition \ref{P: induction scheme} enjoyed by all the tuples showing up in our arguments.

\begin{example}\label{Ex: uncontrollable}
Consider the tuple
\begin{align}\label{E: uncontrollable 2}
    \brac{T_1^{n^2}, T_2^{n^2}, T_3^{n^2}, T_4^{n^2}, T_1^{n^2+n}, T_1^{n^2+n}, T_5^{n^2+2n}, T_5^{n^2+2n}}
\end{align}
from Example \ref{Ex: controllable}. It is a descendant of the tuple
\begin{align}\label{E: uncontrollable 2 ancestor}
    \brac{T_1^{n^2}, T_2^{n^2}, T_3^{n^2}, T_4^{n^2}, T_5^{n^2+n}, T_6^{n^2+n}, T_7^{n^2+2n}, T_8^{n^2+2n}},
\end{align}
obtained by four applications of Proposition \ref{P: type reduction}, in which we substitute $T_5$ for $T_8$ at the index 8, $T_5$ for $T_7$ at the index 7, $T_1$ for $T_5$ at the index 5 and $T_1$ for $T_6$ at the index 6. Corollary \ref{C: Flipping uncontrollable tuples} gives that if $f_1, \ldots, f_8\in L^\infty(\mu)$ are functions such that $f_5, f_6, f_7, f_8$ are invariant under $T_5 T_1\inv, T_6T_1\inv, T_7T_5\inv, T_8T_5\inv$ respectively, then we have
\begin{align*}
    \lim_{N\to\infty}\norm{\E_{n\in[N]}T_1^{n^2}f_1\cdot T_2^{n^2}f_2\cdot T_3^{n^2}f_3\cdot T_4^{n^2}f_4\cdot T_1^{n^2+n}f_5\cdot T_1^{n^2+n}f_6\cdot T_5^{n^2+2n}f_7\cdot T_5^{n^2+2n}f_8}_{L^2(\mu)}\\
    \leq \lim_{N\to\infty}\norm{\E_{n\in[N]}T_1^{n^2}f_1\cdot T_2^{n^2}f_2\cdot T_3^{n^2}f_3\cdot T_4^{n^2}f_4\cdot T_1^{n^2+n}f_5\cdot T_1^{n^2+n}f_6\cdot T_7^{n^2+2n}f_7\cdot T_8^{n^2+2n}f_8}_{L^2(\mu)}
\end{align*}
(in fact, a closer look guarantees that we get an equality and not just for the $L^2(\mu)$ limits, but for each finite average). We moreover obtain the equality of seminorms $\nnorm{f_7}_{s, T_7} = \nnorm{f_5}_{s, T_5},  \nnorm{f_8}_{s, T_8} = \nnorm{f_5}_{s, T_5},$
and the new tuple
\begin{align}\label{E: uncontrollable flipped}
    \brac{T_1^{n^2}, T_2^{n^2}, T_3^{n^2}, T_4^{n^2}, T_1^{n^2+n}, T_1^{n^2+n}, T_7^{n^2+2n}, T_8^{n^2+2n}}
\end{align}
is a descendant of  \eqref{E: uncontrollable 2 ancestor}. Importantly, the new tuple \eqref{E: uncontrollable flipped} has type $(6, 0, 2)$, which is lower than the type $(6, 2, 0)$ of \eqref{E: uncontrollable 2}. Applying Corollary \ref{C: Flipping uncontrollable tuples} to \eqref{E: uncontrollable 2}, we have thus successfully replaced it by a tuple of lower type. Lastly, the new tuple \eqref{E: uncontrollable flipped} is controllable as the indices $7, 8$ satisfy the controllability condition.
\end{example}

\subsection{Recapitulation}\label{SS:recapitulation}
We conclude this section with Proposition \ref{P: induction scheme}, which forms an inductive framework for the proof of Theorem \ref{T:Host Kra characteristic} in the next section. Combining the content of Proposition \ref{P: type reduction} and Corollary \ref{C: Flipping uncontrollable tuples}, it shows that - starting with a tuple $\brac{T_j^{p_j(n)}}_{j\in[\ell]}$ satisfying the good ergodicity property - we reach a tuple of basic type in a finite number of steps. For tuples of basic type, seminorm control will follow from arguments made in Proposition \ref{control of basic types}, and then we will use this fact and induction to go back and get seminorm control for the original tuple $\brac{T_j^{p_j(n)}}_{j\in[\ell]}$.

\begin{proposition}[Inductive framework]\label{P: induction scheme}
Let $\ell\in\N$, $p_1, \ldots, p_\ell\in\Z[n]$ be polynomials, and $(X, \CX, \mu, T_1, \ldots, T_\ell)$ be a system. Suppose that the tuple $\brac{T_j^{p_j(n)}}_{j\in[\ell]}$ has the good ergodicity property. Then there exists $r\in\N$ and a sequence of tuples
\begin{align*}
    \brac{T_{\eta_{0j}}^{\rho_{0j}(n)}}_{j\in[\ell]} \to \brac{T_{\eta_{1j}}^{\rho_{1j}(n)}}_{j\in[\ell]} \to \brac{T_{\eta_{2j}}^{\rho_{2j}(n)}}_{j\in[\ell]}\to \ldots \to \brac{T_{\eta_{rj}}^{\rho_{rj}(n)}}_{j\in[\ell]}
\end{align*}
with types $w_0, \ldots, w_r$ such that $\brac{T_{\eta_{0j}}^{\rho_{0j}(n)}}_{j\in[\ell]} := \brac{T_j^{p_j(n)}}_{j\in[\ell]}$ and for $k\in\{0, \ldots, r-1\}$, the following properties hold.
\begin{enumerate}
    \item If the tuple $\brac{T_{\eta_{kj}}^{\rho_{kj}(n)}}_{j\in[\ell]}$ is controllable, then the tuple $\brac{T_{\eta_{(k+1)j}}^{\rho_{(k+1)j}(n)}}_{j\in[\ell]}$ is chosen using Proposition \ref{P: type reduction}.
    \item If the tuple $\brac{T_{\eta_{kj}}^{\rho_{kj}(n)}}_{j\in[\ell]}$ is uncontrollable, then the tuple $\brac{T_{\eta_{(k+1)j}}^{\rho_{(k+1)j}(n)}}_{j\in[\ell]}$ is chosen using Corollary \ref{C: Flipping uncontrollable tuples}.
    \item The tuple $\brac{T_{\eta_{(k+1)j}}^{\rho_{(k+1)j}(n)}}_{j\in[\ell]}$ is a descendant of $\brac{T_j^{p_j(n)}}_{j\in[\ell]}$. In particular, it has the good ergodicity property.
    \item For the indexing tuple $\eta_{k+1}=(\eta_{(k+1)1}, \ldots, \eta_{(k+1)\ell})$, if $ j\in\FI_t$,  then either $\eta_{(k+1)j} = j$ or $\eta_{(k+1)j}\in\FI_{t'}$ for some $t'<t$.
    \item We have $w_{k+1}<w_k$, i.e. the tuple $\brac{T_{\eta_{(k+1)j}}^{\rho_{(k+1)j}(n)}}_{j\in[\ell]}$ has a lower type than the tuple $\brac{T_{\eta_{kj}}^{\rho_{kj}(n)}}_{j\in[\ell]}$.
    \item The tuple $\brac{T_{\eta_{rj}}^{\rho_{rj}(n)}}_{j\in[\ell]}$ has basic type and the restriction $\eta_r|_{\FI_1}$ is the identity sequence.
\end{enumerate}
\end{proposition}

For the rest of the paper, we call a tuple $\brac{T_{\eta_{j}}^{\rho_{j}(n)}}_{j\in[\ell]}$ a \textit{proper descendant} of the tuple $\brac{T_j^{p_j(n)}}_{j\in[\ell]}$ if it appears in one of the sequences of tuples constructed from $\brac{T_j^{p_j(n)}}_{j\in[\ell]}$ using Proposition \ref{P: induction scheme}.

\begin{proof}
Let $k\in\{0, \ldots, r-1\}$. Suppose that the tuples $\brac{T_{\eta_{0j}}^{\rho_{0j}(n)}}_{j\in[\ell]}, \ldots, \brac{T_{\eta_{kj}}^{\rho_{kj}(n)}}_{j\in[\ell]}$ are already constructed and satisfy the properties listed in the statement of the proposition. If the tuple $\brac{T_{\eta_{kj}}^{\rho_{kj}(n)}}_{j\in[\ell]}$ has basic type, we halt. Otherwise, we choose the tuple $\brac{T_{\eta_{(k+1)j}}^{\rho_{(k+1)j}(n)}}_{j\in[\ell]}$ using Proposition \ref{P: type reduction} if $\brac{T_{\eta_{kj}}^{\rho_{kj}(n)}}_{j\in[\ell]}$ is controllable\footnote{We remark that by the property (iv) applied to the tuple $\brac{T_{\eta_{kj}}^{\rho_{kj}(n)}}_{j\in[\ell]}$, we have $\eta_{kj} = j$ whenever $j\in\FI_1$, hence the type $w_k=(w_{k1}, \ldots, w_{k\ell})$ of $\brac{T_{\eta_{kj}}^{\rho_{kj}(n)}}_{j\in[\ell]}$ satisfies $w_{k1}$>0. By the assumption that $w_k$ is not basic, we also have $w_{kt} >0$ for some $t>1$, therefore we have at least two nonzero indices in $w_k$ and we can act as in Proposition \ref{P: type reduction}.} and using Corollary \ref{C: Flipping uncontrollable tuples} otherwise, so that the properties (i) and (ii) are satisfied. We note from Corollaries \ref{C: ping reduction transitive} and \ref{C: Flipping uncontrollable tuples} that the new tuple is a descendant of $\brac{T_j^{p_j(n)}}_{j\in[\ell]}$, and by Proposition \ref{P: properties of descendants} it has the good ergodicity property, hence the property (iii) holds as well. The property (iv) holds by induction and the way the tuple $\brac{T_{\eta_{(k+1)j}}^{\rho_{(k+1)j}(n)}}_{j\in[\ell]}$ is constructed using Proposition \ref{P: type reduction} or Corollary \ref{C: Flipping uncontrollable tuples}.

For the property (v), we first note that if $\brac{T_{\eta_{kj}}^{\rho_{kj}(n)}}_{j\in[\ell]}$ is controllable, then the property (v) holds for $\brac{T_{\eta_{(k+1)j}}^{\rho_{(k+1)j}(n)}}_{j\in[\ell]}$ by Proposition \ref{P: type reduction}. If $\brac{T_{\eta_{kj}}^{\rho_{kj}(n)}}_{j\in[\ell]}$ is uncontrollable, then we get from Corollary \ref{C: Flipping uncontrollable tuples} that $w_{(k+1)t_{k}} = 0<w_{kt_k}$, where $t_{k}:=t_{w_{k}}$ is the last nonzero index of $w_k$, and we deduce from property (iv) that $w_{(k+1)t} = w_{kt}$ for $1\leq t < t_k$. This point is important, so we explain it in words. What happens is that when we apply Corollary \ref{C: Flipping uncontrollable tuples}, the index $w_{(k+1)t_k}$ goes down to 0 (since all the transformations with indices from $\FI_{t_k}$ get flipped), but the new transformations appearing in their place have indices from $\FI_{t_k+1}, \ldots, \FI_{K_2}$, as given by the property (iv). Hence, $w_{k+1}<w_k$ in this case as well.

It follows from the property (v) and the fact that there are at most $(K_3+1)^{K_2}$ possible types for tuples in $\N_0^{K_2}$ with sum of coordinates  $K_3$, that the sequence eventually terminates. And since, by our construction, it can only terminate on the basic type
	$(K_3,0,\ldots, 0)$,
there exists $r<(K_3+1)^{K_2}\leq (\ell+1)^\ell$ such that the tuple $\brac{T_{\eta_{rj}}^{\rho_{rj}(n)}}_{j\in[\ell]}$ has basic type $(K_3,0,\ldots, 0)$. The second part of property (vi) for this tuple follows from the property (iv) by taking $t=1$.
\end{proof}

\begin{example}[Iterative reduction to tuples of lower type]\label{Ex: induction scheme ex}
For the tuple \eqref{E: uncontrollable 2 ancestor} from Example \ref{Ex: uncontrollable}, Proposition \ref{P: induction scheme} would give, among other options, the following sequence of tuples:
\begin{align*}
        &\brac{T_1^{n^2}, T_2^{n^2}, T_3^{n^2}, T_4^{n^2}, T_5^{n^2+n}, T_6^{n^2+n}, T_7^{n^2+2n}, T_8^{n^2+2n}}\\
        \rightarrow & \brac{T_1^{n^2}, T_2^{n^2}, T_3^{n^2}, T_4^{n^2}, T_5^{n^2+n}, T_6^{n^2+n}, T_7^{n^2+2n}, T_5^{n^2+2n}}\\
        \rightarrow & \brac{T_1^{n^2}, T_2^{n^2}, T_3^{n^2}, T_4^{n^2}, T_5^{n^2+n}, T_6^{n^2+n}, T_5^{n^2+2n}, T_5^{n^2+2n}}\\
        \rightarrow & \brac{T_1^{n^2}, T_2^{n^2}, T_3^{n^2}, T_4^{n^2}, T_1^{n^2+n}, T_6^{n^2+n}, T_5^{n^2+2n}, T_5^{n^2+2n}}\\
        \rightarrow & \brac{T_1^{n^2}, T_2^{n^2}, T_3^{n^2}, T_4^{n^2}, T_1^{n^2+n}, T_1^{n^2+n}, T_5^{n^2+2n}, T_5^{n^2+2n}}\\
        \rightarrow & \brac{T_1^{n^2}, T_2^{n^2}, T_3^{n^2}, T_4^{n^2}, T_1^{n^2+n}, T_1^{n^2+n}, T_7^{n^2+2n}, T_8^{n^2+2n}}\\
        \rightarrow & \brac{T_1^{n^2}, T_2^{n^2}, T_3^{n^2}, T_4^{n^2}, T_1^{n^2+n}, T_1^{n^2+n}, T_7^{n^2+2n}, T_2^{n^2+2n}}\\
        \rightarrow & \brac{T_1^{n^2}, T_2^{n^2}, T_3^{n^2}, T_4^{n^2}, T_1^{n^2+n}, T_1^{n^2+n}, T_2^{n^2+2n}, T_2^{n^2+2n}}.
\end{align*}
Their types, starting from the top tuple, are
\begin{align*}
    (4, 2, 2) > (4, 3, 1) > (4, 4, 0) > (5, 3, 0) > (6, 2, 0) > (6, 0, 2) > (7, 0, 1) > (8, 0, 0),
\end{align*}
which shows that at each step, the new tuple has a lower type than its predecessor. The sixth tuple, counting from the top, has been obtained via Corollary \ref{C: Flipping uncontrollable tuples} (since the fifth tuple is uncontrollable, as explained in Examples \ref{Ex: controllable} and \ref{Ex: uncontrollable}) while all the other tuples have been obtained using Proposition \ref{P: type reduction}. Each subsequent tuple is a descendant of the original tuple, and so each of them has the good ergodicity property thanks to Proposition \ref{P: properties of descendants}. Lastly, the final tuple has basic type, and moreover for its indexing tuple $\eta$, the restriction $\eta|_{\FI_1}=\eta|_{[4]}$ is an identity because no substitution has taken place at the first four indices.
\end{example}

\section{The proof of Theorem \ref{T:Host Kra characteristic}}\label{S:smoothing}
\subsection{Induction scheme}

In all statements in this section, we work in the setting of Proposition~\ref{P: induction scheme}, i.e. all the lower type tuples at which we arrive from some original average are those constructed in Proposition \ref{P: induction scheme} and have the properties listed there.

Theorem \ref{T:Host Kra characteristic} follows by induction from the result below upon setting  $d=L=0$ and letting $\eta$ be the identity tuple.
\begin{proposition}[Seminorm control restated]\label{P:Host Kra characteristic 2}
    Let $d, \ell, L\in\N$ and $q_1, \ldots, q_L\in\Z[n]$ be polynomials. Suppose that the polynomials $p_1, \ldots, p_\ell\in \Z[n]$ have the good ergodicity property for the system $(X,\CX, \mu,T_1, \ldots, T_\ell)$.  Let  $\brac{T_{\eta_{j}}^{\rho_{j}(n)}}_{j\in[\ell]}$ be a proper descendant of the tuple $\brac{T_j^{p_j(n)}}_{j\in[\ell]}$.
    Then there exists $s\in\N$ independent of  $(X,\CX, \mu,T_1, \ldots, T_\ell)$ such that for all functions $f_1, \ldots, f_\ell\in L^\infty(\mu)$ with the good invariance property along $\eta$ with respect to $p_1, \ldots, p_\ell$ and all sequences of functions $\CD_{1}, \ldots, \CD_L\in\FD_d$, we have
	\begin{align}\label{general average vanishes 2}
	    \lim_{N\to\infty}\norm{\E_{n\in [N]} \,\prod_{j\in[\ell]}T_{{\eta_j}}^{\rho_j(n)}f_j \cdot \prod_{j\in[L]}\CD_{j}(q_j(n))}_{L^2(\mu)} = 0
	\end{align}
	whenever $\nnorm{f_j}_{s, T_{\eta_j}} = 0$ for some $j\in[\ell]$.
\end{proposition}

A word of explanation is necessary for the statement of Proposition \ref{P:Host Kra characteristic 2}. We need both polynomials $p_1, \ldots, p_\ell$ and $\rho_1, \ldots, \rho_\ell$. The reason is that for our induction to work, we need the functions $f_1, \ldots, f_\ell$ to have the good invariance property with respect to the original family $p_1, \ldots, p_\ell$ rather than the descendant family $\rho_1, \ldots, \rho_\ell$. This is necessary for a number of reasons: to prove seminorm control for averages of basic types in the proof of Proposition \ref{control of basic types}; to apply Proposition \ref{P:flipping} in the \textit{pong} step of Proposition \ref{P:smoothing}; to derive Proposition \ref{P:Host Kra characteristic 2} from Proposition \ref{P:iterated smoothing} for controllable tuples; and to invoke Corollary \ref{C: Flipping uncontrollable tuples} for uncontrollable tuples in Proposition \ref{P:Host Kra characteristic 2}.
The necessity of keeping track of the invariance property with regards to the original polynomial family has also been explained in Step 2 of Example \ref{Ex: 6 tuple}.

We first prove Proposition \ref{P:Host Kra characteristic 2} for averages \eqref{general average} of basic type. This will serve as the base for induction for Propositions \ref{P:Host Kra characteristic 2}, \ref{P:iterated smoothing}, and \ref{P:smoothing}.

\begin{proposition}[Seminorm control of basic types]\label{control of basic types}
    Let $d, \ell, L\in\N$ and $q_1, \ldots, q_L\in\Z[n]$ be polynomials. Suppose that the polynomials $p_1, \ldots, p_\ell\in \Z[n]$ have the good ergodicity property for the system $(X,\CX, \mu,T_1, \ldots, T_\ell)$. Let  $\brac{T_{\eta_{j}}^{\rho_{j}(n)}}_{j\in[\ell]}$ be a proper descendant of the tuple $\brac{T_j^{p_j(n)}}_{j\in[\ell]}$, and suppose that the type $w$ of $\brac{T_{\eta_{j}}^{\rho_{j}(n)}}_{j\in[\ell]}$ is basic.
    Then there exists $s\in\N$ independent of  $(X,\CX, \mu,T_1, \ldots, T_\ell)$ such that for all 1-bounded functions $f_1, \ldots, f_\ell\in L^\infty(\mu)$ with the good invariance property along $\eta$ with respect to $p_1, \ldots, p_\ell$ and all sequences of functions $\CD_{1}, \ldots, \CD_L\in\FD_d$, we have \eqref{general average vanishes 2}	whenever $\nnorm{f_j}_{s, T_{\eta_j}} = 0$ for some $j\in[\ell]$.
\end{proposition}
A special case of Proposition \ref{control of basic types} has been sketched in Step 2 of Example \ref{Ex: 6 tuple}, and we invite the reader to compare the abstract proof presented below with the argument in Step 2 of Example \ref{Ex: 6 tuple}.
\begin{proof}
We induct on the length $\ell$ of the average. If $\ell = 1$, the statement holds by Proposition \ref{strong PET bound}. We therefore assume that $\ell > 1$, and we will prove Proposition \ref{control of basic types} for fixed $\ell>1$ by invoking Proposition \ref{P:Host Kra characteristic 2} for an average of length $\ell-1$. More specifically, we will show first that there exists an index $m$ satisfying the controllability condition, and that
we can control the average by a $T_{\eta_m}$-seminorm of $f_m$. Then we will replace $f_m$ by a dual function using Proposition \ref{dual decomposition} and the pigeonhole principle, flip the other transformations $T_{\eta_j}$ into $T_j$ using Proposition \ref{P:flipping}, and invoke Proposition \ref{P:Host Kra characteristic 2} for averages of length $\ell-1$ to obtain seminorm control in terms of other functions.

Take any $m\in\FI_1$; we assume for simplicity that $m = \ell$. Proposition \ref{P: induction scheme}(vi) implies that $\eta_\ell=\ell$, and moreover that $\ell$ satisfies the controllability condition. This fact and Proposition \ref{strong PET bound} imply that 
the identity \eqref{general average vanishes 2} holds whenever $\nnorm{f_\ell}_{\b_1, \ldots, \b_s} = 0$ for some
\begin{align}\label{allowed coefficients}
    \b_1, \ldots, \b_s \in\{b_\ell\be_\ell - b_i\be_{\eta_i}:\ i\in\FL\cup\{0\}\setminus\{\ell\}\},
\end{align}
where $b_\ell, b_i$ are the coefficients of $\rho_\ell, \rho_i$ of degree $d_{\ell i}:= \deg(\rho_\ell \be_\ell - \rho_i\be_{\eta_i})$.
Since the tuple $\brac{T_{\eta_j}^{\rho_j(n)}}_{j\in[\ell]}$ has a basic type, it follows that the indices $\eta_i$ in \eqref{allowed coefficients} come from the set $\FI_1$. We have to show that each transformation $T^{\b_1}, \ldots, T^{\b_s}$ is either a nonzero iterate of $T_\ell$ or its invariant functions are invariant under a bounded power of $T_\ell$. For $k\in[s]$, let $\b_k:= b_\ell\be_\ell - b_i\be_{\eta_i}$. If $i = 0$, then $T^{\b_k}$ is indeed a nonzero iterate of $T_\ell$. If $i\in\FI_1$, then $\eta_i = i\neq \ell$ by the property from Proposition \ref{P: induction scheme}(vi) that $\eta|_{\FI_1}$ is the identity tuple.
Then $T^{\b_k} = T_\ell^{b_\ell}T_i^{-b_i} = \brac{T_\ell^{\beta_\ell}T_i^{-\beta_i}}^{\gcd(b_\ell, b_i)}$ for coprime integers $\beta_\ell, \beta_i$, and the good ergodicity property of $\rho_1, \ldots, \rho_\ell$ along $\eta$ (Propositions \ref{P: induction scheme}(iii)) implies that $\CI(T_\ell^{\beta_\ell}T_i^{-\beta_i})\subseteq \CI(T_\ell)$.
If $i\notin\FI_1\cup\{0\}$, then we split into the cases $\eta_i \neq \ell$ and $\eta_i = \ell$. In the former case, we once again use the good ergodicity property of $\rho_1, \ldots, \rho_\ell$ along $\eta$ to conclude that $T^{\b_k} = \brac{T_\ell^{\beta_\ell}T_i^{-\beta_i}}^{\gcd(b_\ell, b_i)}$ and $\CI(T_\ell^{\beta_\ell}T_i^{-\beta_i})\subseteq \CI(T_\ell)$.
In the latter case, the pairwise independence of $\rho_i$ and $\rho_\ell$ implies that $\b_k$ is nonzero, and hence $T^{\b_k}$ is a nonzero iterate of $T_\ell$. Rearranging $\b_1, \ldots, \b_s$, it follows that
\begin{align*}
    \nnorm{f_\ell}_{\b_1, \ldots, \b_s} =  \nnorm{f_\ell}_{c_1 \be_\ell, \ldots, c_{s'} \be_\ell, \b_{s'+1}, \ldots, \b_{s}}
\end{align*}
for some $0 \leq s'\leq s$, nonzero integers $c_1, \ldots, c_{s'}$ and transformations $T^{\b_{s'+1}}, \ldots, T^{\b_s}$ with the property that for every $j\in\{s'+1, \ldots, s\}$, there exists $\b'_j\in\Z^\ell$ and nonzero $c_j\in\Z$ such that $\b_j = c_j \b'_j$ and $\CI(T^{\b'_j})\subseteq\CI(T^{\be_\ell})$. By Lemmas \ref{L:seminorm of power} and \ref{bounding seminorms}, we have
\begin{align*}
    \nnorm{f_\ell}_{\b_1, \ldots, \b_s}= \nnorm{f_\ell}_{c_1 \be_\ell, \ldots, c_{s'} \be_\ell, c_{s'+1}\b'_{s'+1}, \ldots, c_s \b'_s}\ll_{c_1, \ldots, c_s}\nnorm{f_\ell}_{s, T_\ell},
\end{align*}
implying that \eqref{general average vanishes 2} holds whenever $\nnorm{f_\ell}_{s, T_\ell} = 0$.

We now use the seminorm control at $\ell$ with Proposition \ref{dual decomposition} and the pigeonhole principle to deduce that if \eqref{general average vanishes 2} fails, then it also fails when $T_{\eta_\ell}^{\rho_\ell(n)} f_\ell$ is replaced by a sequence $\CD_{L+1}(\rho_\ell(n))$ with $\CD_{L+1}\in\FD_s$. Letting $q_{L+1} := \rho_\ell$ for simplicity, it is enough to show that
	\begin{align*}
	    \lim_{N\to\infty}\norm{\E_{n\in [N]} \,\prod_{j\in[\ell-1]}T_{{\eta_j}}^{\rho_j(n)}f_j \cdot \prod_{j\in[L+1]}\CD_{j}(q_j(n))}_{L^2(\mu)} >0
	\end{align*}
implies $\nnorm{f_j}_{s', T_{\eta_j}} > 0$ for some $s'\in\N$ and every $j\in[\ell-1]$. By Proposition \ref{P:flipping}, there exist 1-bounded functions $f_1', \ldots, f_{\ell-1}'$ and polynomials $\rho_1', \ldots, \rho_{\ell-1}'$ with the good ergodicity property for the system such that
\begin{align*}
    \lim_{N\to\infty}\norm{\E_{n\in [N]} \,\prod_{j\in[\ell-1]}T_{{j}}^{\rho'_j(n)}f'_j \cdot \prod_{j\in[L+1]}\CD_{j}(q_j(n))}_{L^2(\mu)} >0
\end{align*}
and $\nnorm{f'_j}_{s', T_j} \ll \nnorm{f_j}_{s', T_{\eta_j}}$ for $j\in[\ell-1]$ and $s'\geq 1$. Invoking inductively the case $\ell-1$ of Proposition \ref{P:Host Kra characteristic 2}, we deduce that $\nnorm{f'_j}_{s', T_j}>0$ for some $s'\in\N$, and hence $\nnorm{f_j}_{s', T_{\eta_j}}>0$. This proves the claim.
\end{proof}

We also need the following quantitative version of Proposition \ref{P:Host Kra characteristic 2}.
\begin{proposition}[Soft quantitative estimates]\label{P:Host Kra characteristic quantitative}
    Let $d,\gamma, \ell, L\in\N$ and  $q_1, \ldots, q_L\in\Z[n]$ be  polynomials. Suppose that the polynomials  $p_1, \ldots, p_\ell\in \Z[n]$ have the good ergodicity property for the system $(X,\CX, \mu,T_1, \ldots, T_\ell)$. Let  $\brac{T_{\eta_{j}}^{\rho_{j}(n)}}_{j\in[\ell]}$ be a proper descendant of the tuple $\brac{T_j^{p_j(n)}}_{j\in[\ell]}$.
    Then there exists $s\in\N$ independent of  $(X,\CX, \mu,T_1, \ldots, T_\ell)$ with the following property: for any $\veps>0$ there exists $\delta>0$, such that for all $1$-bounded functions $f_1, \ldots, f_\ell\in L^\infty(\mu)$ that are $\gamma$-invariant along $\eta$ with respect to $p_1, \ldots, p_\ell$ and all sequences of  functions $\CD_{1}, \ldots, \CD_L\in\FD_d$, we have
	\begin{align*}
	    \lim_{N\to\infty}\norm{\E_{n\in [N]} \,\prod_{j\in[\ell]}T_{{\eta_j}}^{\rho_j(n)}f_j \cdot \prod_{j\in[L]}\CD_{j}(q_j(n))}_{L^2(\mu)} < \veps
	\end{align*}
	whenever $\nnorm{f_j}_{s, T_{\eta_j}} < \delta$ for some $j\in[\ell]$.
\end{proposition}

\begin{proof}
	We prove Proposition \ref{P:Host Kra characteristic quantitative} for fixed $d, \gamma, \ell, L, \eta$ by assuming Proposition \ref{P:Host Kra characteristic 2} for the same parameters.
	
	Let $s\in\N$ be as in the statement of Proposition \ref{P:Host Kra characteristic 2} and $b_1, \ldots, b_\ell$ be the leading coefficients of $\rho_1, \ldots, \rho_\ell$. Fix $\pi\in[\ell]^L$. We first prove the following qualitative claim: for all  $f_1, \ldots, f_\ell, g_1, \ldots, g_L\in L^\infty(\mu)$, where $f_j$ is $\brac{T_{\eta_j}^{b_{\eta_j}}T_j^{-b_j}}^\gamma$-invariant  for each $j\in[\ell]$ and $g_j$ is $\CZ_d(T_{\pi_j})$-measurable for each $j\in[L]$, we have
	\begin{align}\label{general average vanishes 11}
		\lim_{N\to\infty}\norm{\E_{n\in [N]} \,\prod_{j\in[\ell]}T_{\eta_j}^{\rho_j(n)}f_j \cdot  \prod_{j\in[L]}T^{q_j(n)}_{\pi_j}g_j}_{L^2(\mu)} = 0
	\end{align}
	whenever $\nnorm{f_j}_{s, T_{\eta_j}} = 0$ for some $j\in[\ell]$.
	
	Fix $f_1, \ldots, f_\ell, g_1, \ldots, g_L$ and suppose that \eqref{general average vanishes 11} fails. Using Proposition \ref{dual decomposition} and the pigeonhole principle, we deduce that there exist dual functions $g_1', \ldots, g_L'$ of $T_{\pi_1}, \ldots, T_{\pi_L}$ of level $d$ such that
	\begin{align*}
		\lim_{N\to\infty}\norm{\E_{n\in [N]} \,\prod_{j\in[\ell]}T_{\eta_j}^{\rho_j(n)}f_j \cdot  \prod_{j\in[L]}T^{q_j(n)}_{\pi_j}g_j'}_{L^2(\mu)} > 0.
	\end{align*}
	Setting $\CD_j(n) := T_{\pi_j}^{n}g'_j$ and using Proposition \ref{P:Host Kra characteristic 2}, we deduce that $\nnorm{f_j}_{s, T_{\eta_j}}>0$ for all $j\in[\ell]$, and so the claim follows.
	
	We combine the claim above with Proposition \ref{P:Us} for
	\begin{align*}
	    \CY_j := \begin{cases} \CI\brac{\brac{T_{\eta_j}^{b_{\eta_j}}T_j^{-b_j}}^\gamma},\; &1\leq j \leq \ell\\
	    \CZ_d(T_{\pi_j}),\; &\ell+1\leq j \leq L,
	    \end{cases}
	\end{align*}
    deducing the following: for every $\varepsilon>0$ there exists $\delta_\pi>0$ such that for all 1-bounded functions $f_1, \ldots, f_\ell, g_1, \ldots, g_L\in L^\infty(\mu)$ with $f_j\in L^\infty\brac{\CI\brac{\brac{T_{\eta_j}^{b_{\eta_j}}T_j^{-b_j}}^\gamma}}$ for $j\in[\ell]$ and $g_j\in L^\infty(\CZ_d(T_{\pi_j}))$ for $j\in [L]$, we have
	\begin{align*}
		\lim_{N\to\infty}\norm{\E_{n\in [N]} \,\prod_{j\in[\ell]}T_{\eta_j}^{\rho_j(n)}f_j \cdot  \prod_{j\in[L]}T^{q_j(n)}_{\pi_j}g_j}_{L^2(\mu)} < \varepsilon
	\end{align*}
	whenever $\nnorm{f_j}_{s, T_{\eta_j}} <\delta_\pi$ for some $j\in[\ell]$.
	
	Corollary \ref{P:Host Kra characteristic quantitative} follows by taking
	$\delta:= \min(\delta_\pi\colon \, \pi\in[\ell]^L)$ and recalling that dual functions of $T_j$ of order $d$ are $\CZ_d(T_j)$-measurable, hence for every $j\in[L]$, the sequence $\CD_{j}(q_j(n))$ has the form $\CD_j(q_j(n)) = T_{\pi_j}^{q_j(n)}g_j$ for some $\pi_j\in[\ell]$ and a 1-bounded $\CZ_d(T_j)$-measurable  function $g_j\in L^\infty(\mu)$.
\end{proof}

For controllable tuples, Proposition \ref{P:Host Kra characteristic 2} will be deduced from the following result.
\begin{proposition}[Iterated box seminorm smoothing]\label{P:iterated smoothing}
    Let $d, \ell, L\in\N$ and $q_1, \ldots, q_L\in\Z[n]$ be polynomials. Suppose that the polynomials $p_1, \ldots, p_\ell\in \Z[n]$ have the good ergodicity property for the system $(X,\CX, \mu,T_1, \ldots, T_\ell)$. Let  $\brac{T_{\eta_{j}}^{\rho_{j}(n)}}_{j\in[\ell]}$ be a proper descendant of the tuple $\brac{T_j^{p_j(n)}}_{j\in[\ell]}$.
    Suppose that $\brac{T_{\eta_{j}}^{\rho_{j}(n)}}_{j\in[\ell]}$ of a non-basic type is controllable, and let $m$ be an index satisfying the controllability condition.
    Then there exist $s\in\N$ independent of the system such that for all functions $f_1, \ldots, f_\ell\in L^\infty(\mu)$ with the good invariance property along $\eta$ with respect to $p_1, \ldots, p_\ell$ and all sequences of functions $\CD_{1}, \ldots, \CD_L\in\FD_d$, we obtain \eqref{general average vanishes 2} whenever $\nnorm{f_m}_{s, T_{\eta_m}} = 0$.
\end{proposition}

Proposition \ref{P:iterated smoothing} is a consequence of Proposition \ref{strong PET bound}, followed by an iterated application of the smoothing result given below.
\begin{proposition}[Box seminorm smoothing]  \label{P:smoothing}
    Let $d, \ell, L\in\N$ and $q_1, \ldots, q_L\in\Z[n]$ be polynomials. Suppose that the polynomials $p_1, \ldots, p_\ell\in \Z[n]$ have the good ergodicity property for the system $(X,\CX, \mu,T_1, \ldots, T_\ell)$. Let  $\brac{T_{\eta_{j}}^{\rho_{j}(n)}}_{j\in[\ell]}$ be a proper descendant of the tuple $\brac{T_j^{p_j(n)}}_{j\in[\ell]}$.
    Suppose that $\brac{T_{\eta_{j}}^{\rho_{j}(n)}}_{j\in[\ell]}$ of a non-basic type is controllable, and let $m$ be an index satisfying the controllability condition.
    Then for all vectors  $\b_1, \ldots, \b_{s+1}$ satisfying \eqref{E:coefficientvectors} there exists $s'\in\N$, independent of the system, with the following property: for all  functions $f_1, \ldots, f_\ell\in L^\infty(\mu)$ with the good invariance property along $\eta$ with respect to $p_1, \ldots, p_\ell$ and all sequences of functions $\CD_{1}, \ldots, \CD_L\in\FD_d$, if $\nnorm{f_m}_{{\b_1},\ldots, {\b_{s+1}}}= 0$ implies \eqref{general average vanishes 2}, then \eqref{general average vanishes 2} also holds under the assumption that $\nnorm{f_m}_{{\b_1},\ldots, {\b_{s}}, \be_{\eta_m}^{\times s'}}= 0$.
\end{proposition}

We explain now the induction scheme whereby we prove Propositions \ref{P:Host Kra characteristic 2}-\ref{P:smoothing}. Roughly speaking, the proofs proceed by the induction on the length $\ell$ of an average and - for fixed $\ell$ - by induction on type, where the base case are averages of basic types. More precisely, the induction scheme goes as follows:
\begin{enumerate}
    \item For tuples $(T_{\eta_j}^{\rho_j(n)})_{j\in[\ell]}$ of length $\ell$ of type $w$, Proposition \ref{P:Host Kra characteristic quantitative} follows from Proposition \ref{P:Host Kra characteristic 2} as proved before.
    \item Tuples $(T_{\eta_j}^{\rho_j(n)})_{j\in[\ell]}$ of length $\ell = 1$ have necessarily basic type, and Propositions \ref{P:Host Kra characteristic 2} and \ref{control of basic types} 
     follow easily from Proposition \ref{strong PET bound}.
    \item For tuples $(T_{\eta_j}^{\rho_j(n)})_{j\in[\ell]}$ of length $\ell>1$ and basic type, Proposition \ref{P:Host Kra characteristic 2} is a  consequence of Proposition \ref{control of basic types}.
    \item For tuples $(T_{\eta_j}^{\rho_j(n)})_{j\in[\ell]}$ of length $\ell>1$ and non-basic type $w$, we prove Proposition \ref{P:smoothing} only under the assumption of controllability. This proof goes by inductively invoking Proposition \ref{P:Host Kra characteristic quantitative} in two cases: for tuples of length $\ell$ and type $w'<w$, and for tuples of length $\ell-1$. An iterative application of Proposition \ref{P:smoothing} then yields Proposition \ref{P:iterated smoothing} for tuples of length $\ell$ and type $w$.
    \item  For controllable tuples $(T_{\eta_j}^{\rho_j(n)})_{j\in[\ell]}$ of length $\ell>1$ and non-basic type $w$, we prove Proposition \ref{P:Host Kra characteristic 2} by invoking Proposition \ref{P:iterated smoothing} for tuples of length $\ell$ and type $w$ followed by an application of Proposition \ref{P:Host Kra characteristic 2} for tuples of length $\ell-1$.
    \item Lastly, for uncontrollable tuples $(T_{\eta_j}^{\rho_j(n)})_{j\in[\ell]}$ of length $\ell>1$ and non-basic type $w$, we prove Proposition \ref{P:Host Kra characteristic 2} by inductively invoking Proposition \ref{P:Host Kra characteristic 2} for tuples of length $\ell$ and type $w'<w$.
\end{enumerate}

The way in which step (vi) is carried out has been illustrated in Example \ref{Ex: uncontrollable}, in which seminorm control for the uncontrollable average \eqref{E: uncontrollable 2} is deduced from seminorm control for the lower-type controllable average \eqref{E: uncontrollable flipped}. The example below summarises how steps (iv)-(v) proceed for a controllable average.

\begin{example}[Inductive steps for a controllable average]
Consider the tuple
\begin{align}\label{1232 ind}
    (T_1^{n^2}, T_2^{n^2}, T_3^{n^2+n}, T_2^{n^2 + n})
\end{align}
of type $(3, 1)$,
which is a descendant of the tuple
\begin{align}\label{1234 ind}
    (T_1^{n^2}, T_2^{n^2}, T_3^{n^2+n}, T_4^{n^2 + n})
\end{align}
from Example \ref{Ex: n^2, n^2, n^2+n, n^2+n} of type $(2,2)$.  While proving Proposition \ref{P:smoothing} for \eqref{1232 ind}, we invoke in the \textit{ping} step Proposition \ref{P:Host Kra characteristic quantitative} for tuples
\begin{align*}
    (T_1^{n^2}, T_2^{n^2}, T_1^{n^2+n}, T_2^{n^2 + n})\quad \textrm{and}\quad     (T_1^{n^2}, T_2^{n^2}, T_2^{n^2+n}, T_2^{n^2 + n}).
\end{align*}
They are proper descendants of the original tuple \eqref{1234 ind}, have basic type $(4, 0)$, and Proposition \ref{P:Host Kra characteristic 2} follows for them from Proposition \ref{control of basic types}. In the \textit{pong} step of the proof of Proposition \ref{P:smoothing} for \eqref{1232 ind}, we inductively invoke Proposition \ref{P:Host Kra characteristic quantitative} for the following tuples of length 3, obtained by replacing the first and second term respectively by dual functions:
\begin{align*}
    (*, T_2^{n^2}, T_3^{n^2+n}, T_2^{n^2 + n})\quad \textrm{and}\quad     (T_1^{n^2}, *, T_3^{n^2+n}, T_2^{n^2 + n}).
\end{align*}
Finally, once we prove the $T_3$-seminorm control of the third term in \eqref{1232 ind} using Proposition \ref{P:iterated smoothing}, an iterated version of Proposition \ref{P:smoothing}, we derive seminorm control of other terms in \eqref{1232 ind} as follows. Replacing the third term in \eqref{1232 ind} by a dual function using the newly established $T_3$-control, Proposition \ref{dual decomposition}, and the pigeonhole principle, we get a tuple
\begin{align*}
    (T_1^{n^2}, T_2^{n^2}, *, T_2^{n^2 + n}),
\end{align*}
and then we apply Proposition \ref{P:flipping} to flip the $T_2$ in the last term into $T_4$, obtaining the tuple
\begin{align*}
    (T_1^{n^2}, T_2^{n^2}, *, T_4^{3n^2 + 3n}).
\end{align*}
This tuple has length 3, and it is good for seminorm control by Proposition \ref{P:Host Kra characteristic 2} applied inductively to tuples of length 3. Going back, this gives us seminorm control of the other terms in \eqref{1232 ind}.
\end{example}

\subsection{Proof of Proposition \ref{P:smoothing}}

We prove Proposition \ref{P:smoothing} for the tuple $\brac{T_{\eta_{j}}^{\rho_{j}(n)}}_{j\in[\ell]}$ of non-basic type $w$ with the last nonzero index $t$, which is
a proper descendant of $\brac{T_j^{p_j(n)}}_{j\in[\ell]}$,  by assuming that Proposition \ref{P:Host Kra characteristic 2} holds for tuples of length $\ell-1$ as well as  length $\ell$ and type $w'<w$. For simplicity of notation, we assume that $m = \ell$ satisfies the controllability condition.

 By Proposition \ref{strong PET bound}, the vector $\b_{s+1}$ is nonzero and takes the form $\b_{s+1} = b_\ell \be_{\eta_\ell} - b_i \be_{\eta_i}$ for some $i\in\{0, \ldots, \ell-1\}$, where $b_\ell, b_i$ are the coefficients of $p_\ell, p_i$ of degree $d_{\ell i} := \deg(p_\ell \be_{\eta_\ell} - p_i\be_{\eta_i})$. If $i = 0$, then $b_\ell\neq 0$ since $\b_{s+1}$ is nonzero. If $\eta_i = \eta_\ell$, then the controllability condition implies that $\rho_i, \rho_\ell$ are independent, and so $T^{\b_{s+1}}$ is a nonzero iterate of $T_{\eta_\ell}$. In both these cases, the result follows from the bound
\begin{align*}
     \nnorm{f_\ell}_{{\b_1}, \ldots, {\b_{s}}, c \be_{\eta_\ell}}\ll_{|c|} \nnorm{f_\ell}_{{\b_1}, \ldots, {\b_{s}}, \be_{\eta_\ell}}
 \end{align*}
for any $c\neq 0$, which is a consequence of Lemma \ref{L:seminorm of power}.

For the case $\eta_i \neq \eta_\ell$ and $\eta_i\in \FI_t$, recall first that $\eta_\ell\in\FI_t$ by assumption and the good ergodicity property of $p_1, \ldots, p_\ell$ implies that $\rho_1, \ldots, \rho_\ell$ have the good ergodicity property along $\eta$. Hence, we have $b_\ell = \beta_\ell \gcd(b_\ell, b_i)$, $b_i = \beta_i \gcd(b_\ell, b_i)$ for coprime integers $\beta_\ell, \beta_i\in\Z$ such that $\CI(T^{\beta_\ell \be_{\eta_\ell} - \beta_i \be_{\eta_i}})\subseteq \CI(T_\ell)$. By Lemmas \ref{L:seminorm of power} and \ref{bounding seminorms}, we have
\begin{align*}
    \nnorm{f_\ell}_{{\b_1}, \ldots, {\b_{s+1}}}\ll \nnorm{f_\ell}_{\b_1, \ldots, \b_s, \beta_\ell \be_{\eta_\ell} - \beta_i \be_{\eta_i}}\leq \nnorm{f_\ell}_{{\b_1}, \ldots, {\b_{s}}, \be_{\eta_\ell}}.
\end{align*}

The last remaining case to consider, and the most difficult one, is when $\eta_i\in\FI_{t'}$ with $t'\neq t$ and $b_\ell, b_i\neq 0$. The proof of Proposition \ref{P:smoothing} in this case follows the same two-step strategy that was explained in Example~\ref{Ex: n^2, n^2, n^2+n}, but we also have to take into account additional complications explained in Examples
\ref{Ex: n^2, n^2, n^2+n, n^2+n} and \ref{Ex: 6 tuple}. We first obtain the control of \eqref{general average} by $\nnorm{f_i}_{{\b_1}, \ldots, {\b_{s}}, \be_{\eta_i}^{\times s_1}}$ for some $s_1\in\N$. This is accomplished by using the control by $\nnorm{f_\ell}_{{\b_1}, \ldots, {\b_{s+1}}}$, given by assumption, for an appropriately defined function $\tilde{f_\ell}$ in place of $f_\ell$. Subsequently, we repeat the procedure by applying the newly established control by $\nnorm{f_i}_{{\b_1}, \ldots, {\b_{s}}, \be_{\eta_i}^{\times s_1}}$ for a function $\tilde{f}_i$ in place of $f_i$. This gives us the claimed result.


\smallskip
\textbf{Step 1 (ping): Obtaining control by a seminorm of $f_i$.}
\smallskip

Suppose that
\begin{align}\label{general average is positive 2}
    \lim_{N\to\infty}\norm{\E_{n\in [N]} \,\prod_{j\in[\ell]}T_{{\eta_j}}^{\rho_j(n)}f_j \cdot \prod_{j\in[L]}\CD_{j}(q_j(n))}_{L^2(\mu)}>0.
\end{align}
The good invariance property  of $f_1, \ldots, f_\ell$ implies that $f_\ell$ is invariant under $\brac{T_{\eta_\ell}^{a_{\eta_\ell}}T_\ell^{-a_\ell}}^\gamma = T^{\bc}$ for some nonzero $\gamma\in\Z$ and $\bc := \gamma(a_{\eta_\ell}\be_{\eta_\ell}-a_\ell\be_\ell)$, where $a_{\eta_\ell}, a_\ell$ are the leading coefficients of $p_{\eta_\ell}$ and $p_\ell$.
Combining this with Lemma \ref{P:dual}, we deduce that
\begin{align*}
	\lim_{N\to\infty}\norm{\E_{n\in [N]} \,\prod_{j\in[\ell-1]}T_{{\eta_j}}^{\rho_j(n)}f_j \cdot T_{\eta_\ell}^{\rho_\ell(n)}\E(\tilde{f}_\ell|\CI(T^\bc))\cdot \prod_{j\in[L]}\CD_{j}(q_j(n))}_{L^2(\mu)}>0
\end{align*}
for some function
	\begin{equation*}
	\tilde{f}_\ell:=\lim_{k\to\infty} \E_{n\in [N_k]}
	 \, T_{\eta_\ell}^{-\rho_\ell(n)}g_k\cdot \prod_{\substack{j\in [\ell-1]}}T_{\eta_\ell}^{-\rho_\ell(n)} T_{{\eta_j}}^{\rho_j(n)}\overline{f}_j\cdot \prod_{j\in[L]}T_{\eta_\ell}^{-\rho_\ell(n)}\CD_{j}(q_j(n)),
	\end{equation*}
where the limit is a weak limit. Then our assumption gives
$$
\nnorm{\E(\tilde{f}_\ell|\CI(T^\bc))}_{{\b_1}, \ldots, {\b_{s+1}}}>0.
$$
By Proposition \ref{dual-difference interchange}, we get
\begin{multline*}
\liminf_{H\to\infty}\E_{\uh,\uh'\in [H]^{s}}\\
\lim_{N\to\infty}\norm{\E_{n\in[N]} \, \prod_{j\in[\ell-1]} T_{{\eta_j}}^{\rho_j(n)} (\Delta_{{\b_1}, \ldots, {\b_s}; \uh-\uh'} f_j)\cdot T_{\eta_\ell}^{\rho_\ell(n)}u_{\uh,\uh'}\cdot \prod_{j\in[L]} \CD'_{j, \uh, \uh'}(q_j(n))}_{L^2(\mu)}>0,
\end{multline*}
where $u_{\uh,\uh'}$ are 1-bounded and invariant under both $T^{\b_{s+1}}$ and $T^\bc$, and
\begin{align*}
    \CD'_{j, \uh, \uh'}(n) := \Delta_{{\b_1}, \ldots, {\b_s}; \uh-\uh'}\CD_{j}(n)
\end{align*}
is a product of $2^s$ elements of $\FD_d$.
As a consequence of the $T^{\b_{s+1}}$-invariance of $u_{\uh,\uh'}$, we have
\begin{equation}\label{invariance prop general}
T_{\eta_\ell}^{b_\ell}u_{\uh,\uh'}=T_{\eta_i}^{b_i}u_{\uh,\uh'},
\end{equation}
where we use the identity $T_{\eta_\ell}^{b_\ell}=T_{\eta_i}^{b_i}T^{\b_{s+1}}$.  Let $\lambda\in\N$ be the smallest natural number such that $b_\ell$ divides the coefficients of $\lambda \rho_\ell$.
By the triangle inequality,
\begin{multline*}
    \liminf_{H\to\infty}\E_{\uh,\uh'\in [H]^{s}}\E_{r\in \{0, \ldots, \lambda-1\}}\\
    \lim_{N\to\infty}\norm{ \E_{n\in[N]} \,\prod_{j\in[\ell-1]} T_{{\eta_j}}^{\rho_j(\lambda n + r)} (\Delta_{{\b_1}, \ldots, {\b_s}; \uh-\uh'} f_j)\cdot T_{\eta_\ell}^{\rho_\ell(\lambda n + r)}u_{\uh,\uh'} \cdot \prod_{j\in[L]}\CD'_{j, \uh, \uh'}(q_j(\lambda n + r))}_{L^2(\mu)}>0.
\end{multline*}
Using  \eqref{invariance prop general} and the pigeonhole principle, and setting $\eta'=\tau_{\ell i}\eta$, we get that for some $r_0\in \{0, \ldots, \lambda-1\}$ we have
\begin{multline}\label{E:positive13 2}
	\limsup_{H\to\infty}\E_{\uh,\uh'\in [H]^{s}}
	\lim_{N\to\infty}\norm{\E_{n\in[N]} \, \prod_{j\in[\ell]} T_{{\eta'_j}}^{\rho'_j(n)} f_{j, \uh, \uh'}\cdot \prod_{j\in[L]}\CD'_{j, \uh, \uh'}(q'_j(n))}_{L^2(\mu)}>0,
\end{multline}
where
\begin{align*}
    \rho'_j(n) &:= \begin{cases} \rho_j(\lambda n+r_0) - \rho_j(r_0),\; &j\in[\ell-1]\\
\frac{b_i}{b_\ell}(\rho_\ell(\lambda n+r_0)-\rho_\ell(r_0)), &j = \ell
\end{cases},\\
f_{j, \uh, \uh'} &:= \begin{cases} \Delta_{{\b_1}, \ldots, {\b_s}; \uh-\uh'} T_{\eta_j}^{\rho_j(r_0)} f_j,\; &j\in[\ell-1]\\
T_{\eta_\ell}^{\rho_\ell(r_0)} u_{\uh,\uh'},\; &j=\ell
    \end{cases},
\end{align*}
and $q'_j(n):=q_j(\lambda n+r_0)$. We note that the polynomials $\rho'_1, \ldots, \rho'_\ell$ are as in Proposition \ref{P: type reduction}.

It follows from \eqref{E:positive13 2} that there exists a set $B\subset\N^{2s}$ of positive upper density and $\veps>0$ such that
\begin{align}\label{h-averages 1}
	\lim_{N\to\infty}\norm{\E_{n\in[N]} \, \prod_{j\in[\ell]} T_{{\eta'_j}}^{\rho'_j(n)} f_{j, \uh, \uh'}\cdot \prod_{j\in[L]}\CD'_{j, \uh, \uh'}(q'_j(n))}_{L^2(\mu)}\geq \veps
\end{align}
for every $(\uh,\uh')\in B$.

By assumption, $t$ is the last nonzero index of $w$, implying that $t'<t$. Hence, the new tuple $\brac{T_{\eta'_j}^{\rho'_j(n)}}_{j\in[\ell]}$ has type $w'=\sigma_{tt'}w<w$. Furthermore, by Proposition \ref{propagation of invariance}, the functions $f_{j, \uh, \uh'}$ have the good invariance property along $\eta'$ with respect to $p_1, \ldots, p_\ell$. We therefore inductively apply Proposition \ref{P:Host Kra characteristic quantitative} for tuples of length $\ell$ and type $w'$ to each average \eqref{h-averages 1}. This allows us to conclude that there exist $s_1\in\N$ (independent of the system or the functions) and $\delta>0$ such that
$$\nnorm{\Delta_{{\b_1}, \ldots, {\b_s}; \uh-\uh'} f_i}_{s_1, T_{\eta_i}} \geq \delta$$
for $(\uh, \uh')\in B$.
Hence,
\begin{align}\label{positivity of differences 3}
	\limsup_{H\to\infty}\E_{\uh,\uh'\in [H]^{s}} \nnorm{\Delta_{{\b_1}, \ldots, {\b_s}; \uh-\uh'} f_i}_{s_1, T_{\eta_i}}>0.
\end{align}
Together with Lemma \ref{difference sequences}, the inductive formula for seminorms \eqref{inductive formula}  and H\"older inequality, the inequality \eqref{positivity of differences 3}  implies that
$$
\nnorm{f_i}_{{\b_1}, \ldots, {\b_s}, \be_{\eta_i}^{\times s_1}}>0,
$$
and so the seminorm $\nnorm{f_i}_{{\b_1}, \ldots, {\b_s}, \be_{\eta_i}^{\times s_1}}$ controls the average \eqref{general average}.

\textbf{Step 2 (pong): Obtaining control by a seminorm of $f_\ell$.}
~\

To get the claim that $\nnorm{f_\ell}_{{\b_1}, \ldots, {\b_s}, \be_{\eta_\ell}^{\times s'}}$ controls the average for some $s'\in\N$, we repeat the procedure once more with $f_i$ in place of $f_\ell$. From \eqref{general average is positive 2} it follows that
\begin{align*}
\lim_{N\to\infty}\norm{\E_{n\in [N]} \,\prod_{\substack{j\in[\ell],\\ j\neq i}}T_{{\eta_j}}^{\rho_j(n)}f_j \cdot T_{\eta_i}^{\rho_i(n)}\tilde{f_i} \cdot \prod_{j\in[L]}\CD_{j}(q_j(n))}_{L^2(\mu)}>0
\end{align*}
for some function
	\begin{equation*}
	\tilde{f}_i:=\lim_{k\to\infty} \E_{n\in [N_k]}
	 \, T_{\eta_i}^{-\rho_i(n)}g_k\cdot \prod_{\substack{j\in [\ell],\\ j\neq i}}T_{\eta_i}^{-\rho_i(n)} T_{{\eta_j}}^{\rho_j(n)}\overline{f}_j\cdot \prod_{j\in[L]}T_{\eta_i}^{-\rho_i(n)}\CD_{j}(q_j(n)),
	\end{equation*}
where the limit is a weak limit. Then the previous result gives
$$
\nnorm{\tilde{f}_i}_{{\b_1}, \ldots, {\b_s}, \be_{\eta_i}^{\times s_1}}>0.
$$
By Proposition \ref{dual-difference interchange}, we get
\begin{align*}
    \liminf_{H\to\infty}\E_{\uh,\uh'\in [H]^{s}}\lim_{N\to\infty}\norm{\E_{n\in[N]} \, \prod_{\substack{j\in[\ell],\\ j\neq i}} T_{{\eta_j}}^{\rho_j(n)} (\Delta_{{\b_1}, \ldots, {\b_{s}}; \uh-\uh'} f_j)\cdot \prod_{j\in[L+1]}\CD_{j,\uh,\uh'}(q_j(n))}_{L^2(\mu)}>0
\end{align*}
where
\begin{align*}
    \CD_{j,\uh,\uh'}(n):=\begin{cases} \Delta_{{\b_1}, \ldots, {\b_s}; \uh-\uh'}\CD_{j}(n),\; &j \in[L]\\
    T_{\eta_i}^n T^{-(\b_1 h_1'+\cdots+\b_s h'_s)}\prod\limits_{\ueps\in\{0,1\}^s}\CC^{|\ueps|}\CD_{s_1, T_{\eta_i}}(\Delta_{{\b_1}, \ldots, {\b_{s}}; \uh^{\ueps}}\tilde{f}_i),\; &j=L+1.
    \end{cases}
\end{align*}
Thus, the sequence of functions $\CD_{j,\uh,\uh'}$ is a product of $2^s$ elements of $\FD_d$ if $j\in[L]$, and it is a product of $2^{s}$ elements of $\FD_{s_1}$ for $j=L+1$. Consequently, there exists $\veps>0$ and a set $B'\subset\N^{2s}$ of positive lower density such that for every $(\uh, \uh')\in B'$, we have
\begin{align*}
    \lim_{N\to\infty}\norm{\E_{n\in[N]} \, \prod_{\substack{j\in[\ell],\\ j\neq i}} T_{{\eta_j}}^{\rho_j(n)} (\Delta_{{\b_1}, \ldots, {\b_{s}}; \uh-\uh'} f_j)\cdot \prod_{j\in[L+1]}\CD_{j,\uh,\uh'}(q_j(n))}_{L^2(\mu)} > \veps.
\end{align*}
Proposition \ref{propagation of invariance} implies that the functions $(g_{1,\uh, \uh'}, \ldots, g_{\ell, \uh, \uh'})_{(\uh, \uh')\in \N^{2s}}$ given by
\begin{align*}
    g_{j, \uh, \uh'} := \begin{cases}\Delta_{{\b_1}, \ldots, {\b_{s}}; \uh-\uh'} f_j,\; &j \neq i\\
    1,\; &j = i
    \end{cases}
\end{align*}
have the good invariance property along $\eta$ with respect to $p_1, \ldots, p_\ell$. Proposition \ref{P:flipping} then gives polynomials $\rho_1', \ldots, \rho_\ell', q'_1, \ldots, q'_L\in\Z[n]$ and 1-bounded functions $g'_{1,\uh, \uh'}, \ldots, g'_{\ell,\uh, \uh'}$ with $g'_{i, \uh, \uh'} := 1$ such that
\begin{align}\label{average2step 2}
    \lim_{N\to\infty}\norm{\E_{n\in[N]} \, \prod_{\substack{j\in[\ell],\\ j\neq i}} T_j^{\rho'_j(n)} g'_{j,\uh, \uh'}\cdot \prod_{j\in[L+1]}\CD_{j,\uh,\uh'}(q'_j(n))}_{L^2(\mu)} > \veps
\end{align}
and
\begin{align}\label{seminorm bound in flipping}
\nnorm{g'_{j, \uh, \uh'}}_{s', T_j}\ll \nnorm{\Delta_{{\b_1}, \ldots, {\b_{s}}; \uh-\uh'} f_j}_{s', T_{\eta_j}}
\end{align}
for every $j\in[\ell]\setminus\{i\}$, $(\uh,\uh')\in B'$ and $s'\geq 2$. We note from Proposition \ref{P:flipping} that the absolute constant in \eqref{seminorm bound in flipping} does not depend on $(\uh, \uh')$. Observing that the averages \eqref{average2step 2} have length $\ell-1$ and $\rho'_1, \ldots, \rho'_{i-1}, \rho'_{i+1}, \ldots, \rho'_\ell$ have the good ergodicity property for the system $(X, \CX, \mu, T_1, \ldots, T_{i-1}, T_{i+1}, \ldots, T_\ell)$ (another consequence of Propositions \ref{P:flipping} and \ref{P: properties of descendants}), we conclude from \eqref{seminorm bound in flipping} and Proposition \ref{P:Host Kra characteristic quantitative} that there exists $\delta>0$ and $s'\in\N$ satisfying
\begin{align*}
    \nnorm{\Delta_{{\b_1}, \ldots, {\b_{s}}; \uh-\uh'} f_\ell}_{s', T_{\eta_\ell}}>\delta
\end{align*}
for every $(\uh, \uh')\in B'$. Consequently, we deduce that
\begin{align*}
    \liminf_{H\to\infty}\E_{\uh, \uh'\in [H]^{s}}\nnorm{\Delta_{{\b_1}, \ldots, {\b_{s}}; \uh-\uh'} f_\ell}_{s', T_{\eta_\ell}}>0.
\end{align*}
It then follows from Lemma \ref{difference sequences} and the H\"older inequality that $\nnorm{f_\ell}_{{\b_1}, \ldots, {\b_{s}}, \be_{\eta_\ell}^{\times s'}}>0$, as claimed.

\subsection{Proof of Proposition \ref{P:Host Kra characteristic 2}}
We induct on the length $\ell$ of the average, and for each fixed $\ell$ we further induct on type.
In the base case $\ell=1$, Proposition \ref{P:Host Kra characteristic 2} follows directly from Proposition \ref{strong PET bound}. We assume therefore that the average has length $\ell>1$ and type $w$ and the statement holds for averages of length $\ell-1$ as well as length $\ell$ and type $w'<w$. If the type $w$ is basic, then Proposition \ref{P:Host Kra characteristic 2} follows from Proposition \ref{control of basic types}, so we assume that $w$ is not basic. We argue differently depending on whether the average is controllable or not.

\smallskip
\textbf{Case 1: controllable averages.}
\smallskip

If the average is controllable, then Proposition \ref{P:iterated smoothing} implies that there exist $m\in[\ell]$ and $s\in\N$ such that \eqref{general average vanishes 2} holds whenever $\nnorm{f_m}_{s, T_{\eta_m}}=0$. Suppose now that
\begin{align*}
	\lim_{N\to\infty}\norm{\E_{n\in [N]} \,\prod_{j\in[\ell]}T_{{\eta_j}}^{\rho_j(n)}f_j \prod_{j\in[L]}\CD_{j}(q_j(n))}_{L^2(\mu)} > 0.
\end{align*}
Applying the fact that $\nnorm{f_m}_{s, T_{\eta_m}}$ controls this average, Proposition \ref{dual decomposition} and the pigeonhole principle, we replace  $f_m$ by a dual function of level $s$, so that we have
\begin{align*}
	\lim_{N\to\infty}\norm{\E_{n\in [N]} \,\prod_{\substack{j\in[\ell],\\ j\neq m}}T_{{\eta_j}}^{\rho_j(n)}f_j \cdot \CD(\rho_j(n)) \cdot\prod_{j\in[L]}\CD_{j}(q_j(n))}_{L^2(\mu)} > 0
\end{align*}
for some $\CD\in\FD_s$. By Proposition \ref{P:flipping}, there exist 1-bounded functions $(f_j')_{j\in[\ell]}$ with $f'_m = 1$ and polynomials $\rho_1', \ldots, \rho_{\ell}', q_1', \ldots, q_L'\in\Z[n]$ such that
\begin{align*}
	\lim_{N\to\infty}\norm{\E_{n\in [N]} \,\prod_{\substack{j\in[\ell],\\ j\neq m}}T_j^{\rho_j'(n)}f_j' \cdot \CD(\rho_j(n)) \cdot\prod_{j\in[L]}\CD_{j}(q_j'(n))}_{L^2(\mu)} > 0,
\end{align*}
 and $\nnorm{f'_j}_{s, T_j} \ll \nnorm{f_j}_{s, T_{\eta_j}}$ provided $s\geq 2$ (which we can assume without loss of generality). Moreover, the fact that $\rho'_1, \ldots, \rho'_\ell$ are descendants of $p_1, \ldots, p_\ell$ and Proposition \ref{P: properties of descendants} imply that $\brac{T_j^{\rho_j'(n)}}_{j\in[\ell], j\neq m}$ has the good ergodicity property. By the case $\ell-1$ of Proposition \ref{P:Host Kra characteristic 2}, we deduce that $\nnorm{f'_j}_{s, T_j}>0$ for $j\neq m$, and hence $\nnorm{f_j}_{s, T_{\eta_j}}>0$ for $j\neq m$.

\smallskip
\textbf{Case 2: uncontrollable averages.}
\smallskip

If the average is uncontrollable, then we apply Proposition \ref{P: induction scheme} to deduce the existence of polynomials $q_1', \ldots, q_L'\in\Z[n]$, a tuple $\brac{T_{\eta_j'}^{\rho'_j(n)}}_{j\in[\ell]}$ of type $w'<w$ that is a proper descendant of $\brac{T_j^{p_j(n)}}_{j\in[\ell]}$ as well as functions $f_1', \ldots, f_\ell'\in L^\infty(\mu)$ satisfying
\begin{align*}
    &\lim_{N\to\infty}\norm{\E_{n\in [N]} \,\prod_{j\in[\ell]}T_{{\eta_j}}^{\rho_j(n)}f_j \prod_{j\in[L]}\CD_{j}(q_j(n))}_{L^2(\mu)}\leq \\
     &\lim_{N\to\infty}\norm{\E_{n\in [N]} \,\prod_{j\in[\ell]}T_{{\eta'_j}}^{\rho'_j(n)}f'_j \prod_{j\in[L]}\CD_{j}(q'_j(n))}_{L^2(\mu)}
\end{align*}
and $\nnorm{f'_j}_{s, T_{\eta'_j}}\ll \nnorm{f_j}_{s, T_{\eta_j}}$ for every $s\geq 2$ and $j\in[\ell]$. Moreover, the functions $f'_1, \ldots, f'_\ell$ have the good invariance property along $\eta'$ with respect to $p_1, \ldots, p_\ell$. By the induction hypothesis, there exists $s\in\N$ such that the second average above vanishes whenever $\nnorm{f'_j}_{s, T_{\eta'_j}}=0$, and so the first average also vanishes whenever $\nnorm{f_j}_{s, T_{\eta_j}}=0$. This establishes the seminorm control over the tuple $\brac{T_{\eta_j}^{\rho_j(n)}}_{j\in[\ell]}$.

\section{Proofs of joint ergodicity results}\label{S: other proofs}
In this section, we derive Theorem \ref{T: weak joint ergodicity} and Corollaries \ref{C: joint ergodicity} and \ref{C:DKS conjecture holds}. We start with two observations that connect the notions of joint ergodicity and weak joint ergodicity. Their proofs are straightforward, hence we skip them.
\begin{lemma}\label{L:ergodic sequence}
Let $(X, \CX, \mu, T)$ be a system. Suppose that there exists a sequence $a:\N\to\Z$ such that $(T^{a(n)})_{n\in\N}$ is ergodic for $\mu$. Then $T$ is ergodic.
\end{lemma}
\begin{lemma}\label{L:joint vs. weak joint}
Let $a_1, \ldots, a_\ell:\N\to\Z$ be sequences and $(X, \CX, \mu, T_1, \ldots, T_\ell)$ be a system. The sequences are jointly ergodic for the system if and only if they are weakly jointly ergodic and the transformations $T_1, \ldots, T_\ell$ are ergodic.
\end{lemma}

We continue with the proof of Theorem \ref{T: weak joint ergodicity}.

\begin{proof}[Proof of Theorem \ref{T: weak joint ergodicity}]
Suppose first that the polynomials $p_1, \ldots, p_\ell$ have the good ergodicity property for the system $(X, \CX, \mu, T_1, \ldots, T_\ell)$. By Theorem \ref{T:Host Kra characteristic}, this implies that $p_1, \ldots, p_\ell$ are good for the seminorm control for the system $(X, \CX, \mu, T_1, \ldots, T_\ell)$. This result, property (ii), and Theorem \ref{T: criteria for joint ergodicity} imply that $p_1, \ldots, p_\ell$ are weakly jointly ergodic for the system.

Conversely, suppose that the polynomials are weakly jointly ergodic for the system. The condition (ii) follows by taking $f_1, \ldots, f_\ell$ to be nonergodic eigenfunctions of respective transformations. To prove condition (i), suppose that $p_i/c_i = p_j/c_j$ for some $i\neq j$ and coprime integers $c_i, c_j$, and there exists a function $f$ invariant under $T_i^{c_i}T_j^{-c_j}$ that is not simultaneously invariant under $T_i$ and $T_j$. The invariance property of $f$ gives $T_i^{c_i} f = T_j^{c_j}{f}$, and the same holds for $\overline{f}$. The coprimeness of $c_i, c_j$ implies that the polynomial $p_i/c_i = p_j/c_j$ has integer coefficients, and so we have
\begin{align*}
    \E_{n\in[N]}T_i^{p_i(n)}f \cdot T_j^{p_j(n)}\overline{f} = \E_{n\in[N]} T_i^{p_i(n)}|f|^2 = \E_{n\in[N]} T_j^{p_j(n)}|f|^2,
\end{align*}
which by the weak joint ergodicity of $p_1, \ldots, p_\ell$ converges to $\E(|f|^2|\CI(T_i)) = \E(|f|^2|\CI(T_j))$ in $L^2(\mu)$. On the other hand, the weak joint ergodicity of $p_1, \ldots, p_\ell$  implies that
\begin{align*}
    \lim_{N\to\infty}\E_{n\in[N]}T_i^{p_i(n)}f \cdot T_j^{p_j(n)}\overline{f} = \E(f|\CI(T_i))\cdot \overline{\E(f|\CI(T_j))}
\end{align*}
in $L^2(\mu)$. Hence, $\E(|f|^2|\CI(T_i)) = \E(f|\CI(T_i))\cdot \overline{\E(f|\CI(T_j))}$. The properties of the conditional expectation and the Cauchy-Schwarz inequality imply that
\begin{align*}
    \int \abs{\E(f|\CI(T_i))\cdot \overline{\E(f|\CI(T_j))}}\, d\mu &\leq \norm{\E(f|\CI(T_i))}_{L^2(\mu)} \cdot \norm{\E(f|\CI(T_j))}_{L^2(\mu)}\\
    &\leq \norm{f}_{L^2(\mu)}^2 = \int |f|^2\, d\mu = \int \E(|f|^2|\CI(T_i))\, d\mu.
\end{align*}
The two inequalities above become an equality precisely when $f = \E(f|\CI(T_i)) = \E(f|\CI(T_j))$ holds $\mu$-a.e., i.e. when $f$ is simultaneously invariant under $T_i$ and $T_j$, and so either this is the case, contradicting the assumptions on $f$, or $\E(|f|^2|\CI(T_i)) \neq \E(f|\CI(T_i))\cdot \overline{\E(f|\CI(T_j))}$, contradicting the weak joint ergodicity of $p_1, \ldots, p_\ell$.
\end{proof}

We now derive Corollary \ref{C: joint ergodicity} from Theorem \ref{T: weak joint ergodicity}.
\begin{proof}[Proof of Corollary \ref{C: joint ergodicity}]
By Lemma \ref{L:joint vs. weak joint}, the polynomials $p_1, \ldots, p_\ell$ are jointly ergodic for the system $(X, \CX, \mu, T_1, \ldots, T_\ell)$ if and only if they are weakly jointly ergodic for this system and the transformations $T_1, \ldots, T_\ell$ are ergodic. Theorem \ref{T: weak joint ergodicity} in turn implies that this is equivalent to the system having the good ergodicity property, the transformations $T_1, \ldots, T_\ell$ being ergodic, and the equation \eqref{E: good for eq nonerg} holding for all eigenfunctions. Since the transformations are ergodic, all the eigenfunctions $\chi_j$ of $T_j$ satisfy $T_j \chi_j = \lambda_j \chi_j$ for a constant $\lambda_j$, and so the condition \eqref{E: good for eq nonerg} reduces to \eqref{E: good for eq} upon taking $\lambda_j = e(\alpha_j)$ and realising that $\int \chi_j\, d\mu= 0$ unless $\alpha_j = 0$. Lastly, the good ergodicity property and the ergodicity of the transformations $T_1, \ldots, T_\ell$ jointly imply the very good ergodicity property.
\end{proof}

Finally, we prove Corollary \ref{C:DKS conjecture holds}.
\begin{proof}[Proof of Corollary \ref{C:DKS conjecture holds}]
The forward direction follows from \cite[Proposition 5.3]{DKS19}, and so it is enough to deduce the reverse direction. Our goal is to show that the conditions (i) and (ii) in the statement of Conjecture \ref{C:DKS conjecture} imply the conditions (i) and (ii) in the statement of Corollary \ref{C: joint ergodicity}. The condition (ii) in Conjecture \ref{C:DKS conjecture}, i.e. the ergodicity of $(T_1^{p_1(n)}, \ldots, T_\ell^{p_\ell(n)})_{n\in\N}$, implies the condition (ii) in Corollary \ref{C: joint ergodicity} by taking eigenfunctions. By Lemma \ref{L:ergodic sequence}, it also implies the ergodicity of $T_1, \ldots, T_\ell$ because each sequence $(T_j^{p_j(n)})_{n\in\N}$ is ergodic.

To establish the very good ergodicity property of $p_1, \ldots, p_\ell$, suppose that $p_i/c_i = p_j/c_j$ for some relatively prime $c_i, c_j\in\Z$, and $f$ is a nonconstant function invariant under $T_i^{c_i}T_j^{-c_j}$. Letting $q := p_i/c_i = p_j/c_j$ and noting that it has integer coefficients due to the coprimeness of $c_i, c_j$, we observe that
\begin{align*}
    \lim_{N\to\infty}\E_{n\in[N]}T_i^{p_i(n)}T_j^{-p_j(n)}f = \lim_{N\to\infty}\E_{n\in[N]}(T_i^{c_i}T_j^{-c_j})^{q(n)} f = f\neq \int f\, d\mu,
\end{align*}
contradicting the ergodicity of $(T_i^{p_i(n)} T_j^{-p_j(n)})_{n\in\N}$.
\end{proof}

\end{document}